\documentclass{amsart}

\usepackage{a4wide}
\usepackage{amsmath}
\usepackage{amsaddr}
\usepackage{amsfonts}
\usepackage{amssymb}
\usepackage[utf8]{inputenc}
\usepackage[english]{babel}
\usepackage{enumerate}
\usepackage{url}
\usepackage{array}
\usepackage{multirow}
\usepackage{enumitem}
\usepackage[style]{abstract}

\newcommand{\coniug}{\varphi}
\newcommand{\diag}{D}
\newcommand{\const}{\gamma}

\newtheorem{theorem}{Theorem}[section]

\newtheorem{corollary}[theorem]{Corollary}

\newtheorem{definition}[theorem]{Definition}

\newtheorem{lemma}[theorem]{Lemma}

\newtheorem{proposition}[theorem]{Proposition}

\newcommand{\abGal}[1] {\operatorname{Gal}\big(\overline{#1}/#1\big)}
\newtheorem{remark}[theorem]{Remark}
\let\oldremark\remark
\renewcommand{\remark}{\oldremark\normalfont}

\numberwithin{equation}{section}

\newcommand{\lcm}{\operatorname{lcm}}
\newcommand{\prol}{N(G)}

\begin{document}

\title{Bounds for Serre's open image theorem for elliptic curves over number fields}
\author{Davide Lombardo}
\address{Laboratoire de Mathématiques d'Orsay\footnote{Laboratoire de Mathématiques d'Orsay, Univ.~Paris-Sud, CNRS, Université Paris-Saclay, 91405 Orsay, France.}}
\email{davide.lombardo@math.u-psud.fr}

\maketitle

\begin{abstract}
\vspace{10pt}
For an elliptic curve $E/K$ without potential complex multiplication we bound the index of the image of $\abGal{K}$ in $\operatorname{GL}_2(\widehat{\mathbb{Z}})$, the representation being given by the action on the Tate modules of $E$ at the various primes. The bound is explicit and only depends on $[K:\mathbb{Q}]$ and on the stable Faltings height of $E$. We also prove a result relating the structure of closed subgroups of $\operatorname{GL}_2(\mathbb{Z}_\ell)$ to certain Lie algebras naturally attached to them.
\end{abstract}

\vspace{11pt}

{\small \noindent\textbf{Keywords}: {Galois representations, elliptic curves, Lie algebras, open image theorem}

}

\section{Introduction}\label{sec:Intro}
We are interested in studying Galois representations attached (via $\ell$-adic Tate modules) to elliptic curves $E$ defined over an arbitrary number field $K$ and without complex multiplication, i.e. such that $\operatorname{End}_{\overline{K}}(E)=\mathbb{Z}$. Let us recall briefly the setting and fix some notation: the action of $\abGal{K}$ on the torsion points of $E_{\overline{K}}$ gives rise to a family of representations (indexed by the rational primes $\ell$)
\[
\rho_\ell:\operatorname{Gal}\left(\overline{K}/K\right) \to \operatorname{GL}(T_\ell(E)),
\]
where $T_\ell(E)$ denotes the $\ell$-adic Tate module of $E$. As $T_\ell(E)$ is a free module of rank 2 over $\mathbb{Z}_\ell$ it is convenient to fix bases and regard these representations as morphisms
\[
\rho_\ell:\operatorname{Gal}\left(\overline{K}/K\right) \to \operatorname{GL}_2(\mathbb{Z}_\ell),
\]
and it is the image $G_\ell$ of these maps that we aim to study. It is also natural to encode all these representations in a single `adelic' map
\[
\rho_\infty : \operatorname{Gal}\left(\overline{K}/K\right) \to \operatorname{GL}_2(\widehat{\mathbb{Z}}),
\]
whose components are the $\rho_\ell$ and whose image we denote $G_\infty$. By a theorem of Serre (\cite[§4, Théorème 3]{MR0387283}) $G_\infty$ is open in $\operatorname{GL}_2(\widehat{\mathbb{Z}})$, and the purpose of the present study is to show that the adelic index $[\operatorname{GL}_2(\widehat{\mathbb{Z}}):G_\infty]$ is in fact bounded by an explicit function depending only on the stable Faltings height $h(E)$ of $E$ and on the degree of $K$ over $\mathbb{Q}$, generalizing and making completely explicit a result proved by Zywina \cite{2011arXiv1102.4656Z} in the special case $K=\mathbb{Q}$. More precisely we show:

\begin{theorem}\label{thm:OverK}
Let $E/K$ be an elliptic curve that does not admit complex multiplication. The inequality
\[
\left[\operatorname{GL}_2(\widehat{\mathbb{Z}}) : \rho_\infty\left(\abGal{K}\right)\right] < \displaystyle \gamma_1 \cdot [K:\mathbb{Q}]^{\gamma_2} \cdot \max\left\{1,h(E),\log [K:\mathbb{Q}]\right\}^{2\gamma_2}
\]
holds, where $\gamma_1=\exp(10^{21483})$ and $\gamma_2=2.4 \cdot 10^{10}$.
\end{theorem}

\begin{remark}
We actually prove a more precise result (theorem \ref{thm:Final}), from which the present bound follows through elementary estimates. The large constants appearing in this theorem have a very strong dependence on those of theorem \ref{thm:Isogeny}; unpublished results that Eric Gaudron and Ga\"el Rémond have been kind enough to share with the author show that the statement can be improved to
\[
\left[\operatorname{GL}_2(\widehat{\mathbb{Z}}) : \rho_\infty\left(\abGal{K}\right)\right] < \displaystyle \gamma_3 \cdot \left([K:\mathbb{Q}] \cdot \max\left\{1,h(E),\log [K:\mathbb{Q}]\right\} \right)^{\gamma_4}
\]
with the much better constants $\gamma_3=\exp\left(1.9 \cdot 10^{10} \right)$ and $\gamma_4=12395$, cf. remark \ref{rmk:ImprovedFinale}.
\end{remark}

As an easy corollary we also get:
\begin{corollary}\label{cor:FieldGeneratedByTorsion}
Let $E/K$ be an elliptic curve that does not admit complex multiplication. There exists a constant $\const(E/K)$ with the following property: for every $x \in E_{\operatorname{tors}}(\overline{K})$ (of order denoted $N(x)$) the inequality
\[
[K(x):K] \geq \const(E/K) N(x)^2
\]
holds. We can take $\const(E/K)=\left(\zeta(2) \cdot \big[ \operatorname{GL}_2\big( \widehat{\mathbb{Z}} \big) : \rho_\infty \abGal{K} \big] \right)^{-1}$, which can be explicitly bounded thanks to the main theorem.
\end{corollary}

\begin{remark} This corollary (with the same proof, but with a non-effective $\const(E/K)$) follows directly from the aforementioned theorem of Serre (\cite[§4, Théorème 3]{MR0387283}). The exponent $2$ for $N(x)$ is best possible, as is easily seen from the proof by taking $N=\ell$, a prime large enough that $G_\ell=\operatorname{GL}_2(\mathbb{Z}_\ell)$.

It should also be pointed out that for a general (possibly CM) elliptic curve Masser (\cite[p. 262]{MR1015810}) proves an inequality of the form
\[
\displaystyle [K(x):K] \geq \gamma'(K) h(E)^{-3/2} \frac{N(x)}{\log N(x)},
\]
where $\gamma'(K)$ is an effectively computable (but non-explicit) constant that only depends on $[K:\mathbb{Q}]$.
\end{remark}

\medskip

We briefly sketch the proof strategy, highlighting differences and similarities between our approach and that of \cite{2011arXiv1102.4656Z}. By a technique due to Masser and W\"ustholz (cf.~\cite{MR1209248}, \cite{MR1217345} and \cite{MR1619802}), and which is by now standard, it is possible to give a bound on the largest prime $\ell$ for which the representation modulo $\ell$ is not surjective; an argument of Serre then shows that (for $\ell \geq 5$) this implies full $\ell$-adic surjectivity. This rids us of all the primes larger than a computable bound (actually, of all those that do not divide a quantity that can be bounded explicitly in terms of $E$). We then have to deal with the case of non-surjective reduction, that is, with a finite number of `small' primes.

In \cite{2011arXiv1102.4656Z} these small primes are treated using two different techniques. All but a finite number of them are dealt with by studying a family of Lie algebras attached to $G_\ell$; this analysis is greatly simplified by the fact that the reduction modulo $\ell$ of $G_\ell$ is not contained in a Borel subgroup of $\operatorname{GL}_2(\mathbb{F}_\ell)$, a result depending on the hard theorem of Mazur on cyclic $\ell$-isogenies. The remaining primes belong to an explicit list (again given by Mazur's results), and are treated by an application of Faltings' theorem to certain modular curves. This approach, however, has two important drawbacks. On the one hand, effective results on cyclic isogenies do not seem -- at present -- to be available for arbitrary number fields, so the use of Mazur's theorem is a severe obstacle in generalizing this technique to number fields larger than $\mathbb{Q}$.
On the other hand, and perhaps more importantly, the use of Faltings' theorem is a major hindrance to effectivity, since making the result explicit for a given number field $K$ would require understanding the $K$-points of a very large number of modular curves, a task that currently seems to be far beyond our reach.

\medskip

While we do not introduce any new ideas in the treatment of the large primes, relying by and large on the methods of Masser-W\"ustholz, we do put forward a different approach for the small primes that allows us to bypass both the difficulties mentioned above. With respect to \cite{2011arXiv1102.4656Z}, the price to pay to avoid the use of Mazur's theorem is a more involved analysis of the Lie algebras associated with subgroups of $\operatorname{GL}_2(\mathbb{Z}_\ell)$, which is done here without using a congruence filtration, but dealing instead with all the orders at the same time; this approach seems to be more natural, and proves more suitable for generalization to arbitrary number fields. We also avoid the use of Faltings' theorem entirely. This too comes at a cost, namely replacing uniform bounds with functions of the Faltings height of the elliptic curve, but it has the advantage of giving a completely explicit result, which does not depend on the (potentially very complicated) arithmetic of the $K$-rational points on the modular curves.

\smallskip
The organization of the paper reflects the steps alluded to above: in section \ref{sec:IsogenyBounds} we recall an explicit form of the isogeny theorem (as proved by Gaudron and Rémond in \cite{PolarisationsEtIsogenies} building on the work of Masser and W\"ustholz) and an idea of Masser that will help improve many of the subsequent estimates by replacing an inequality with a divisibility condition. In sections 3 through 6 we prove the necessary results on the relation between Lie algebras and closed subgroups of $\operatorname{GL}_2(\mathbb{Z}_\ell)$; the main technical tool we use to show that the Galois image is large is the following theorem, which is proved in sections \ref{sec:RecoveringGOdd} (for odd $\ell$) and \ref{sec:RecoveringGEven} (for $\ell=2$):

\begin{theorem}\label{thm:ArbitrarySubgroups}
Let $\ell$ be an odd prime (resp. $\ell=2$). For every closed subgroup $G$ of $\operatorname{GL}_2(\mathbb{Z}_\ell)$ (resp.~every closed subgroup whose reduction modulo 2 is trivial if $\ell=2$) define $L(G)$ to be the $\mathbb{Z}_\ell$-span of $\left\{ g - \frac{\operatorname{tr}(g)}{2}  \cdot \operatorname{Id} \bigm\vert g \in G \right\}$. 

Let $H$ be a closed subgroup of $\operatorname{GL}_2(\mathbb{Z}_\ell)$. There is a closed subgroup $H_1$ of $H$, of index at most 24\, (resp.~with trivial reduction modulo 2 and of index at most 192 for $\ell=2$), such that the following implication holds for all positive integers $s$: if $L(H_1)$ contains $\ell^{s}\mathfrak{sl}_2(\mathbb{Z}_\ell)$, then $H_1$ itself contains
\[
\mathcal{B}_\ell(4s)=\left\{ g \in \operatorname{SL}_2(\mathbb{Z}_\ell) \bigm\vert g \equiv \operatorname{Id} \pmod{\ell^{4s}} \right\} \quad \text{(resp. }\mathcal{B}_2(6s)\text{ for }\ell=2\text{)}.
\]
\end{theorem}

The methods of these sections are then applied in section \ref{sec:GaloisGroups} to get bounds valid for \textit{every} prime $\ell$ (cf. theorem \ref{thm:GeneralIndexBound}, which might have some independent interest), while section \ref{sec:LargePrimes} deals with the large primes through the aforementioned ideas of Masser and W\"ustholz. Finally, in section \ref{sec:Finale} we put it all together to get the adelic estimate.

\section{Preliminaries on isogeny bounds}\label{sec:IsogenyBounds}
The main tool that makes all the effective estimates possible is a very explicit isogeny-type theorem taken from \cite{PolarisationsEtIsogenies}, which builds on the seminal work of Masser and W\"ustholz (cf. \cite{MR1207211} and \cite{MR1217345}). To state it we will need some notation: we let $\alpha(g)=2^{10}g^3$ and define, for any abelian variety $A/K$ of dimension $g$,
\[
b([K:\mathbb{Q}],g,h(A))=\left( (14g)^{64g^2} [K:\mathbb{Q}] \max\left(h(A), \log [K:\mathbb{Q}],1 \right)^2 \right)^{\alpha(g)}.
\]

\begin{theorem}{(\cite{PolarisationsEtIsogenies} Théorème 1.4; cf. also the section `Cas elliptique' )}\label{thm:Isogeny}
Let $K$ be a number field and $A, A^*$ be two abelian $K$-varieties of dimension $g$. If $A, A^*$ are isogenous over $K$, then there exists a $K$-isogeny $A^* \to A$ whose degree is bounded by $b([K:\mathbb{Q}],\dim(A),h(A))$.

If $E$ is an elliptic curve without complex multiplication over $\overline{K}$, then the same holds with $b([K:\mathbb{Q}],\dim(A),h(A))$ replaced by
\[
10^{13} [K:\mathbb{Q}]^2 \max\left(h(E), \log [K:\mathbb{Q}],1 \right)^2.
\]
\end{theorem}

\begin{remark} As the notation suggests, the three arguments of $b$ will always be the degree of a number field $K$, the dimension $g$ of an abelian variety $A/K$ and its stable Faltings height $h(A)$.
\end{remark}

\begin{remark}\label{rmk:IsogenyImprovedVersion1}
Unpublished results of Gaudron and Rémond show that if $A$ is the $N$-th power of an elliptic curve $E/K$ and $A^*$ is $K$-isogenous to $A$, then a $K$-isogeny $A^* \to A$ exists whose degree does not exceed $10^{13N} [K:\mathbb{Q}]^{2N} \max\left(h(E), \log [K:\mathbb{Q}],1 \right)^{2N}$.
\end{remark}

The following theorem follows easily from the arguments in Masser's paper \cite{MR1619802}; however, since it is never stated explicitly in the form we need, in the interest of completeness we include a short proof.

\begin{theorem}{(Masser)}
Suppose that $A/K$ is an abelian variety that is isomorphic over $K$ to a product $A_1^{e_1} \times \ldots \times A_n^{e_n}$, where $A_1,\ldots,A_n$ are simple over $K$, mutually non-isogenous over $K$, and have trivial endomorphism ring over $K$. Let $b \in \mathbb{R}$ be a constant with the following property: for every $K$-abelian variety $A^*$ isogenous to $A$ over $K$ there exists an isogeny $\psi:A^* \to A$ with $\deg \psi \leq b$. Then there exists an integer $b_0 \leq b$ with the following property: for every $K$-abelian variety $A^*$ isogenous to $A$ over $K$ there exists an isogeny $\psi_0:A^* \to A$ with $\deg \psi_0 \bigm\vert b_0$.
\end{theorem}
\begin{proof}
We take the notation of \cite{MR1619802}, which we briefly recall. Let $m$ be a positive integer and $G$ be a $\abGal{K}$-submodule of $A[m]$. 
For every $K$-endomorphism $\tau$ of $A$ we denote by $\ker_m \tau$ the intersection $\ker \tau \cap A[m]$; we also define \[f_m(G):=\min_{\tau} \left[ \ker_m \tau:G \right],\]
where the minimum is taken over all $\tau$ in $\operatorname{End}_K(A)$ with $G \subseteq \ker_m \tau$. By \cite[Lemma 3.3]{MR1619802} we have $f_m(G) \leq b$ for every positive integer $m$ and every Galois submodule $G$ of $A[m]$. We set $b_0:=\max_{m,G} f_m(G)$, where the maximum is taken over all positive integers $m$ and all Galois submodules $G$ of $A[m]$: clearly we have $b_0 \leq b$. 
Now if $A^*$ is a $K$-abelian variety that is $K$-isogenous to $A$ over $K$, then by \cite[Lemma 4.1]{MR1619802} there exists a $K$-isogeny $\psi:A^* \to A$ such that $\deg \psi \bigm\vert b_0$, and this establishes the theorem. Notice that in order to apply \cite[Lemma 4.1]{MR1619802} we need $i(\operatorname{End}_K(A))=1$ (in the notation of \cite{MR1619802}), which can be deduced as in \cite[p.~185, proof of Theorem 2]{MR1619802}.
\end{proof}

We will denote by $b_0(K,A)$ the minimal $b_0$ with the property of the above theorem; in particular $b_0(K,A) \leq b([K:\mathbb{Q}],h(A),\dim(A))$. Consider now $b_0(K',A)$ as $K'$ ranges over the finite extensions of $K$ of degree bounded by $d$. On one hand, $b_0(K,A)$ divides $b_0(K',A)$; on the other hand $b_0(K',A) \leq b(d[K:\mathbb{Q}],h(A),\dim(A))$ stays bounded, and therefore the number
\[
b_0(K,A;d)=\displaystyle  \underset{[K':K] \leq d}\lcm b_0(K',A)
\]
is finite. The function $b_0(K,A;d)$ is studied in \cite{MR1619802}, Theorem D, mostly through the following elementary lemma:

\begin{lemma}{(\cite[Lemma 7.1]{MR1619802})}
Let $X, Y \geq 1$ be real numbers and $\mathcal{B}$ be a family of natural numbers. Suppose that for every positive integer $t$ and every subset $A$ of $\mathcal{B}$ with $|A|=t$ we have $\lcm(A) \leq X Y^t$. The least common multiple of the elements of $\mathcal{B}$ is then finite, and does not exceed $4^{eY} X^{1+\log(C)}$, where $e=\exp(1)$.
\end{lemma}

Adapting Masser's argument to the function $b(d[K:\mathbb{Q}],h(A),\dim(A))$ at our disposal it is immediate to prove:
\begin{proposition}\label{prop_b0}
If $A$ is of dimension $g \geq 1$ and satisfies the hypotheses of the previous theorem, then
\[
b_0(K,A;d) \leq 4^{\exp(1) \cdot \left(d(1+\log d )^2\right)^{\alpha(g)} } b([K:\mathbb{Q}],\dim(A),h(A))^{1+ \alpha(g) ( \log(d) + 2 \log(1+\log d)) }.
\]

If $E$ is an elliptic curve without complex multiplication over $\overline{K}$, then the number $b_0(K,E;d)$ is bounded by
\[
4^{\exp(1) \cdot d^{2}(1+\log d)^2} \\ \left( 10^{13} [K:\mathbb{Q}]^2 \max\left(h(E), \log [K:\mathbb{Q}],1 \right)^2 \right)^{1+2\log d +2 \log(1+\log d)}.
\]
\end{proposition}

\begin{proof}
We can clearly assume $d \geq 2$. We apply the lemma to $\mathcal{B}=\left\{ b_0(K',A) \right\}_{[K':K]\leq d}$. Choose $t$ elements of $\mathcal{B}$, corresponding to extensions $K_1, \ldots, K_t$ of $K$, and set $L=K_1\cdots K_t$. We claim that
\[
\max\left\{ \log(d^t [K:\mathbb{Q}]),1 \right\} \leq (1+\log(d))^t\max\left\{1,\log [K:\mathbb{Q}] \right\}.
\]
Indeed the right hand side is clearly at least 1, so it suffices to show the inequality
\[
t\log(d)+ \log [K:\mathbb{Q}] \leq (1+\log(d))^t\max\left\{1,\log [K:\mathbb{Q}] \right\};
\]
as $\log(d)>0$, we have $(1+\log(d))^t \geq 1+t\log(d)$ by Bernoulli's inequality, and the claim follows. We thus see that $\lcm(b_0(K_1,A), \ldots, b_0(K_t,A))$ divides
\[
\begin{aligned}
b_0(L,A) & \leq b([L:\mathbb{Q}],\dim(A),h(A)) \\
																								& \leq b(d^t [K:\mathbb{Q}],\dim(A),h(A)) \\
																								& \leq \left(\left(d(1+\log d)^2\right)^{\alpha(g)}\right)^t b([K:\mathbb{Q}],\dim(A),h(A)),
\end{aligned}
\]
so we can apply the above lemma with
\[
X=b([K:\mathbb{Q}],\dim(A),h(A)), \; Y=\left(d(1+\log d)^2\right)^{\alpha(g)}
\]
to get the desired conclusion. The second statement is proved in the same way using the corresponding improved bound for elliptic curves.
\end{proof}

\begin{remark}\label{rmk:IsogenyImprovedVersion2}
We are only going to use the function $b_0(K,A;d)$ for bounded values of $d$ (in fact, $d \leq 24$), so the essential feature of the previous proposition is to show that, under this constraint, $b_0(K,A;d)$ is bounded by a polynomial in $b([K:\mathbb{Q}],\dim(A),h(A))$.

Also notice that, if $A=E^2$ is the square of an elliptic curve $E/K$, then using the improved version of theorem \ref{thm:Isogeny} mentioned in remark \ref{rmk:IsogenyImprovedVersion1} we get
\[
b_0\left(K,E^2;d\right) \leq 
4^{\exp(1) \cdot d^{4}(1+\log d)^4} \\ \left( 10^{26} [K:\mathbb{Q}]^4 \max\left(h(E), \log [K:\mathbb{Q}],1 \right)^4 \right)^{1+4\log d +4 \log(1+\log d)}.
\]

\end{remark}

We record all these facts together as a theorem for later use:
\begin{theorem}\label{cor:MassersTrick}
Suppose $A/K$ is an abelian variety, isomorphic over $K$ to a product of simple abelian varieties, each having trivial endomorphism ring over $K$. There exists a positive integer $b_0(K,A)$, not exceeding $b([K:\mathbb{Q}],\dim(A),h(A))$, with the following property: if $A^*$ is isogenous to $A$ over $K$, then there exists an isogeny $A^* \to A$, defined over $K$, whose degree \textit{divides} $b_0(K,A)$.
Furthermore, for every fixed $d$ the function
\[
b_0(K,A;d)=\lcm_{[K':K] \leq d} b_0(K',A)
\]
exists and is bounded by a polynomial in $b([K:\mathbb{Q}],\dim(A),h(A))$.
\end{theorem}

\section{Group theory for $\operatorname{GL}_2(\mathbb{Z}_\ell)$}\label{sec:GroupTheory}

Let $\ell$ be any rational prime. The subject of the following four sections is the study of certain Lie algebras associated with closed subgroups of $\operatorname{GL}_2(\mathbb{Z}_\ell)$; the construction we present is inspired from Pink's paper \cite{MR1241950}, but we will have to extend his results in various directions: in particular, our statements apply to $\operatorname{GL}_2(\mathbb{Z}_\ell)$ (and not just to $\operatorname{SL}_2(\mathbb{Z}_\ell)$), to \textit{any} $\ell$, including 2, and to arbitrary (not necessarily pro-$\ell$) subgroups. The present section contains a few necessary, although elementary, preliminaries on congruence subgroups, and introduces the relevant objects and notations.

\subsection{Congruence subgroups of $\operatorname{SL}_2(\mathbb{Z}_{\ell}$)}\label{subsec:CongruenceSubgroups}

We aim to study the structure of the congruence subgroups of $\operatorname{SL}_2(\mathbb{Z}_\ell)$, which we denote
\[
\mathcal{B}_\ell(n)=\left\{x \in \operatorname{SL}_2(\mathbb{Z}_\ell) \bigm\vert x \equiv \operatorname{Id} \pmod{\ell^n}\right\}.
\]

\noindent\textbf{Notation.} We let $v_\ell$ be the standard discrete valuation of $\mathbb{Z}_\ell$ and set $v=v_\ell(2)$ (namely $v=0$ if $\ell \neq 2$ and $v=1$ otherwise). We also let $\displaystyle \binom{\frac{1}{2}}{k}$ denote the generalized binomial coefficient $\displaystyle \binom{\frac{1}{2}}{k}= \displaystyle \frac{1}{k!} \prod_{i=0}^{k-1} \left(\frac{1}{2}-i\right)$ and define $\sqrt{1+t}$ to be the formal power series $\displaystyle \sum_{k \geq 0} \binom{\frac{1}{2}}{k} t^k$.

\smallskip

The first piece of information we need is the following description of a generating set for $\mathcal{B}_\ell(n)$:

\begin{lemma}\label{lemma:generators}
For $n \geq 1$ the group $\mathcal{B}_\ell(n)$ is generated by the elements
\[
L_a=\left( \begin{matrix} 1 & 0 \\ a & 1 \end{matrix} \right),\; R_b=\left( \begin{matrix} 1 & b \\ 0 & 1 \end{matrix} \right) \mbox{ and } \diag_c=\left( \begin{matrix} 1+c & 0 \\ 0 & \frac{1}{1+c} \end{matrix} \right)
\]
for $a,b,c$ ranging over $\ell^n \mathbb{Z}_\ell$.
\end{lemma}

\begin{proof}
Let $x=\left( \begin{matrix} x_{11} & x_{12} \\ x_{21} & x_{22} \end{matrix} \right)$ be an element of $\mathcal{B}_\ell(n)$. Since $x_{11} \equiv 1 \pmod \ell$, it is in particular a unit, so $a=-\displaystyle \frac{x_{21}}{x_{11}}$ has valuation $v_\ell(a)= v_\ell(x_{21}) \geq n$, i.e. $a \in \ell^n\mathbb{Z}_\ell$. Next we compute
\[
L_a x = \left( \begin{matrix} x_{11} & x_{12} \\ 0 & a x_{12}+x_{22} \end{matrix}\right);
\]
we are thus reduced to the case $x_{21}=0$. Under this hypothesis, and choosing $b=\displaystyle -\frac{x_{12}}{x_{11}}$, it is easily seen that $x R_b \in \mathcal{B}_\ell(n)$ is diagonal, and since every diagonal matrix in $\mathcal{B}_\ell(n)$ is by definition of the form $\diag_c$ for some $c \in \ell^n \mathbb{Z}_\ell$ we are done.\end{proof}

We will also need a description of the derived subgroup of $\mathcal{B}_\ell(n)$; in order to prove the relevant result, we first need a simple-minded lemma on valuations that will actually come in handy in many instances:

\begin{lemma}\label{lemma:valuations2}
Let $x \in \mathbb{Z}_\ell$. We have:

\begin{enumerate}[leftmargin=*]
\item
For $\ell=2$ and $v_2(x) \geq 3$ the series $\sqrt{1+x}=\displaystyle \sum_{k \geq 0} \binom{\frac{1}{2}}{k} x^k$ converges to the only solution $\lambda$ of the equation $\lambda^2=1+x$ that satisfies $\lambda \equiv 1 \pmod 4$. The inequality $v_2(\sqrt{1+x}-1) \geq v_2(x)-1$ holds.

\item
For $\ell \neq 2$ and $v_\ell(x)>0$ the series $\sqrt{1+x}=\displaystyle \sum_{k \geq 0} \binom{\frac{1}{2}}{k} x^k$ converges to the only solution $\lambda$ of the equation $\lambda^2=1+x$ that satisfies $\lambda \equiv  1 \pmod \ell$. The equality $v_\ell(\sqrt{1+x}-1) = v_\ell(x)$ holds.
\end{enumerate}
\end{lemma}

\begin{proof}
For $\ell=2$ we have
\[
v_2 \left(\displaystyle \binom{\frac{1}{2}}{k} \right) = v_2 \left( \frac{(1/2)(-1/2)...(-(2k-3)/2)}{k!} \right) = -k - v_2(k!) \geq -2k,
\]
while for any other prime
\[
v_\ell \left(\displaystyle \binom{\frac{1}{2}}{k} \right) = v_\ell \left( \prod_{i=1}^{k-1} (2i-1) \right) - v_\ell(k!) \geq - v_\ell(k!) \geq -\frac{1}{\ell-1} k.
\]

Convergence of the series is then immediate in both cases, and the identity of power series $\left(\sum_{k \geq 0} \displaystyle \binom{\frac{1}{2}}{k} t^k\right)^2=1+t$ implies that, for every $x$ such that the series converges, $\sum_{k \geq 0} \displaystyle \binom{\frac{1}{2}}{k} x^k$ is indeed a solution to the equation $\lambda^2=1+x$.

Let now $\ell=2$. Note that in the series expansion $\sqrt{1+x}-1 = \sum_{k \geq 1} \displaystyle \binom{\frac{1}{2}}{k} x^k$ all the terms, except perhaps the first one, have valuation at least
\[
(v_2(x)-2) \cdot 2 \geq v_2(x)-1;
\]
as for the first term, it is simply $\frac{x}{2}$, so it has exact valuation $v_2(x)-1$ and we are done; a similar argument works for $\ell \neq 2$, except now $v_\ell \left( \frac{x}{2} \right)=v_\ell(x)$. The congruence $\sqrt{1+x} \equiv 1 \pmod 4$ (resp. modulo $\ell$) now follows.
\end{proof}

\smallskip

\begin{lemma}\label{lemma:DerivedSubgroup}
For $n \geq 1$ the derived subgroup of $\mathcal{B}_\ell(n)$ contains $\mathcal{B}_\ell(2n+2v)$.
\end{lemma}

\begin{proof}
Take $R_b=\left( \begin{matrix}1 & b \\ 0 & 1 \end{matrix}\right)$ with $b \equiv 0 \pmod{\ell^{2n+2v}}$ and set $\beta=\ell^{n}$. By the above lemma $1+\frac{b}{\beta}$ has a square root $y$ congruent to $1$ modulo $\ell$ that automatically satisfies $y \equiv 1 \pmod {\ell^{n}}$, so 
\[
M=\left(
\begin{array}{cc}
 y & 0 \\
 0 & \frac{1}{y}
\end{array}
\right) \; \text{and} \; N= \left(\begin{matrix} 1 & \beta \\ 0 & 1 \end{matrix} \right)
\]
both belong to $\mathcal{B}_\ell(n)$. It is immediate to compute
\[
MNM^{-1}N^{-1} = \left( \begin{matrix} 1 & \beta(y^2-1) \\ 0 & 1 \end{matrix} \right) = \left( \begin{matrix} 1 & b \\ 0 & 1 \end{matrix}\right),
\]
so $R_b$ is an element of $\mathcal{B}_\ell(n)'$. Similar identities also show that, for every $a \equiv 0 \pmod {2^{2n+2v}}$, the derived subgroup $\mathcal{B}_\ell(n)'$ contains $\left( \begin{matrix} 1 & 0 \\ a & 1\end{matrix} \right)=L_a$. To finish the proof (using lemma \ref{lemma:generators}) we now just need to show that $\mathcal{B}_\ell(n)'$ contains $\diag_c$ for every $c \equiv 0 \pmod {\ell^{2n+2v}}$. This is done through an identity similar to the above, namely we set
\[
M=\left(\displaystyle
\begin{array}{cc}\displaystyle
 \sqrt{1+c} & 0 \\
 \frac{-c}{\beta  \sqrt{1+c}} & \frac{1}{\sqrt{1+c}}
\end{array}
\right) \; \text{and} \; N=\left(
\begin{array}{cc}\displaystyle
 1 & \beta  \\
 \frac{c}{\beta } & c+1
\end{array}
\right)
\]
and compute that $MNM^{-1}N^{-1}=\left(
\begin{array}{cc}
 1+c & 0 \\
 0 & \frac{1}{1+c}
\end{array}
\right)=\diag_c$. The only thing left to check is that $M$ and $N$ actually belong to $\mathcal{B}_\ell(n)$, which is easily done by observing that $\sqrt{1+c} \equiv 1 \pmod {\ell^n}$ by the series expansion and that $v_\ell \left(\displaystyle \frac{-c}{\beta \sqrt{1+c}}\right) \geq 2n+2v-n \geq n$.
\end{proof}

\smallskip

To conclude this paragraph we describe a finite set of generators for the congruence subgroups of $\operatorname{SL}_2(\mathbb{Z}_2)$:

\begin{lemma}\label{lemma:LeftRightGenerators1}
Let $a,u \in \mathbb{Z}_2$ and $L_a=\left( \begin{matrix} 1 & 0 \\ a & 1 \end{matrix} \right)$. Let $G$ be a closed subgroup of $\operatorname{SL}_2(\mathbb{Z}_2)$. If $L_a \in G$, then $G$ also contains $L_{au}= \left( \begin{matrix} 1 & 0 \\ au & 1\end{matrix}  \right)$. Similarly, if $G$ contains $R_b=\left( \begin{matrix} 1 & b \\ 0 & 1 \end{matrix} \right)$, then it also contains $R_{bu}$ for every $u \in \mathbb{Z}_2$. Finally, if $c \equiv 0 \pmod 4$ and $G$ contains $\diag_c=\left( \begin{matrix} 1+c & 0 \\ 0 & \frac{1}{1+c} \end{matrix} \right)$, then $G$ contains $\diag_{cu}$ for every $u \in \mathbb{Z}_2$.

Let $s$ be an integer no less than 2. If $a,b,c \in 4\mathbb{Z}_2$ are such that $\max\left\{v_2(a), v_2(b), v_2(c)\right\} \leq s$, and if $G$ contains $L_a, R_b$ and $\diag_c$, then $G$ contains $\mathcal{B}_2(s)$.
\end{lemma}
\begin{proof}
We show that the set $W$ consisting of the $w$ in $\mathbb{Z}_2$ such that $L_{aw}$ belongs to $G$ is a closed subgroup of $\mathbb{Z}_2$ containing $1$. Indeed, $L_{aw_1} L_{aw_2} =L_{a(w_1+w_2)}$ by an immediate direct calculation, so in particular $L_{aw}^{-1}=L_{-aw}$; furthermore $1 \in W$ by hypothesis, and if $w_n$ is a sequence of elements of $W$ converging to $w$, then $\left\{L_{aw_n}\right\} \subseteq G$ converges to $L_{aw}$, and since $G$ is closed $L_{aw}$ itself belongs to $G$, so $w \in W$. It follows that $W$ is closed and contains the integers, and since $\mathbb{Z}$ is dense in $\mathbb{Z}_2$ we get $W=\mathbb{Z}_2$ as claimed. Given that $u \mapsto R_{bu}$ is a group morphism the same proof also works for the family $R_{bu}$.
The situation with the family $\diag_{cu}$ is slightly different, in that $u \mapsto \diag_{cu}$ is not a group morphism; however, if $w \in \mathbb{Z}_2$, then we see that
\[
\left(\diag_c\right)^w = \left( \begin{matrix} (1+c)^w & 0 \\ 0 & \frac{1}{(1+c)^w} \end{matrix} \right)
\]
is well-defined and belongs to $G$ (indeed this is trivially true for $w \in \mathbb{Z}$, and then we just need argue by continuity). As $c \equiv 0 \pmod 4$ we also have the identity $(1+c)^w=\exp(w \log(1+c))$, since all the involved power series converge: more precisely, for any $\gamma$ in $4\mathbb{Z}_2$ the series $\sum_{j=1}^{\infty} (-1)^{j+1} \frac{\gamma^j}{j}$ converges and defines $\log(1+\gamma)$, and since the inequality $v_2(\gamma^j)-v_2(j)>v_2(\gamma)$ holds for every $j \geq 2$ we have $v_2(\log(1+\gamma))=v_2(\gamma) \geq 2$. Suppose now that $v_2(\gamma) \geq v_2(c)$: then $w=\frac{\log(1+\gamma)}{\log(1+c)}$ exists in $\mathbb{Z}_2$, so we can consider $(1+c)^w=\exp(w\log(1+c))=\exp(\log(1+\gamma))=1+\gamma$ and therefore for any such $\gamma$ the matrix $\diag_{\gamma}$ belongs to $G$.
The last statement is now an immediate consequence of lemma \ref{lemma:generators}.
\end{proof}

\subsection{Lie algebras attached to subgroups of $\operatorname{GL}_2(\mathbb{Z}_\ell)$}\label{subsection_LieAlgebras}
Our study of the groups $G_\ell$ will go through suitable integral Lie algebras, for which we introduce the following definition:

\begin{definition}
Let $A$ be a commutative ring. A \textbf{Lie algebra over $A$} is a finitely presented $A$-module $M$ together with a bracket $[\cdot,\cdot]:M \times M \to M$ that is $A$-bilinear, antisymmetric and satisfies the Jacobi identity.
For any $A$, the module $\mathfrak{sl}_2(A)= \left\{ M \in M_2(A) \bigm| \operatorname{tr}(M)=0 \right\}$ endowed with the usual commutator is a Lie algebra over $A$. The same is true for $\mathfrak{gl}_2(A)$, the set of all $2 \times 2$ matrices with coefficients in $A$.
\end{definition}

We restrict our attention to the case $A=\mathbb{Z}_\ell$, and try to understand closed subgroups $G$ of $\operatorname{GL}_2(\mathbb{Z}_\ell)$ by means of a surrogate of the usual Lie algebra construction. In order to do so, we introduce the following definitions, inspired by those of \cite{MR1241950}:

\begin{definition}
Let $G$ be a closed subgroup of $\operatorname{GL}_2(\mathbb{Z}_\ell)$; if $\ell=2$, suppose that the image of $G$ in $\operatorname{GL}_2(\mathbb{F}_2)$ is trivial. We set
\[
\begin{array}{cccc}
\Theta : & G & \to & \mathfrak{sl}_2 \left( \mathbb{Z}_\ell \right) \\
         & g & \mapsto & g - \frac{1}{2} \operatorname{tr}(g) \cdot \operatorname{Id}.
\end{array}
\]

Note that this definition makes sense even for $\ell=2$, since by hypothesis the $2$-adic valuation of the trace of $g$ is at least 1.
\end{definition}

\begin{definition}
The \textbf{special Lie algebra} of $G$, denoted $L(G)$ (or simply $L$ if no confusion can arise), is the closed subgroup of $\mathfrak{sl}_2(\mathbb{Z}_\ell)$ topologically generated by $\Theta(G)$. We further define $C(G)$, or simply $C$, as the closed subgroup of $\mathbb{Z}_\ell$ topologically generated by all the traces $\operatorname{tr} (xy)$ for $x,y$ in $L(G)$.
\end{definition}

\begin{remark}
\begin{enumerate}[leftmargin=*] \item $L(G)$ is indeed a Lie algebra because of the identity
\[
[\Theta(x),\Theta(y)]=\Theta(xy)-\Theta(yx).
\]

\item If $G$ is a subgroup of $H$ then $L(G)$ is contained in $L(H)$.

\item $C$ is a $\mathbb{Z}_\ell$-module: indeed it is a $\mathbb{Z}$-module, and the action of $\mathbb{Z}$ is continuous for the $\ell$-adic topology, so it extends to an action of $\mathbb{Z}_\ell$ since $C$ is closed. Therefore $C$ is an ideal of $\mathbb{Z}_\ell$.
\end{enumerate}
\end{remark}

The key importance of $L(G)$, at least for odd $\ell$, lies in the following result:
\begin{theorem}{(\cite[Theorem 3.3]{MR1241950})}\label{thm:PinkSL2}
Let $\ell$ be an odd prime and $G$ be a pro-$\ell$ subgroup of $\operatorname{SL}_2(\mathbb{Z}_\ell)$. Set $L_2=[L(G),L(G)]$ and
\[
H_2=\left\{x \in \operatorname{SL}_2(\mathbb{Z}_\ell) \bigm\vert \Theta(x) \in L_2, \operatorname{tr}(x) -2 \in C(G)\right\}.
\]
Then $H_2$ is the derived subgroup of $G$.
\end{theorem}

\smallskip

On the other hand, for $\ell=2$ the property of $\Theta$ that will be crucial for our study of $L$ is the following approximate addition formula:
\begin{lemma}{(\cite[Formula 1.3]{MR1241950})}\label{lemma:AdditionFormula}
For every $g_1,g_2 \in \operatorname{GL}_2(\mathbb{Z}_\ell)$, if $\ell \neq 2$ (resp.~for every $g_1, g_2 \in \left\{ x \in \operatorname{GL}_2(\mathbb{Z}_2) \bigm\vert \operatorname{tr}(x) \equiv 0 \pmod 2 \right\}$, for $\ell=2$), the following identity holds:
\[
2\left( \Theta(g_1g_2)-\Theta(g_1)-\Theta(g_2) \right) \\ =[\Theta(g_1),\Theta(g_2)]+\left( \operatorname{tr}(g_1)-2\right) \Theta(g_2) + \left( \operatorname{tr}(g_2)-2\right) \Theta(g_1).
\]
\end{lemma}

\smallskip

In what follows we will often want to recover partial information on $G$ from information about the reduction of $G$ modulo various powers of $\ell$. It is thus convenient to use the following notation: 

\smallskip

\noindent \textbf{Notation.} We denote $G(\ell^n)$ the image of the reduction map
$
G \to \operatorname{GL}_2(\mathbb{Z}/\ell^n\mathbb{Z}).
$
We also let $\pi$ be the projection map $G \to G(\ell)$.

\smallskip

We now record a simple fact about modules over DVRs we will need later:

\begin{lemma}\label{lemma:DVRs}
Let $A$ be a DVR, $n$ a positive integer, $M$ a subset of $A^n$ and $N=\langle M \rangle$ the submodule of $A^n$ generated by $M$. Denote by $\pi_k$ the projection $A^n \to A$ on the $k$-th component. There exist a basis $x_1, \ldots, x_m$ of $N$ consisting of elements of $M$ and scalars $\left(\sigma_{ij}\right)_{1 \leq j < i \leq m} \subseteq A$ with the following property: if we define inductively $t_1=x_1$ and $t_i=x_i - \Sigma_{j < i} \sigma_{ij} t_j$ for $i \geq 2$, then $\pi_k \left( x_i - \Sigma_{j < l } \sigma_{ij} t_j \right)=0$ for every $1 \leq k < l \leq i \leq m$. The $t_j$ are again a basis of $N$.
\end{lemma}
\begin{proof}
We proceed by induction on $n$. The case $n=1$ is easy: $M$ is just a subset of $A$, and the claim is that the ideal generated by $M$ can also be generated by a single element of $M$, which is clear.
Consider now a subset $M$ of $A^{n+1}$. Let $\nu$ be the discrete valuation of $A$; the set $\left\{\nu(\pi_1(x)) \bigm| x \in M\right\}$ consists of non-negative integers, therefore it admits a minimum $k_1$. Take $x_1$ to be any element of $M$ such that $\nu(\pi_1(x_1))=k_{1}$. For every element $m \in M$ we can form $\displaystyle f(m)=m - \frac{\pi_{1}(m)}{\pi_{1}(x_1)} x_{1}$, which is again an element of $A^{n+1}$ since by definition of $x_1$ we have $\pi_1(x_1) \bigm\vert \pi_1(m)$. It is clear enough that $\pi_{1} (f(m))=0$ for all $m \in M$. Therefore $f(M)$ is a subset of $\left\{0\right\} \oplus A^{n}$, and it is also apparent that the module generated by $x_{1}$ and $f(M)$ is again $N$. Apply the induction hypothesis to $f(M)$ (thought of as a subset of $A^n$). It yields a basis $f(x_2), \ldots, f(x_{m})$ of $f(M)$, scalars $\left(\tau_{ij}\right)_{2 \leq j < i \leq m}$, and a sequence $u_2=f(x_2), u_i=f(x_i)-\sum_{2 \leq j < i} \tau_{ij} u_j$, such that $\pi_k( f(x_i) - \Sigma_{2 \leq j < l} \tau_{ij} u_j )=0$ for $2 \leq k < l \leq i \leq m$. We also have $\pi_1( f(x_i) - \Sigma_{2 \leq j < l} \tau_{ij} u_j )=0$ if we think the $u_i$ as elements of $A^{n+1}$. It is now enough to show that, with this choice of the $x_i$, it is possible to find scalars $\sigma_{ij}, 1 \leq j < i \leq m$, in such a way that $t_i=u_i$ for $i \geq 2$, and this we prove again by induction. By definition $\displaystyle u_2=f(x_2)=x_2-\frac{\pi_{1}(x_2)}{\pi_{1}(x_1)} x_{1}$, so we can take $\displaystyle \sigma_{21}=\frac{\pi_{1}(x_2)}{\pi_{1}(x_1)}$. Assuming we have proved the result up to level $i$, then, we have
\[
u_{i+1} = f(x_{i+1})- \sum_{2 \leq j < i+1} \tau_{ij} u_j = x_{i+1}- \frac{\pi_{1}(x_{i+1})}{\pi_{1}(x_1)}x_1 - \sum_{2 \leq j < i+1} \tau_{ij} t_j,
\]
and we simply need to take $\displaystyle \sigma_{i+1,1}=\frac{\pi_{1}(x_{i+1})}{\pi_{1}(x_1)}$ and $\sigma_{ij}=\tau_{ij}$.

As for the last statement, observe that the matrix giving the transformation from the $x_i$ to the $t_j$ is unitriangular, hence invertible.
\end{proof}

\subsection{Subgroups of $\operatorname{GL}_2(\mathbb{Z}_\ell), \operatorname{SL}_2(\mathbb{Z}_\ell)$, and their reduction modulo $\ell$}
In view of the next sections it is convenient to recall some well-known facts about the subgroups of $\operatorname{GL}_2(\mathbb{F}_\ell)$, starting with the following definition:

\begin{definition}
A subgroup $J$ of $\operatorname{GL}_2(\mathbb{F}_\ell)$ is said to be:
\begin{itemize}[leftmargin=*]
\item \textbf{split Cartan}, if $J$ is conjugated to the subgroup of diagonal matrices. In this case the order of $J$ is prime to $\ell$.
\item \textbf{nonsplit Cartan}, if there exists a subalgebra $A$ of $\operatorname{M}_2(\mathbb{F}_\ell)$ that is a field and such that $J=A^\times$. The order of $J$ is prime to $\ell$, and $J$ is conjugated to $\left\{ \left( \begin{matrix} a & b \varepsilon \\ b & a \end{matrix} \right) \in \operatorname{GL}_2(\mathbb{F}_\ell) \right\}$, where $\varepsilon$ is a fixed quadratic nonresidue.
\item \textbf{the normalizer of a split (resp. nonsplit) Cartan}, if there exists a split (resp. nonsplit) Cartan subgroup $\mathcal{C}$ such that $J$ is the normalizer of $\mathcal{C}$. The index $[J:\mathcal{C}]$ is 2, and $\ell$ does not divide the order of $J$ (unless $\ell=2$).
\item \textbf{Borel}, if $J$ is conjugated to the subgroup of upper-triangular matrices. In this case $J$ has a unique $\ell$-Sylow, consisting of the matrices of the form $\left( \begin{matrix} 1 & \ast \\ 0 & 1 \end{matrix} \right)$.
\item \textbf{exceptional}, if the projective image $\mathbb{P}J$ of $J$ in $\operatorname{PGL}_2(\mathbb{F}_\ell)$ is isomorphic to either $A_4, S_4$ or $A_5$, in which case the order of $\mathbb{P}J$ is either 12, 24 or 60.
\end{itemize}
\end{definition}

The above classes essentially exhaust all the subgroups of $\operatorname{GL}_2(\mathbb{F}_\ell)$. More precisely we have:

\begin{theorem}{(Dickson's classification, cf. \cite{MR0387283})}\label{thm:Dickson}
Let $\ell$ be a prime number and $J$ be a subgroup of $\operatorname{GL}_2(\mathbb{F}_\ell)$. Then we have:
\begin{itemize}[leftmargin=*]
\item if $\ell$ divides the order of $J$, then either $J$ contains $\operatorname{SL}_2(\mathbb{F}_\ell)$ or it is contained in a Borel subgroup;
\item if $\ell$ does not divide the order of $J$, then $J$ is contained in a (split or nonsplit) Cartan subgroup, in the normalizer of one, or in an exceptional group.
\end{itemize}

As subgroups of $\operatorname{SL}_2(\mathbb{F}_\ell)$ are in particular subgroups of $\operatorname{GL}_2(\mathbb{F}_\ell)$, the above classification also covers all subgroups of $\operatorname{SL}_2(\mathbb{F}_\ell)$. Cartan subgroups of $\operatorname{SL}_2(\mathbb{F}_\ell)$ are cyclic (both in the split and nonsplit case).
\end{theorem}

The next lemma can be proved by direct inspection of the group structure of $A_4, S_4$ and $A_5$, and will help us quantify how far exceptional subgroups are from being abelian:

\begin{lemma}\label{lemma:AbSubgroups}
The groups $A_4$ and $S_4$ have abelian subgroups of order $N$ if and only if $1 \leq N \leq 4$. The group $A_5$ has abelian subgroups of order $N$ if and only if $1 \leq N \leq 5$.
\end{lemma}

The following lemma, due to Serre, will prove extremely useful in showing that $G_\ell=\operatorname{GL}_2(\mathbb{Z}_\ell)$ using only information about the reduction of $G_\ell$ modulo $\ell$:
\begin{lemma}\label{lemma_SerreLift}
Let $\ell \geq 5$ be a prime and $G$ be a closed subgroup of $\operatorname{SL}_2(\mathbb{Z}_\ell)$. Suppose that the image of $G$ in $\operatorname{SL}_2(\mathbb{F}_\ell)$ is equal to $\operatorname{SL}_2(\mathbb{F}_\ell)$: then $G=\operatorname{SL}_2(\mathbb{Z}_\ell)$. Similarly, if $H$ is a closed subgroup of $\operatorname{GL}_2(\mathbb{Z}_\ell)$ whose image in $\operatorname{GL}_2(\mathbb{F}_\ell)$ contains $\operatorname{SL}_2(\mathbb{F}_\ell)$, then $H'=\operatorname{SL}_2(\mathbb{Z}_\ell)$.
\end{lemma}
\begin{proof}
The first statement is \cite[IV-23, Lemma 3]{SerreAbelianRepr}. For the second, consider the closed subgroup $H'$ of $\operatorname{SL}_2(\mathbb{Z}_\ell)$. Since by assumption we have $\ell > 3$, the finite group $\operatorname{SL}_2(\mathbb{F}_\ell)$ is perfect, so the image of $H'$ in $\operatorname{SL}_2(\mathbb{F}_\ell)$ contains $\operatorname{SL}_2(\mathbb{F}_\ell)' = \operatorname{SL}_2(\mathbb{F}_\ell)$. It then follows from the first part of the lemma that $H'=\operatorname{SL}_2(\mathbb{Z}_\ell)$ as claimed.
\end{proof}

The following definition will prove useful to translate statements about subgroups of $\operatorname{SL}_2(\mathbb{Z}_\ell)$ into analogous results for subgroups of $\operatorname{GL}_2(\mathbb{Z}_\ell)$ and vice versa:

\begin{definition}
Let $G$ be a closed subgroup of $\operatorname{GL}_2(\mathbb{Z}_\ell)$ (resp. $\operatorname{GL}_2(\mathbb{F}_\ell)$). The \textbf{saturation} of $G$, denoted $\operatorname{Sat}(G)$, is the group generated in $\operatorname{GL}_2(\mathbb{Z}_\ell)$ (resp. $\operatorname{GL}_2(\mathbb{F}_\ell)$) by $G$ and $\mathbb{Z}_\ell^\times \cdot \operatorname{Id}$ (resp. $\mathbb{F}_\ell^\times \cdot \operatorname{Id}$). The group $G$ is said to be \textbf{saturated} if $G=\operatorname{Sat}(G)$.
We also denote by $G^{\det=1}$ the group $G \cap \operatorname{SL}_2(\mathbb{Z}_\ell)$ (resp. $G \cap \operatorname{SL}_2(\mathbb{F}_\ell)$).
\end{definition}

\begin{lemma}\label{lemma:PropertiesOfSaturation}
The following hold:
\begin{enumerate}[leftmargin=*]
\item For every closed subgroup $G$ of $\operatorname{GL}_2(\mathbb{Z}_\ell)$ the groups $G$ and $\operatorname{Sat}(G)$ have the same derived subgroup and the same special Lie algebra.
\item The two associations $G \mapsto G^{\det=1}$ and $H \mapsto \operatorname{Sat}(H)$ are mutually inverse bijections between the sets
\[
\mathcal{G} = \left\{ G \text{ subgroup of } \operatorname{GL}_2(\mathbb{Z}_\ell) \left| \begin{array}{c} G \mbox{ is saturated, } \\  \det(g) \mbox{ is a square for every } g \text{ in } G \end{array} \right. \right\}
\]
and
\[
\mathcal{H}=\left\{ H  \text{ subgroup of } \operatorname{SL}_2(\mathbb{Z}_\ell) \bigm\vert -\operatorname{Id} \in H \right\}.
\]
For every $G$ in $\mathcal{G}$ the groups $G$ and $G^{\det=1}$ have the same derived subgroup and the same special Lie algebra.
\item The map $G \mapsto \operatorname{Sat}(G)$ commutes with reducing modulo $\ell$, i.e.
\[
\left(\operatorname{Sat}(G)\right)(\ell)=\operatorname{Sat}(G(\ell)).
\]
If $\ell$ is odd and $G$ is saturated we also have $G(\ell)^{\det=1}=G^{\det=1}(\ell)$.
\end{enumerate}
\end{lemma}

\begin{proof}
\begin{enumerate}[leftmargin=*]
\item 
The statement is obvious for the derived subgroup. As for the special Lie algebra, let $\lambda g$ be any element of $\operatorname{Sat}(G)$, where $\lambda \in \mathbb{Z}_\ell^\times$ and $g \in G$.  As $L(G)$ is a $\mathbb{Z}_\ell$-module, $\Theta(\lambda g)=\lambda \Theta(g)$ belongs to $L(G)$, hence $L(\operatorname{Sat}(G))\subseteq L(G)$. The other inclusion is trivial.
\item
The first statement is immediate to check since the determinant of any homothety is a square; the other follows by writing $G=\operatorname{Sat}(H)$ and applying (1) to $(\operatorname{Sat}(H))^{\det=1}=H$ and $\operatorname{Sat}(H)$.
\item This is clear for the saturation. For $G \mapsto G^{\det=1}$ note that $G(\ell)^{\det=1}$ contains $G^{\det=1}(\ell)$, so we need to show the opposite inclusion. Take any matrix $[g]$ in $G(\ell)^{\det=1}$. By definition $[g]$ is the reduction of a certain $g \in G$ whose determinant is 1 modulo $\ell$. As $\ell$ is odd and $\det(g)$ is congruent to $1$ modulo $\ell$ we can apply lemma \ref{lemma:valuations2} and write $\det(g)=\lambda^2$, where $\lambda=\sqrt{1+(\det(g)-1)}$ is congruent to 1 modulo $\ell$. As $G$ is saturated, it contains $\lambda^{-1} \operatorname{Id}$, hence also $\lambda^{-1} g$, whose determinant is $1$ by construction. Furthermore, as $\lambda \equiv 1 \pmod \ell$, the two matrices $\lambda^{-1}g$ and $g$ are congruent modulo $\ell$. 
We have thus found an element of $G$ of determinant 1 that maps to $[g]$, so $G^{\det=1} \to G(\ell)^{\det=1}$ is surjective.
\end{enumerate}
\end{proof}

Finally, since we will be mainly concerned with the pro-$\ell$ part of our groups, we will find it useful to give this object a name:

\smallskip

\noindent\textbf{Notation.}
If $G$ is a closed subgroup of $\operatorname{SL}_2(\mathbb{Z}_\ell)$ we write $\prol$ for its maximal normal subgroup that is a pro-$\ell$ group.

\smallskip

The following lemma shows that $N(G)$ is well-defined and gives a description of it:

\begin{lemma}\label{lemma:StructureOfNG} Let $G$ be a closed subgroup of $\operatorname{SL}_2(\mathbb{Z}_\ell)$ and $\pi:G \to G(\ell)$ the projection modulo $\ell$: then $G$ admits a unique maximal normal pro-$\ell$ subgroup $\prol$, which can be described as follows.
\begin{enumerate}[leftmargin=*]

\item If $G(\ell)$ is of order prime to $\ell$, then $\prol=\ker \pi$ and $\displaystyle G(\ell) \cong \frac{G}{\prol}$.

\item If the order of $G(\ell)$ is divisible by $\ell$, and furthermore $G(\ell)$ is contained in a Borel subgroup, then $\prol$ is the inverse image in $G$ of the unique $\ell$-Sylow $S$ of $G(\ell)$.

\item If $G(\ell)$ is all of $\operatorname{SL}_2(\mathbb{F}_\ell)$, then $N(G)=\ker \pi$ and $\displaystyle G(\ell) \cong \frac{G}{\prol}$.
\end{enumerate}
\end{lemma}

\begin{proof}
Let $N$ be a pro-$\ell$ normal subgroup of $G$. The image $\pi(N)$ is a normal pro-$\ell$ subgroup of $G(\ell)$, hence it is trivial in cases (1) and (3) and it is either trivial or the unique $\ell$-Sylow of $G(\ell)$ in case (2). In cases (1) and (3) it follows that $N \subseteq \ker \pi$, and since $\ker \pi$ is pro-$\ell$ we see that $\ker \pi$ is the unique maximal normal pro-$\ell$ subgroup of $G$. In case (2), let $S$ be the unique $\ell$-Sylow of $G(\ell)$. It is clear that $N$ is contained in $\pi^{-1}(S)$, which on the other hand is pro-$\ell$ and normal in $G$. Indeed, by choosing an appropriate (triangular) basis for $G(\ell)$ we can define
\[
\begin{array}{ccccc}
G & \to & G(\ell) & \to & \mathbb{F}_\ell^\times \\
 g & \mapsto  & \left( \begin{matrix} a & b \\ 0 & 1/a \end{matrix}\right) & \mapsto & a,
\end{array}
\]
whose kernel is exactly $\pi^{-1}(S)$.
\end{proof}

\section{Recovering $G$ from $L(G)$, when $\ell$ is odd}\label{sec:RecoveringGOdd}
Our purpose in this section (for $\ell \neq 2$) and the next (for $\ell = 2$) is to prove results that yield information on $G$ from analogous information on $L(G)$. The statements we are aiming for are the following:

\begin{theorem}\label{thm:GeneralReduction}
Let $\ell$ be an odd prime and $G$ a closed subgroup of $\operatorname{SL}_2(\mathbb{Z}_\ell)$.

\begin{enumerate}[leftmargin=*,label=(\roman*)]
\item Suppose that $G(\ell)$ is contained in a Cartan or Borel subgroup, and that $|G/\prol| \neq 4$. Then the following implication holds for all positive integers $s$:

\begin{center}
$(\star)$ \hspace{6pt} if $L(G)$ contains $\ell^s \mathfrak{sl}_2(\mathbb{Z}_\ell)$, then $L(\prol)$ contains $\ell^{2s} \mathfrak{sl}_2(\mathbb{Z}_\ell)$.
\end{center}
\item Without any assumption on $G$, there is a closed subgroup $H$ of $G$ that satisfies $[G:H] \leq 12$ and the conditions in (i) (so $H$ has property $(\star)$).
\end{enumerate}
\end{theorem}

\begin{theorem}\label{thm:Reconstruction2}
Let $\ell$ be an odd prime, and $G$ a closed subgroup of $\operatorname{GL}_2(\mathbb{Z}_\ell)$.

\begin{enumerate}[leftmargin=*,label=(\roman*)]
\item Suppose that $G$ satisfies the two conditions: 

\begin{enumerate}
\item $\det(g)$ is a square in $\mathbb{Z}_\ell^\times$ for every $g \in G$;

\item $\operatorname{Sat}(G)^{\det=1}$ satisfies the hypotheses of theorem \ref{thm:GeneralReduction} (i). 
\end{enumerate}

Then the following implication holds for all positive integers $s$:

\begin{center}
$(\star \star)$ \hspace{6pt} if $L(G)$ contains $\ell^s \mathfrak{sl}_2(\mathbb{Z}_\ell)$, then $G'$ contains $\mathcal{B}_\ell(4s)$.
\end{center}

\item Without any assumption on $G$, either $G'=\operatorname{SL}_2(\mathbb{Z}_\ell)$ or there is a closed subgroup $H$ of $G$ that satisfies both $[G:H] \leq 24$ and the conditions in (i) (so $H$ has property $(\star \star)$).
\end{enumerate}
\end{theorem}

\begin{remark}\label{rem_ReadableConditions}
Let us make condition (b) in this theorem a little more explicit. By the description of the maximal normal pro-$\ell$ subgroup given in lemma \ref{lemma:StructureOfNG}, the conditions on $G$ can be read off $(\operatorname{Sat}(G))^{\det=1}(\ell)$ as follows: $(\operatorname{Sat}(G))^{\det=1}(\ell)$ should be either a cyclic group or a group of order divisible by $\ell$ that is contained in a Borel subgroup of $\operatorname{GL}_2(\mathbb{F}_\ell)$; in the first case we ask that the order of $(\operatorname{Sat}(G))^{\det=1}(\ell)$ be different from 4, while in the second the condition reads $\left|\operatorname{Sat}(G)^{\det=1}(\ell)/S\right| \neq 4$, where $S$ is the unique $\ell$-Sylow of $\operatorname{Sat}(G)^{\det=1}(\ell)$. With this description, it is clear that condition (b) is true if $\operatorname{Sat}(G)^{\det=1}(\ell)$ contained in a Borel or Cartan subgroup and its order is not divisible by 4.
\end{remark}

\smallskip

Let us remark that the statements numbered (ii) in the above theorems require a case by case analysis, which will be carried out in section \ref{sec:BoundedIndex} for theorem \ref{thm:Reconstruction2} (the proof of theorem \ref{thm:GeneralReduction} (ii) is perfectly analogous). In the same section we will also show that part (i) of theorem \ref{thm:Reconstruction2} can be reduced to the corresponding statement in theorem \ref{thm:GeneralReduction}, so the core of the problem lies in proving the result for $\operatorname{SL}_2(\mathbb{Z}_\ell)$. Before delving into the details of the proof (that involves a certain amount of calculations) we describe the general idea, which is on the contrary quite simple. The following paragraph should only be considered as outlining the main ideas, without any pretense of formality.

If $G$ is as in theorem \ref{thm:GeneralReduction} (i), then $G/\prol$ is cyclic, and we can fix a generator $[g] \in G/\prol$ that lifts to a certain $g \in G$. Denote by $\coniug$ the operator $x \mapsto g^{-1}xg$: then $\coniug$ acts on $G$ and, since it fixes $\operatorname{Id}$, also on $L(G)$. Furthermore it preserves $L(\prol) \subseteq L(G)$ by normality of $\prol$ in $G$, and obviously it fixes $\Theta(g)$. If we were working over $\mathbb{Q}_\ell$ instead of $\mathbb{Z}_\ell$ we would have a decomposition $L(G) \cong \langle \Theta(g) \rangle \oplus M$, where $M$ is a $\coniug$-stable subspace of dimension 2, and the projection operator $p:L(G) \to M$ could be expressed as a polynomial in $\coniug$. We would also expect $M$ to consist of elements coming from $\prol$, because $\langle \Theta(g) \rangle$ is simply the special Lie algebra of $\langle g \rangle$; this would provide us with many nontrivial elements in $L(N(G))$. We would finally deduce the equality $L(\prol)=\mathfrak{sl}_2(\mathbb{Q}_\ell)$ by exploiting the fact that $L(\prol)$ is a Lie algebra of dimension at least 2 that is also stable under $\coniug$. This point of view also suggests that we cannot expect the theorem to hold when $G(\ell)$ is exceptional: if $G/\prol$ is a simple group, then we expect the special Lie algebra of $G$ not to be solvable, and since the only non-solvable subalgebra of $\mathfrak{sl}_2(\mathbb{Q}_\ell)$ is $\mathfrak{sl}_2(\mathbb{Q}_\ell)$ itself, $L(G)$ should be very large even if $\prol$ is very small.

In what follows we prove (i) of theorem \ref{thm:GeneralReduction} first when $|G/\prol|=2$ and then in case $G(\ell)$ is respectively contained in a split Cartan, Borel, or nonsplit Cartan subgroup; we then discuss the optimality of the statement, showing through examples that it cannot be extended to the exceptional case and that $\ell^{2s}$ cannot be replaced by anything smaller. Finally, in section \ref{sec:BoundedIndex} we finish the proof of theorem \ref{thm:Reconstruction2}.

\smallskip

\noindent \textbf{Notation.}
For $x \in L(G)$ we set $\pi_{ij}(x) =x_{ij}$, the coefficient in the $i$-th row and $j$-th column of the matrix representation of $x$ in $\mathfrak{sl}_2(\mathbb{Z}_\ell)$. The maps $\pi_{ij}$ are obviously linear and continuous.

\subsection{The case $|G/\prol|=2$}
Suppose first that $G(\ell)$ is contained in a Cartan subgroup, so that $G/N(G) \cong G(\ell)$. The only nontrivial element $x$ in $G(\ell)$ satisfies the relations $x^2=\operatorname{Id}$ and $\det(x)=1$, so it must be $-\operatorname{Id}$. It follows that $G$ contains an element $g$ of the form $-\operatorname{Id}+\ell A$ for a certain $A \in \operatorname{M}_2(\mathbb{Z}_\ell)$. Considering the sequence
\[
g^{\ell^n}=\left( -\operatorname{Id}+\ell A \right)^{\ell^n} = -\operatorname{Id} + O(\ell^{n+1})
\]
and given that $G$ is closed we see that $-\operatorname{Id}$ is in $G$. Next observe that for every $h \in G$ either $h$ or $-h$ belongs to $\prol$. If $g_1, g_2, g_3$ are elements of $G$ such that $\Theta(g_1),\Theta(g_2),\Theta(g_3)$ is a basis for $L(G)$, then on the one hand for each $i$ either $g_i$ or $-g_i$ belongs to $\prol$, and on the other $\Theta(-g_i)=-\Theta(g_i)$, so $L(G)=L(\prol)$ and the claim follows. 

Next suppose $G(\ell)$ is contained in a Borel subgroup. We can assume that the order of $G(\ell)$ is divisible by $\ell$, for otherwise $G(\ell)$ is cyclic and we are back to the previous case. The canonical projection $G \to G/\prol$ factors as
\[
\begin{array}{ccccc}
G & \to & G(\ell) & \to & \mathbb{F}_\ell^\times \\
g & \mapsto  & \left( \begin{matrix} a & b \\ 0 & 1/a \end{matrix}\right) & \mapsto & a,
\end{array}
\]
so if $G/\prol$ has order 2 we can find in $G(\ell)$ an element of the form $\left( \begin{matrix} -1 & b \\ 0 & -1 \end{matrix}\right)$. Taking the $\ell$-th power of this element shows that $G(\ell)$ contains $-\operatorname{Id}$ and we conclude as above.

\subsection{The split Cartan case}
Suppose that $G(\ell)$ is contained in a split Cartan, so that, by choosing a suitable basis, we can assume that $G(\ell)$ is contained in the subgroup of diagonal matrices of $\operatorname{SL}_2(\mathbb{F}_\ell)$. Fix an element $g \in G$ such that $[g] \in G(\ell)$ is a generator. By assumption the order of $[g]$ is not 4, and by the previous paragraph we can assume it is not 2; furthermore it is not divisible by $\ell$. The minimal polynomial of $[g]$ is then separable, and $[g]$ has two distinct eigenvalues in $\mathbb{F}_\ell^\times$. It follows that $g$ can be diagonalized over $\mathbb{Z}_\ell$ (its characteristic polynomial splits by Hensel's lemma), 
and we can choose a basis in which $g=\left(\begin{matrix} a & 0 \\ 0 & 1/a \end{matrix} \right)$, where $a$ is an $\ell$-adic unit. Note that our assumption that $|G(\ell)|$ does not divide 4 implies in particular that $a^4 \not \equiv 1 \pmod \ell$. A fortiori $\ell$ does not divide $a^2-1$, so the diagonal coefficients of $\Theta(g)=\left(\begin{matrix} \frac{a^2-1}{2a} & 0 \\ 0 & -\frac{a^2-1}{2a} \end{matrix}\right)$ are $\ell$-adic units.
The following lemma allows us to choose a basis of $L(G)$ containing $\Theta(g)$:

\begin{lemma}\label{lemma:Basis}
Suppose $g \in G$ is such that $\Theta(g)$ is not zero modulo $\ell$. The algebra $L(G)$ admits a basis of the form $\Theta(g), \Theta(g_2),\Theta(g_3)$, where $g_2, g_3$ are in $G$.
\end{lemma}

\begin{proof}
Recall that $L(G)$ is of rank 3 since it contains $\ell^s \mathfrak{sl}_2(\mathbb{Z}_\ell)$. Start by choosing $g_1,g_2,g_3 \in G$ such that $\Theta(g_1),\Theta(g_2),\Theta(g_3)$ is a basis for $L(G)$. As $\Theta(g)$ is not zero modulo $\ell$, from an equality of the form
\[
\Theta(g) = \sum_{i=1}^3 \lambda_i \Theta(g_i)
\]
we deduce that at least one of the $\lambda_i$ is an $\ell$-adic unit, and we can assume without loss of generality that it is $\lambda_1$. But then
\[
\Theta(g_1) = \lambda_1^{-1} \left( \Theta(g) - \lambda_2 \Theta(g_2)-\lambda_3\Theta(g_3) \right),
\]
and we can replace $g_1$ with $g$.
\end{proof}

Recall that we denote by $\coniug$ the endomorphism of $\mathfrak{sl}_2(\mathbb{Z}_\ell)$ given by $x \mapsto g^{-1}xg$. We now prove that $L(\prol)$ is $\coniug$-stable and, more generally, describe the $\coniug$-stable subalgebras of $\mathfrak{sl}_2(\mathbb{Z}_\ell)$.

\begin{lemma}\label{lemma:LIsCStable}
Let $\ell$ be an odd prime, $G$ a closed subgroup of $\operatorname{GL}_2(\mathbb{Z}_\ell)$, $N$ a normal closed subgroup of $G$ and $g$ an element of $G$. The special Lie algebra $L(N)$ is stable under $\coniug$.
\end{lemma}
\begin{proof}
As $\Theta(N)$ generates $L(N)$ it is enough to prove that $\coniug$ stabilizes $\Theta(N)$. Let $x=\Theta(n)$ for a certain $n \in N$: then
\[
g^{-1}xg = g^{-1} \left( n - \frac{\operatorname{tr}(n)}{2} \operatorname{Id} \right) g = g^{-1}ng - \frac{\operatorname{tr}(g^{-1}ng)}{2} \operatorname{Id}=\Theta(g^{-1}ng),
\]
and this last element is in $\Theta(N)$ since $N$ is normal in $G$.
\end{proof}

\begin{lemma}\label{lemma:StableNotI} Let $s$ be a non-negative integer. Let $L$ be a $\coniug$-stable Lie subalgebra of $\mathfrak{sl}_2(\mathbb{Z}_\ell)$ and $x_{11}$, $x_{12}$, $x_{21}$, $y_{11}$, $y_{12}$, $y_{21}$ be elements of $\mathbb{Z}_\ell$ with $v_\ell(x_{21}) \leq s$ and $v_\ell(y_{12}) \leq s$. If $L$ contains both $l_1=\left( \begin{matrix} x_{11} & x_{12} \\ x_{21} & -x_{11} \end{matrix} \right)$ and $l_2=\left( \begin{matrix} y_{11} & y_{12} \\ y_{21} & -y_{11} \end{matrix} \right)$, then it contains all of $\ell^{2s}\mathfrak{sl}_2(\mathbb{Z}_\ell)$.
\end{lemma}
\begin{proof}
Consider first the case $x_{12}=y_{21}=0$. We compute
\[
\coniug(l_1) = \left( \begin{matrix} x_{11} & 0 \\ a^{2}x_{21} & -x_{11} \end{matrix} \right),
\]
so $L$ contains $\left( \begin{matrix} x_{11} & 0 \\ a^{2}x_{21} & -x_{11} \end{matrix} \right)-l_1=\left( \begin{matrix} 0 & 0 \\ (a^{2}-1)x_{21} & 0 \end{matrix} \right)$, where by our hypothesis on $a$ the valuation of the bottom-left coefficient is at most $s$. Analogously, $L$ contains $\left( \begin{matrix} 0 & (a^{2}-1)y_{12} \\ 0 & 0 \end{matrix} \right)$, and since it is a Lie algebra it also contains the commutator
\[
\left[\left( \begin{matrix} 0 & (a^{2}-1)y_{12} \\ 0 & 0 \end{matrix} \right) , \left( \begin{matrix} 0 & 0 \\ (a^{2}-1)x_{21} & 0 \end{matrix} \right) \right] = \left( \begin{matrix} (a^{2}-1)^2 x_{21}y_{12} & 0 \\ 0 & -(a^{2}-1)^2x_{21}y_{12} \end{matrix} \right),
\]
whose diagonal coefficients have valuation at most $2s$. This establishes the lemma in case $x_{12}$ and $y_{21}$ are both zero, since the three elements we have found generate $\ell^{2s}\mathfrak{sl}_2(\mathbb{Z}_\ell)$.
The general case is then reduced to the previous one by replacing $l_1,l_2$ with
\[
a^2 \coniug(l_1)-l_1=\left( \begin{matrix} (a^2-1)x_{11} & 0 \\ (a^{4}-1)x_{21} & -(a^2-1)x_{11} \end{matrix} \right)
\]
and $a^{-2}\coniug(l_2)-l_2$, and noticing that since $\ell \nmid a^4-1$ we have $v_\ell((a^{4}-1)x_{21})=v_\ell(x_{21})$ and $v_\ell((a^{-4}-1)y_{12})=v_\ell(y_{12})$).
\end{proof}

\smallskip

We know from lemma \ref{lemma:LIsCStable} that $L(\prol)$ is $\coniug$-stable, so in order to apply lemma \ref{lemma:StableNotI} to $L(\prol)$ we just need to find two elements $l_1, l_2$ in $L(\prol)$ with the property that $v_\ell \circ \pi_{21}(l_1) \leq s$ and $v_\ell \circ \pi_{12}(l_2) \leq s$. Since the values of the diagonal coefficients do not matter for the application of this lemma we will simply write $\ast$ for any diagonal coefficient appearing from now on. In particular we write $g_2,g_3,\Theta(g_2),\Theta(g_3)$ in coordinates as follows:
\[
g_i=\left(\begin{matrix} \ast & g^{(i)}_{12} \\ g^{(i)}_{21} & \ast \end{matrix}\right) , \Theta(g_i)= \left(\begin{matrix} \ast & g^{(i)}_{12} \\ g^{(i)}_{21} & \ast \end{matrix}\right).
\]

As $[g]$ generates $G(\ell)$, for $i=2,3$ there exist $k_i \in \mathbb{N}$ such that $[g_i]=[g]^{k_i}$, or equivalently such that $g^{-k_i}g_i \in \prol$. Since $\Theta(g), \Theta(g_2),\Theta(g_3)$ generate $\ell^{s}\mathfrak{sl}_2(\mathbb{Z}_\ell)$, but the off-diagonal coefficients of $\Theta(g)$ vanish, we can choose two indices $i_1, i_2 \in \left\{2,3\right\}$ such that $v_\ell \circ \pi_{21}(\Theta(g_{i_1})) \leq s$ and $v_\ell \circ \pi_{12}(\Theta(g_{i_2})) \leq s$. On the other hand, $L(N(G))$ contains
\[
\Theta(g^{-k_i}g_i)=\Theta \left( \left( \begin{matrix} a^{-k_i} & 0 \\ 0 & a^{k_i} \end{matrix} \right) \left(\begin{matrix} \ast & g^{(i)}_{12} \\ g^{(i)}_{21} & \ast \end{matrix}\right) \right) = \left( \begin{matrix} \ast & a^{-k_i} g^{(i)}_{12} \\ a^{k_i} g^{(i)}_{21} & \ast \end{matrix} \right),
\]
where $a^{\pm k_i}$ is an $\ell$-adic unit. The $\ell$-adic valuation of the off-diagonal coefficients of $\Theta(g^{-k_{i}}g_{i})$ is then the same as that of the corresponding coefficients of $\Theta(g_{i})$, and we find two elements $l_1=\Theta(g^{-k_{i_1}} g_{i_1})$ and $l_2=\Theta(g^{-k_{i_2}} g_{i_2})$ that satisfy $v_\ell \circ \pi_{21}(l_1) \leq s$ and $v_\ell \circ \pi_{12}(l_2) \leq s$ as required. We can now apply lemma \ref{lemma:StableNotI} with $(L,g,l_1,l_2)=(L(\prol),g,\Theta(g_{i_1}),\Theta(g_{i_2}))$ and deduce that $L(\prol)$ contains $\ell^{2s}\mathfrak{sl}_2(\mathbb{Z}_\ell)$, as claimed.

\subsection{The Borel case}
Suppose $G(\ell)$ is included in a Borel subgroup. If the order of $G(\ell)$ is prime to $\ell$, then $G(\ell)$ is in fact contained in a split Cartan subgroup, and we are reduced to the previous case. We can therefore assume without loss of generality that the order of $G(\ell)$ is divisible by $\ell$.
In this case we know that $\prol$ is the inverse image in $G$ of the unique $\ell$-Sylow of $G(\ell)$, and that the canonical projection $G \to G/\prol$ factors as
\[
\begin{array}{ccccc}
G & \to & G(\ell) & \to & \mathbb{F}_\ell^\times \\
g & \mapsto  & \left( \begin{matrix} a & b \\ 0 & 1/a \end{matrix}\right) & \mapsto & a.
\end{array}
\]

Let $H$ be the image of this map. The group $H$ is cyclic and we can assume that its order does not divide 4: it is not 4 by hypothesis and if it is 1 or 2 we are done. Let $g$ be any inverse image in $G$ of a generator of $H$. The matrix representing $g$ can be diagonalized over $\mathbb{Z}_\ell$ since the characteristic polynomial of $[g] \in G(\ell)$ is separable, and the same exact argument as in the previous paragraph shows that we can choose a basis of $L(G)$ of the form $\Theta(g),\Theta(g_2),\Theta(g_3)$. By definition of $H$ we see that for $i=2,3$ there is an integer $k_i$ such that $[g_i]=[g]^{k_i}$ in $G/\prol$, and the rest of the proof is identical to that of the previous paragraph.

\subsection{The nonsplit Cartan case}
Suppose now that $G(\ell)$ is contained in a nonsplit Cartan subgroup. Fix a $g \in G$ such that $[g]$ generates $G(\ell)$. We know that $[g]$ is of the form $\left( \begin{matrix} [a] & [b\varepsilon] \\ [b] & [a] \end{matrix} \right)$, where $[\varepsilon]$ is a fixed quadratic nonresidue modulo $\ell$. In order to put $g$ into a standard form we need the following elementary lemma, which is an $\ell$-adic analogue of the Jordan canonical form over the reals.

\begin{lemma}
Up to a choice of basis of $\mathbb{Z}_\ell^2$, the matrix representing $g$ can be chosen to be of the form $\left( \begin{matrix} a & b\varepsilon \\ b & a \end{matrix} \right)$ for certain $a,b,\varepsilon$ lifting $[a],[b],[\varepsilon]$, and where moreover $a,b$ are $\ell$-adic units.
\end{lemma}
\begin{proof}
The characteristic polynomial of $[g]$ splits over $\mathbb{F}_\ell\left[\sqrt{[\varepsilon]}\right]$, so by Hensel's lemma the characteristic polynomial of $g$ splits over $\mathbb{Z}_\ell\left[\sqrt{\varepsilon} \right]$. The two eigenvalues of $g$ in $\mathbb{Z}_\ell\left[\sqrt{\varepsilon} \right]$ are of the form $a \pm b \sqrt{\varepsilon}$ for certain $a,b \in \mathbb{Z}_\ell$ (the notation is coherent: since the eigenvalues of $[g]$ are simply the projections of the eigenvalues of $g$, we have that $a,b$ map respectively to $[a],[b]$ modulo $\ell$).

By definition of eigenvalue we can find a vector $\mathbf{v}_{+} \in \mathbb{Z}_\ell[\sqrt{\varepsilon}]^2$ such that $g \mathbf{v}_{+} = (a + b \sqrt{\varepsilon})\mathbf{v}_{+}$. Normalize $\mathbf{v}_{+}$ in such a way that at least one of its coordinates is an $\ell$-adic unit, write $\mathbf{v}_+ = \mathbf{w}+\mathbf{z}\sqrt{\varepsilon}$ for certain $\mathbf{w},\mathbf{z} \in \mathbb{Z}_\ell^2$ and set $\mathbf{v}_-=\mathbf{w}-\mathbf{z}\sqrt{\varepsilon}$. As $g$ has its coefficients in $\mathbb{Z}_\ell$, the vector $\mathbf{v}_-$ is an eigenvector for $g$, associated with the eigenvalue $a-b\sqrt{\varepsilon}$. The projections of $\mathbf{v}_{\pm}$ in $\left( \mathbb{F}_\ell\left[\sqrt{[\varepsilon]}\right] \right)^2$ are therefore nonzero eigenvectors of $[g]$ corresponding to different eigenvalues, hence they are linearly independent. It follows that $\mathbf{w}=\frac{\mathbf{v}_+ + \mathbf{v}_-}{2}, \mathbf{z}=\frac{\mathbf{v}_+ - \mathbf{v}_-}{2 \sqrt{\varepsilon}}$ are independent modulo $\ell \mathbb{Z}_\ell[\sqrt{\varepsilon}]$, and since $\mathbf{w}, \mathbf{z}$ lie in $\mathbb{Z}_\ell^2$ they are a fortiori independent modulo $\ell$. The matrix $\left( \mathbf{z} \bigm\vert \mathbf{w} \right)$ is then invertible modulo $\ell$, so it lies in $\operatorname{GL}_2(\mathbb{Z}_\ell)$ and can be used as base-change matrix. It is now straightforward to check that in this basis the element $g$ is represented by the matrix $\left( \begin{matrix} a & b\varepsilon \\ b & a \end{matrix} \right)$. Finally notice that $a$ and $b$ are units: if $[b]=0$ or $[a]=0$ it is easy to check that the order of $G(\ell)$ divides 4, against the assumptions.
\end{proof}

\medskip

We can also assume that $G$ contains $-\operatorname{Id}$, since replacing $G$ with $G \cdot \left\{\pm \operatorname{Id} \right\}$ does not alter neither the derived subgroup nor the special Lie algebra of $G$. By lemma \ref{lemma:Basis} the algebra $L(G)$ admits a basis of the form $\Theta(g), \Theta(g_2),\Theta(g_3)$, where $g$ is as above and $g_2, g_3$ are in $G$. We write in coordinates
\[
g_2=\left(\begin{matrix} y_{11} & y_{12} \\ y_{21} & y_{22} \end{matrix}\right), \Theta(g_2)= \left(\begin{matrix} \frac{y_{11}-y_{22}}{2} & y_{12} \\ y_{21} & -\frac{y_{11}-y_{22}}{2} \end{matrix}\right),
\]
\[
g_3=\left(\begin{matrix} z_{11} & z_{12} \\ z_{21} & z_{22} \end{matrix}\right), \Theta(g_3)= \left(\begin{matrix} \frac{z_{11}-z_{22}}{2} & z_{12} \\ z_{21} & -\frac{z_{11}-z_{22}}{2} \end{matrix}\right).
\]

\subsubsection{Projection operators, $\coniug$-stable subalgebras}
Recall that $\coniug$ denotes $x \mapsto g^{-1}xg$. Following our general strategy we now describe projection operators associated with the action of $\coniug$ and $\coniug$-stable subalgebras of $\mathfrak{sl}_2(\mathbb{Z}_\ell)$.

\begin{lemma}\label{lemma:NonsplitCartanProjection}
Let $E,F \in \mathbb{Z}_\ell$. If the matrix $\left(\begin{matrix} -F & -\varepsilon E \\ E & F \end{matrix}\right)$ belongs to $L(\prol)$, then $L(\prol)$ also contains
\[
\left(\begin{matrix} -F & 0 \\ 0 & F \end{matrix}\right), \left(\begin{matrix} -E & 0 \\ 0 & E \end{matrix}\right), \left(\begin{matrix} 0 & -\varepsilon E \\ E & 0 \end{matrix}\right), \text{ and } \left(\begin{matrix} 0 & -\varepsilon F \\ F & 0 \end{matrix}\right).
\]
\end{lemma}
\begin{proof}
We know from lemma \ref{lemma:LIsCStable} that $L(\prol)$ is $\coniug$-stable, so the identity
\begin{equation}\label{eq:NonsplitCartanRotationIdentity}
\frac{1}{2a b}\left( \coniug \left(\begin{array}{cc} -F  & -\varepsilon E \\ E & F\end{array}\right) -\left(a^2+b^2 \varepsilon \right)\left(\begin{array}{cc} \displaystyle -F  & -\varepsilon  E \\ E & F\end{array}\right) \right) = \left(
\begin{array}{cc}
 - \varepsilon E  & - \varepsilon F  \\
 F &  \varepsilon  E
\end{array}
\right)
\end{equation}
shows that $\left(
\begin{array}{cc}
 - \varepsilon E  & - \varepsilon F  \\
 F & \varepsilon  E
\end{array}
\right)$ is in $L(\prol)$. At least one between $F/E$ and $E/F$ is an $\ell$-adic integer, and we can assume it is $F/E$ (the other case being perfectly analogous). In particular we have $v_\ell(F) \geq v_\ell(E)$. It follows that $L(\prol)$ contains
\[
\frac{F}{E} \left( \begin{matrix} - F & -\varepsilon E \\ E & F \end{matrix}\right) - \left( \begin{matrix} - \varepsilon E & -\varepsilon F \\ F & \varepsilon E \end{matrix}\right)=\left(\begin{matrix} \frac{ \varepsilon E^2 - F^2}{E} & 0 \\ 0 &   - \frac{\varepsilon E^2 - F^2}{E} \end{matrix} \right).
\]
If $v_\ell(F) > v_\ell(E)$ we have $v_\ell(\varepsilon E^2-F^2)=2v_\ell(E)$, while if $v_\ell(F)=v_\ell(E)$ we can write
\[
F=\ell^{v_\ell(E)} \zeta, \; E=\ell^{v_\ell(E)} \gamma,
\]
where $\zeta, \gamma$ are not zero modulo $\ell$. In this second case we have $\varepsilon E^2- F^2=\ell^{2v_\ell(E)}\left(\varepsilon \gamma^2-  \zeta^2\right)$, and $\left(\varepsilon \gamma^2-\zeta^2\right)$ does not vanish modulo $\ell$ since $[\varepsilon]$ is not a square in $\mathbb{F}_\ell^\times$. Hence $v_\ell(\varepsilon E^2-F^2)=2v_{\ell}(E)$ holds in any case, and (due to the denominator $E$) we have found in $L(\prol)$ a matrix whose off-diagonal coefficients vanish and whose diagonal coefficients have the same valuation as $E$. By the stability of $L(\prol)$ under multiplication by $\ell$-adic units we have thus proved that $L(\prol)$ contains $\left(\begin{matrix} -E & 0 \\ 0 & E \end{matrix}\right)$. Identity (\ref{eq:NonsplitCartanRotationIdentity}) applied to this element shows that $L(\prol)$ also contains $\left(
\begin{array}{cc}
 0 & -\varepsilon E  \\
 E & 0
\end{array}
\right)$, hence by difference $\left(\begin{matrix} -F & 0 \\ 0 & F \end{matrix}\right)$ is in $L(\prol)$ as well. Applying equation (\ref{eq:NonsplitCartanRotationIdentity}) to this last matrix we finally deduce that $L(\prol)$ also contains $\left(\begin{matrix} 0 & -\varepsilon F \\ F & 0 \end{matrix}\right)$.
\end{proof}

\begin{lemma}\label{lemma:NonsplitCartanFinal}
Let $E,F$ be elements of $\mathbb{Z}_\ell$ satisfying $\min\left\{v_\ell(F), v_\ell(E) \right\} \leq s$. If $\left(\begin{matrix} -F & -\varepsilon E \\ E & F \end{matrix}\right)$ belongs to $L(\prol)$, then $L(\prol)$ contains $\ell^{2s} \mathfrak{sl}_2(\mathbb{Z}_\ell)$.
\end{lemma}

\begin{proof}
Suppose $v_\ell(F) \leq s$, the other case being similar. The special Lie algebra $L(\prol)$ contains $\left(\begin{matrix} -F & 0 \\ 0 & F \end{matrix}\right)$, $\left(\begin{matrix} 0 & -\varepsilon F \\ F & 0 \end{matrix}\right)$ by the previous lemma, so (given that $v_\ell(F) \leq s$) it also contains $\ell^s \left(\begin{matrix} 1 & 0 \\ 0 & -1 \end{matrix}\right)$, $\ell^s \left(\begin{matrix} 0 & -\varepsilon  \\ 1 & 0 \end{matrix}\right)$. Taking the commutator of these two elements yields another element of $L(\prol)$, namely
\[
\left[\ell^s \left( \begin{matrix} 0 & -\varepsilon \\ 1 & 0 \end{matrix} \right), \ell^s \left( \begin{matrix} 1 & 0 \\ 0 & -1 \end{matrix} \right) \right] = \ell^{2s} \left( \begin{matrix} 0 & 2\varepsilon \\ 2 & 0 \end{matrix} \right).
\]

Finally, since
\[
\frac{1}{2}\ell^{2s} \left( \begin{matrix} 0 & 2\varepsilon \\ 2 & 0 \end{matrix} \right) + \ell^{2s} \left( \begin{matrix} 0 & -\varepsilon \\ 1 & 0 \end{matrix} \right) = \ell^{2s} \left( \begin{matrix} 0 & 0 \\ 2 & 0 \end{matrix} \right),
\]
it is immediately checked that $L(\prol)$ contains a basis of $\ell^{2s} \mathfrak{sl}_2(\mathbb{Z}_\ell)$ as desired.
\end{proof}

\subsubsection{The case when $g_2, g_3 \notin \prol$.}
Let us assume for now that $g_i \not \in \prol$ and $-g_i \not \in \prol$ for $i=2,3$. We will deal later with the case when some of these elements already belong to $\prol$. Given that by hypothesis $L(G)$ contains $\ell^s \mathfrak{sl}_2(\mathbb{Z}_\ell)$ we must have a representation
\[
\ell^s \left( \begin{matrix} 1 & 0 \\ 0 & -1 \end{matrix} \right) = \sum_{i=1}^3 \lambda_i \Theta(g_i)
\]
for certain scalars $\lambda_1,\lambda_2,\lambda_3 \in \mathbb{Z}_\ell$. However, the diagonal coefficients of $\Theta(g)$ vanish, therefore there exists an index $i \in \left\{2,3\right\}$ such that $v_\ell \circ \pi_{11}(\Theta(g_i)) \leq s$. Renumbering $g_2, g_3$ if necessary we can assume $i=2$. In coordinates, the condition $v_\ell \circ \pi_{11}(\Theta(g_2)) \leq s$ becomes $v_\ell(y_{11}-y_{22}) \leq s$.

\medskip

Now since $[g]$ generates $G(\ell)$ there is an integer $k$ such that $[g]^{-k}=[g_2]$ in $G(\ell)$; in other words, both $g_2 g^{k}$ and $g^{k} g_2$ are trivial modulo $\ell$ and therefore belong to $\prol$. It is immediate to check that the matrix $g^k$ is of the form $\left(\begin{matrix} c & d \varepsilon \\ d & c \end{matrix} \right)$ for certain $c,d \in \mathbb{Z}_\ell$. Now if $d$ is $0$ modulo $\ell$, then (since $c^2-\varepsilon d^2 \equiv 1 \pmod \ell$) we have $c \equiv \pm 1 \pmod \ell$, so either $g_2$ or $-g_2$ reduces to the identity modulo $\ell$ and is therefore in $\prol$, against our assumption. Hence $d$ is an $\ell$-adic unit. We then introduce
\[
g_4=\left(
\begin{array}{cc} c & d \varepsilon  \\ d & c\end{array}\right)\left(\begin{array}{cc} y_{11} & y_{12} \\ y_{21} & y_{22}\end{array}\right), \; g_5=\left(\begin{array}{cc} y_{11} & y_{12} \\ y_{21} & y_{22}\end{array}\right)\left(\begin{array}{cc} c & d \varepsilon  \\ d & c\end{array}\right).
\]

By construction $g_4$ and $g_5$ are elements of $\prol$, whence $\Theta(g_4), \Theta(g_5)$ are elements of $L(\prol)$. In particular $L(\prol)$ contains their difference
\[
\Theta(g_4)-\Theta(g_5)=g_4-g_5=\left(
\begin{array}{cc}
 -d (y_{12}- \varepsilon  y_{21}) & d \varepsilon  \left(-y_{11}+y_{22}\right) \\
 d \left(y_{11}-y_{22}\right) & d \left(y_{12}-\varepsilon  y_{21}\right)
\end{array}
\right),
\]
where (given that $d,\varepsilon$ are $\ell$-adic units) $v_\ell \circ \pi_{21} (\Theta(g_4)-\Theta(g_5)) \leq s$ and $v_\ell \circ \pi_{12} (\Theta(g_4)-\Theta(g_5)) \leq s$. Applying lemma \ref{lemma:NonsplitCartanFinal} to the element $\Theta(g_4)-\Theta(g_5)$ we have just constructed we therefore deduce $L(\prol) \supseteq \ell^{2s} \mathfrak{sl}_2(\mathbb{Z}_\ell)$ as desired.

\subsubsection{The case when one generator belongs to $\prol$.}
Let $x=\left(
\begin{array}{cc}
 x_{11} & x_{12} \\
 x_{21} & -x_{11}
\end{array}
\right)$ denote any element of $\mathfrak{sl}_2(\mathbb{Z}_\ell)$. It is easy to check that
\[
\frac{1}{2ab} \left((3 + 4 \varepsilon b^2) (\coniug x-x) - \coniug(\coniug x-x) \right)=\left(\begin{array}{cc}
 x_{12}-\varepsilon  x_{21} & 2\varepsilon x_{11} \\
 -2 x_{11} & - x_{12}+ \varepsilon  x_{21}
\end{array}
\right),
\]
and furthermore if $x$ belongs to $L(\prol)$, then $\left(\begin{array}{cc}
 x_{12}-\varepsilon  x_{21} & 2\varepsilon x_{11} \\
 -2 x_{11} & - x_{12}+ \varepsilon  x_{21}
\end{array}
\right)$ is in $L(\prol)$ as well.

Suppose now that either $g_2$ or $-g_2$ (resp. $g_3$ or $-g_3$) belongs to $\prol$. Since $\Theta(-g_i) =-\Theta(g_i)$ we can assume that $g_2$ (resp. $g_3$) itself belongs to $\prol$. Take $\left( \begin{matrix} x_{11} & x_{12} \\ x_{21} & -x_{11} \end{matrix} \right)$ to be $\Theta(g_2)$ (resp. $\Theta(g_3)$). Subtracting $\displaystyle \frac{x_{21}}{b}\Theta(g_1)$ from $\Theta(g_2)$ we get $\left( \begin{matrix} x_{11} & x_{12}-\varepsilon x_{21} \\ 0 & -x_{11} \end{matrix} \right) \in L(G)$, and since we know that
\[
\displaystyle \Theta(g_2)- \frac{\pi_{21}(\Theta(g_2))}{b} \Theta(g_1), \; \displaystyle \Theta(g_3)- \frac{\pi_{21}(\Theta(g_3))}{b} \Theta(g_1)
\]
together span $\ell^s \left( \begin{matrix} 1 & 0 \\ 0 & -1 \end{matrix} \right) \oplus \ell^s \left( \begin{matrix} 0 & 1 \\ 0 & 0 \end{matrix} \right)$, we see that at least one among the coefficients of the matrix $\displaystyle \Theta(g_2)- \frac{\pi_{21}(\Theta(g_2))}{b} \Theta(g_1) = \Theta(g_2)- \frac{x_{21}}{b} \Theta(g_1)$ must have valuation at most $s$, that is $\min\left\{ v_\ell(x_{11}), v_\ell(x_{12}-\varepsilon x_{21}) \right\} \leq s$. We now apply lemma \ref{lemma:NonsplitCartanFinal} to $\displaystyle  \left(
\begin{array}{cc}
 x_{12}-\varepsilon  x_{21} & 2\varepsilon x_{11} \\
 -2 x_{11} & - x_{12}+ \varepsilon  x_{21}
\end{array}
\right)$, which is in $L(\prol)$, to deduce $L(\prol) \supseteq \ell^{2s} \mathfrak{sl}_2(\mathbb{Z}_\ell)$, and we are done.

\subsection{Optimality}
The following examples show that it is neither possible to extend theorem \ref{thm:Reconstruction2} to the exceptional case nor to improve the exponent $2s$.

\begin{proposition}
Let $\ell$ be a prime $\equiv 1 \pmod 4$. For every $t \geq 1$ there exists a closed subgroup $G$ of $\operatorname{SL}_2(\mathbb{Z}_\ell)$ whose special Lie algebra is $\mathfrak{sl}_2(\mathbb{Z}_\ell)$ and whose maximal pro-$\ell$ subgroup is contained in $\mathcal{B}_\ell(t)$.
\end{proposition}

\begin{proof}
Notice that the following six elements form a finite subgroup $H$ of $\operatorname{PSL}_2(\mathbb{Z}[i])$
\[
\left(\begin{matrix} 1&0 \\ 0&1 \end{matrix} \right),\left(\begin{matrix} 0&1 \\ -1&1 \end{matrix} \right),\left(\begin{matrix} 1&-1 \\ 1&0 \end{matrix} \right), \left(\begin{matrix} 0&i \\ i&0 \end{matrix} \right),\left(\begin{matrix} -i&i \\ 0&i \end{matrix} \right),\left(\begin{matrix} i&0 \\ i&-i \end{matrix} \right),
\]
and that $H$ is isomorphic to $S_3$: indeed, it is the group of permutations of $\left\{0,1,\infty\right\} \subset \mathbb{P}^1\left(\mathbb{Z}[i]\right)$.
The inverse image $\tilde{H}$ of $H$ in $\operatorname{SL}_2(\mathbb{Z}[i])$ is therefore a finite group of cardinality 12. Now since $\ell \equiv 1 \pmod 4$ there is a square root of $-1$ in $\mathbb{Z}_\ell$, so $\mathbb{Z}[i] \hookrightarrow \mathbb{Z}_\ell$ and $\tilde{H} \hookrightarrow \operatorname{SL}_2(\mathbb{Z}_\ell)$. Consider $G=\tilde{H} \cdot \mathcal{B}_\ell(t) \subset \operatorname{SL}_2(\mathbb{Z}_\ell)$. It is clear that $\mathcal{B}_\ell(t)$ is normal in $G$. Since $\frac{G}{\mathcal{B}_\ell(t)}$ is isomorphic to a quotient of $\tilde{H}$ (and therefore has order prime to $\ell$), the subgroup $\mathcal{B}_\ell(t)$ is clearly the maximal pro-$\ell$ subgroup of $G$. Furthermore, the special Lie algebra of $G$ contains the three elements
\[
\Theta\left( \left(\begin{matrix} 0&1 \\ -1&1 \end{matrix} \right)\right)=\left(\begin{matrix} -1/2 & 1 \\ -1 & 1/2 \end{matrix}\right), \, \Theta\left(\left(\begin{matrix} 0&i \\ i&0 \end{matrix} \right)\right)=\left(\begin{matrix} 0&i \\ i&0 \end{matrix} \right), \, \Theta\left(\left(\begin{matrix} i&0 \\ i&-i \end{matrix} \right)\right)=\left(\begin{matrix} i&0 \\ i&-i \end{matrix} \right),
\]
that are readily checked to be a basis of $\mathfrak{sl}_2(\mathbb{Z}_\ell)$.
\end{proof}

\medskip

On the other hand, the following example shows that there exist subgroups of $\operatorname{SL}_2(\mathbb{Z}_\ell)$ such that $L(G)$ contains $\ell^s \mathfrak{sl}_2(\mathbb{Z}_\ell)$, but $L(\prol)$ only contains $\ell^{2s} \mathfrak{sl}_2(\mathbb{Z}_\ell)$. Fix $s \geq 1$, an integer $N>4$ and a prime $\ell$ congruent to 1 modulo $N$; then $\mathbb{Z}_\ell^\times$ contains a primitive $N$-th root of unity $a$, and we let $g=\left(\begin{matrix} a & 0 \\ 0 & 1/a \end{matrix} \right)$. The module $M=\ell^s \left( \begin{matrix} 0 & 1 \\ 0 & 0 \end{matrix} \right) \oplus \ell^s \left( \begin{matrix} 0 & 0 \\ 1 & 0 \end{matrix} \right) \oplus \ell^{2s} \left( \begin{matrix} 1 & 0 \\ 0 & -1 \end{matrix} \right)$ is a Lie subalgebra of $\mathfrak{sl}_2(\mathbb{Z}_\ell)$, so by Theorem 3.4 of \cite{MR1241950}
\[
H = \left\{x \in \operatorname{SL}_2(\mathbb{Z}_\ell) \bigm\vert \operatorname{tr}(x) \equiv 2 \pmod{\ell^{2s}}, \Theta(x) \in M \right\}
\]
is a pro-$\ell$ group with special Lie algebra $M$. Let $G$ be the group generated by $g$ and $H$. Up to units $\Theta(g)$ is $\left( \begin{matrix} 1 & 0 \\ 0 & -1 \end{matrix} \right)$, so $L(G)$ contains all of $\ell^s \mathfrak{sl}_2(\mathbb{Z}_\ell)$. On the other hand, $H$ is normal in $G$: one simply needs to check that $g^{-1}Mg =M$, and this is obvious from the equality
\[
g^{-1} \left(\begin{matrix} x_{11} & x_{12} \\ x_{21} & -x_{11} \end{matrix} \right) g = \left(
\begin{array}{cc}
 x_{11} & \frac{x_{12}}{a^2} \\
 a^2 x_{21} & -x_{11}
\end{array}
\right).
\]

Finally, $H$ is maximal among the pro-$\ell$ subgroups of $G$, since $G/H$ is a quotient of $\langle g \rangle \cong \mathbb{Z}/N\mathbb{Z}$, hence of order prime to $\ell$. Therefore $\prol=H$ and $L(\prol)=L(H)=M$ contains $\ell^t \mathfrak{sl}_2(\mathbb{Z}_\ell)$ only for $t \geq 2s$.

\subsection{Proof of theorem \ref{thm:Reconstruction2}}\label{sec:BoundedIndex}
We now prove (i) of theorem \ref{thm:Reconstruction2} by reducing it to the corresponding statement in theorem \ref{thm:GeneralReduction}. 

As $G$ and $\operatorname{Sat}(G)$ have the same special Lie algebra and derived subgroup we can assume $G=\operatorname{Sat}(G)$. As $G$ is saturated and satisfies the condition on the determinant, we know from lemma \ref{lemma:PropertiesOfSaturation} that $G=\operatorname{Sat}(H)$ for $H=G^{\det=1}$. By the same lemma we also have $L(H)=L(G)$ and $G'=H'$. 

By assumption $H$ satisfies the hypotheses of theorem \ref{thm:GeneralReduction} (i), so $H$ has property $(\star)$. As $L(G)=L(H)$ contains $\ell^s\mathfrak{sl}_2(\mathbb{Z}_\ell)$ we deduce that $L_0=L(N(H))$ contains $\ell^{2s}\mathfrak{sl}_2(\mathbb{Z}_\ell)$, and since $N(H)$ is a pro-$\ell$ group we can apply theorem \ref{thm:PinkSL2} to it. In order to do so we need to estimate $C(N(H))=\operatorname{tr}\left(L_0 \cdot L_0 \right)$ and $[L_0,L_0]$. Note that
\[
C(N(H)) \ni \operatorname{tr} \left( \ell^{2s} \left( \begin{matrix} 1 & 0 \\ 0 & -1 \end{matrix} \right) \cdot \ell^{2s} \left( \begin{matrix} 1 & 0 \\ 0 & -1 \end{matrix} \right) \right) = 2\ell^{4s},
\]
so given that $\ell$ is odd we have $C(L_0) \supseteq (2\ell^{4s})=(\ell^{4s})$. Likewise,
\[
[L_0,L_0] \supseteq [\ell^{2s} \mathfrak{sl}_2(\mathbb{Z}_\ell),\ell^{2s} \mathfrak{sl}_2(\mathbb{Z}_\ell)]=\ell^{4s} \mathfrak{sl}_2(\mathbb{Z}_\ell),
\]
so the derived subgroup of $N(H)$ (which is clearly included in $H'=G'$) is
\[
N(H)' = \left\{ x \in \operatorname{SL}_2(\mathbb{Z}_\ell) \bigm\vert \operatorname{tr} x-2 \in C(N(H)), \Theta(x) \in [L_0,L_0] \right\},
\]
and by the above it contains
\[
\left\{ x \in \operatorname{SL}_2(\mathbb{Z}_\ell) \bigm\vert \operatorname{tr} x \equiv 2 \pmod{\ell^{4s}}, \Theta(x) \equiv 0 \pmod{\ell^{4s}} \right\} \supseteq \mathcal{B}_\ell(4s),
\]
which concludes the proof of (i).

\medskip

We are now left with the task of proving (ii). Consider first the map
\[
G \stackrel{\det}{\rightarrow} \mathbb{Z}_\ell^\times \to \frac{\mathbb{Z}_\ell^\times}{\mathbb{Z}_\ell^{\times 2}} \cong \frac{\mathbb{Z}}{2\mathbb{Z}}
\]
and let $G_1$ be its kernel: then $[G:G_1] \leq 2$, so we can replace $G$ with $G_1$ and assume that the condition on the determinant is satisfied. We are reduced to showing that, under this hypothesis, either $G'=\operatorname{SL}_2(\mathbb{Z}_\ell)$ or there exists a subgroup $H$ of index at most $12$ that satisfies the right conditions on $\operatorname{Sat}(H)^{\det=1}$. For notational simplicity we let $\pi$ denote the projection map $G \to G(\ell)$. We now distinguish cases according to $\ell$ and $G(\ell)$ (cf. theorem \ref{thm:Dickson}):

\smallskip

\noindent - if $\ell \geq 5$ and $G(\ell)$ contains $\operatorname{SL}_2(\mathbb{F}_\ell)$, then it follows from lemma \ref{lemma_SerreLift} that $G'=\operatorname{SL}_2(\mathbb{Z}_\ell)$.

\noindent - if $\ell = 3$ we let $S$ denote either a $3$-Sylow of $G(3)$, if the order of $G(3)$ is a multiple of 3, or the trivial group $\{\operatorname{Id}\}$, if it is not. Notice that $G(3)$ is a subgroup of $\left\{g \in \operatorname{GL}_2(\mathbb{F}_3) \bigm\vert \det(g) \mbox{ is a square} \right\}$, which has order 24, so the index $[G(3):S]$ is at most 8. We set $H=\pi^{-1}(S)$. It is clear that $[G:H] \leq 8$, and $H$ satisfies the conditions in (i) by remark \ref{rem_ReadableConditions}, because $(\operatorname{Sat}H)^{\det=1}(3)$ is either $\{\pm \operatorname{Id}\}$ or a group of order 6.

\noindent - if $G(\ell)$ is exceptional, then by lemma \ref{lemma:AbSubgroups} there exists a cyclic subgroup $B$ of $\mathbb{P}G(\ell)$ with $[\mathbb{P}G(\ell) : B] \leq 12$: such a $B$ can be taken to have order $3$ (resp.~$5$) if $\mathbb{P}G(\ell)$ is isomorphic to $A_4$ or $S_4$ (resp.~to $A_5$). Fix a generator $[b]$ of $B$ and let $\xi$ be the composition $G \to G(\ell) \to \mathbb{P}G(\ell)$. We set $H:=\xi^{-1}(B)$; it is clear that $[G:H] \leq 12$. Let now $b \in G(\ell)$ be an element that maps to $[b]$ in $B$, and let $m$ be the (odd) order of $[b]$. We know that $\det b$ is a square in $\mathbb{F}_\ell^\times$, hence there exists a $\lambda \in \mathbb{F}_\ell^\times$ such that $\det (\lambda b)=1$. Notice now that $(\lambda b)^m$ is a homothety (it projects to the trivial element in $\mathbb{P}G(\ell)$) and has determinant 1, so it is either $\operatorname{Id}$ or $-\operatorname{Id}$; replacing $\lambda$ by $-\lambda$ if necessary, we can assume that $(\lambda b)^{m}=-\operatorname{Id}$.
By construction, every element in $\left(\operatorname{Sat}(H)^{\det=1}\right)(\ell)=\operatorname{Sat}(H(\ell))^{\det=1}$ can be written as $\pm (\lambda b)^n$ for some $n \in \mathbb{N}$ and for some choice of sign. Now using the fact that $(\lambda b)^m=-\operatorname{Id}$ we see that $\left(\operatorname{Sat}(H)^{\det=1}\right)(\ell)$ is cyclic, generated by $\lambda b$: since the order of $\lambda b$ is either 6 or 10, $H$ satisfies the conditions in (i) by remark \ref{rem_ReadableConditions}.

\noindent - if $G(\ell)$ is contained in a (split or nonsplit) Cartan subgroup then the same is true for the group $\left(\operatorname{Sat}(G)^{\det=1}\right)(\ell)$. If $\left(\operatorname{Sat}(G)^{\det=1}\right)(\ell)$ does not have order $4$ we are done, so suppose this is the case. Then $\mathbb{P}G(\ell)$ has at most 4 elements, and we can take
\[
H=\ker \left( G \rightarrow G(\ell) \to \mathbb{P}G(\ell) \right):
\]
this $H$ has index at most $4$ in $G$, and $H(\ell)$ has trivial image in $\operatorname{\mathbb{P}GL}_2(\mathbb{F}_\ell)$, so $H(\ell)$ is contained in the homotheties subgroup of $\operatorname{GL}_2(\mathbb{F}_\ell)$.
Therefore $(\operatorname{Sat}(H))^{\det=1}(\ell) = \operatorname{Sat}(H(\ell))^{\det=1}=\left\{ \pm \operatorname{Id} \right\}$ and $H$ satisfies the conditions in (i).

\noindent - if $G(\ell)$ is contained in the normalizer of a (split or nonsplit) Cartan subgroup $\mathcal{C}$, but not in $\mathcal{C}$ itself, then $G$ has a subgroup $G_1$ of index 2 whose image modulo $\ell$ is contained in $\mathcal{C}$, and we are reduced to the Cartan case.

\noindent - if $G(\ell)$ is contained in a Borel subgroup, then the same is true for $\operatorname{Sat}(G)^{\det=1}(\ell)$. To ease the notation we set $G_2=\operatorname{Sat}(G)^{\det=1}$. We can also assume that $\ell$ divides the order of $G(\ell)$ (hence that of $G_2(\ell)$ as well), for otherwise we are back to the (split) Cartan case. Now if $|G_2/N(G_2)| \neq 4$ we can set $H=G$; if, on the contrary, $|G_2/N(G_2)| = 4$ we consider the group morphism
\[
\begin{array}{cccccc}
\tau: & G & \to & G(\ell) & \to & \mathbb{F}_\ell^\times \\
         & g & \mapsto & [g] = \left( \begin{matrix} a & b \\ 0 & c \end{matrix} \right) & \mapsto & a/c.
\end{array}
\]

Every $g \in G$ is of the form $\lambda g_2$ for suitable $\lambda \in \mathbb{Z}_\ell^\times$ and $g_2 \in G_2$, and since $\tau(\lambda g_2)=\tau(g_2)$ we deduce $\tau(G)=\tau(G_2)$. On the other hand, when restricted to $G_2$ the function $\tau$ becomes
\[
g  \mapsto  [g] = \left( \begin{matrix} a & b \\ 0 & 1/a \end{matrix} \right)  \mapsto  a^2,
\]
and as we have already remarked $g  \mapsto  [g] = \left( \begin{matrix} a & b \\ 0 & 1/a \end{matrix} \right)  \mapsto a$ is the quotient map $G_2 \twoheadrightarrow G_2/N(G_2)$. Hence $\tau$ factors through the quotient $G_2/N(G_2)$ and we have $\left|\tau(G)\right|=\left|\tau(G_2)\right| \bigm\vert 4$. We take $H$ to be the kernel of $\tau$. Then it is clear that $[G:H]$ divides $4$, and we claim that $H$ satisfies the conditions in (i). To check this last claim, notice first that $H(\ell)$ is a subgroup of $G(\ell)$, so it is contained in a Borel subgroup. We also have $\ker \pi \subseteq H$, so $G/H \cong \frac{G/\ker \pi}{H/\ker \pi} = \frac{G(\ell)}{H(\ell)}$; in particular $[G(\ell):H(\ell)]$ divides $4$, and therefore the order of $H(\ell)$ is divisible by $\ell$. Finally, any matrix $\left( \begin{matrix} a & b \\ 0 & c \end{matrix} \right)$ in $H(\ell)$ satisfies $a/c=1$ by construction, so the intersection $\operatorname{Sat}(H(\ell)) \cap \operatorname{SL}_2(\mathbb{F}_\ell)$ consists of matrices $\left( \begin{matrix} a & b \\ 0 & c \end{matrix} \right)$ with $a=c$ and $ac=1$, so $a=c=\pm 1$. This implies that the quotient of $\operatorname{Sat}(H)^{\det=1}(\ell)$ by its $\ell$-Sylow has at most 2 elements, and since this quotient is exactly $\operatorname{Sat}(H)^{\det=1} / N\left( \operatorname{Sat}(H)^{\det=1} \right)$ the result follows. \qed

\begin{remark}
For future applications, we remark that the same proof shows that the inequality $[G:H] \leq 24$ appearing in theorem \ref{thm:Reconstruction2} (ii) can be replaced by the condition $[G:H] \bigm\vert 48$, and even by $[G:H] \bigm\vert 24$ if in addition $G$ satisfies $\det(G) \subseteq \mathbb{Z}_\ell^{\times 2}$.
\end{remark}

\section{Recovering $G$ from $L(G)$, when $\ell=2$}\label{sec:RecoveringGEven}
We now consider closed subgroups of $\operatorname{GL}_2(\mathbb{Z}_2)$, and endeavour to show results akin to those of the previous section. For $\operatorname{GL}_2(\mathbb{Z}_2)$ the statement is as follows:

\begin{theorem}\label{thm:GL2Z2}
Let $G$ be a closed subgroup of $\operatorname{GL}_2(\mathbb{Z}_2)$.
\begin{enumerate}[leftmargin=*,label=(\roman*)]
\item
Suppose that $G(4)$ is trivial and $\det(G) \equiv 1 \pmod 8$. The following implication holds for all positive integers $n$: if $L(G)$ contains $2^{n} \mathfrak{sl}_2(\mathbb{Z}_2)$, then the derived subgroup $G'$ of $G$ contains the principal congruence subgroup $\mathcal{B}_2(12n+2)$.

\item Without any assumption on $G$, the subgroup
\[
H=\ker (G \to G(4)) \cap \ker \left( G \to G(8) \stackrel{\det}{\rightarrow} (\mathbb{Z}/8\mathbb{Z})^\times \right)
\]
satisfies $[G:H] \leq 2 \cdot 96=192$ and the conditions in (i).
\end{enumerate}
\end{theorem}

Note that (ii) is immediate: the order of $\operatorname{GL}_2(\mathbb{Z}/4\mathbb{Z})$ is 96, and once we demand that $G(4)$ is trivial the determinant modulo 8 can only take two different values. As in the previous section, the core of the problem lies in understanding the subgroups of $\operatorname{SL}_2(\mathbb{Z}_2)$, so until the very last paragraph of this section the letter $G$ will denote a closed subgroup of $\operatorname{SL}_2(\mathbb{Z}_2)$. In view of the result we want to prove, we will also enforce the assumption that $G$ has trivial reduction modulo 4; indeed in this context the relevant statement is:

\begin{theorem}\label{thm:SubgroupsSL2Z2}
Let $G$ be a closed subgroup of $\operatorname{SL}_2(\mathbb{Z}_2)$ whose reduction modulo 4 is trivial, and let $s$ be an integer no less than 2. If $L(G)$ contains $2^{s}\mathfrak{sl}_2(\mathbb{Z}_2)$, then $G$ contains $\mathcal{B}_2(6s)$.
\end{theorem}

 The idea of the proof is quite simple: despite the fact there is in general no reason why $\Theta(G)$ should be a group under addition, we will show that for every pair $x,y$ of elements of $\Theta(G)$ it is possible to find an element that is reasonably close to $x+y$ and that lies again in $\Theta(G)$. The error term will turn out to be quadratic in $x$ and $y$, which is not quite good enough by itself, since a correction of this order of magnitude could still be large enough to destroy any useful information about $x+y$; the technical step needed to make the argument work is that of multiplying all the elements we have to deal with by a power of 2 large enough that the quadratic error term becomes negligible with respect to the linear part. The rest of the proof is really just careful bookkeeping of the correction terms appearing in the various addition formulas. We shall continue using the notation from the previous section:

\smallskip

\noindent \textbf{Notation.} For $x \in L:=L(G)$ we set $\pi_{ij}(x) =x_{ij}$, the coefficient in the $i$-th row and $j$-th column of the matrix representation of $x$ in $\mathfrak{sl}_2(\mathbb{Z}_2)$. The maps $\pi_{ij}$ are linear and continuous.

\smallskip

We start with a compactness lemma. Our arguments only yield (arbitrarily good) approximations of elements of $\Theta(G)$, and we need to know that this is enough to show that the matrices we are approximating actually belong to $\Theta(G)$.

\begin{lemma}\label{lemma:TraceIsVeryMuch2}
Let $G$ be a closed subgroup of $\operatorname{SL}_2(\mathbb{Z}_\ell)$, $g$ be an element of $G$, and $e \geq 2$. Suppose that $\Theta(g) \equiv 0 \pmod {2^{e}}$: then $\operatorname{tr}(g)-2$ is divisible by $2^{2e}$. Moreover $\Theta^{-1} : \Theta(G) \cap 2^{2}\mathfrak{sl}_2(\mathbb{Z}_2) \to G$ is well defined and continuous, and the intersection $\Theta(G) \cap 2^{2}\mathfrak{sl}_2(\mathbb{Z}_2)$ is compact.
\end{lemma}

\begin{proof}
Write $\Theta(g)=\left( \begin{matrix} a & b \\ c & -a \end{matrix} \right)$ and $\displaystyle g=\frac{\operatorname{tr}(g)}{2} \operatorname{Id} + \Theta(g)$. As $G$ is a subgroup of $\operatorname{SL}_2(\mathbb{Z}_2)$, we have the identity
\[
1 = \det g = \det \left( \frac{\operatorname{tr}(g)}{2} \operatorname{Id} + \Theta(g) \right) = \left(\frac{\operatorname{tr}(g)}{2} \right)^2 -a^2-bc.
\]

Furthermore $G$ (hence $g$) is trivial modulo $4$ by assumption, so an immediate calculation shows that $1=\det(g) \equiv 1+(\operatorname{tr}(g)-2) \pmod 8$. It follows that $\frac{\operatorname{tr}(g)}{2}$ is the unique solution to the equation $\lambda^2=1+a^2+bc$ that is congruent to 1 modulo 4, hence $\displaystyle \frac{\operatorname{tr}(g)}{2}=\sqrt{1+a^2+bc}=\displaystyle \sum_{j=0}^\infty \binom{1/2}{j} (a^2+bc)^j$ by lemma \ref{lemma:valuations2}. Given that $a^2+bc \equiv 0 \pmod {2^{2e}}$ and $2e > 3$, using again lemma \ref{lemma:valuations2} we find
\[
v_2\left( \operatorname{tr}(g) -2 \right)=v_2\left(2 \left( \frac{\operatorname{tr}(g)}{2} -1\right)\right) = 1+v_2\left(\sqrt{1+(a^2+bc)}-1\right) \geq 2e.
\]

The case $e=2$ of the above computation shows that every $x \in 2^{2}\mathfrak{sl}_2(\mathbb{Z}_2)$ admits exactly one inverse image in $\operatorname{SL}_2(\mathbb{Z}_2)$ that reduces to the identity modulo 4, so $\Theta: \mathcal{B}_2(2) \to 2^2\mathfrak{sl}_2(\mathbb{Z}_2)$ is a continuous bijection: we have just described the (two-sided) inverse, so we only need to check that the image of $\mathcal{B}_2(2)$ through $\Theta$ does indeed land in $2^2\mathfrak{sl}_2(\mathbb{Z}_2)$. We have to show that if $g=\left( \begin{matrix} d  & b \\ c & e \end{matrix} \right)$ is any element of $\mathcal{B}_2(2)$, then $\Theta(g)=\left( \begin{matrix} \frac{d-e}{2} & b \\ c & \frac{e-d}{2} \end{matrix} \right)$ has all its coefficients divisible by 4. This is obvious for $b$ and $c$. For the diagonal ones, note that $de -bc=1$, so $de \equiv 1 \pmod 8$ and hence $d \equiv e \pmod 8$ and $\frac{d-e}{2} \equiv 0 \pmod 4$ as required.
Observe now that $a^2+bc=\frac{1}{2} \operatorname{tr}\left(\Theta(g)^2\right)$, so we can write
\[
\Theta^{-1}(x)= x+ \sqrt{1+\frac{1}{2} \operatorname{tr}(x^2)} \cdot \operatorname{Id},
\]
which is manifestly continuous. Therefore $\Theta$ establishes a homeomorphism between $\mathcal{B}_2(2)$ and $2^2\mathfrak{sl}_2(\mathbb{Z}_2)$.

In particular, we have a well-defined and continuous map $\Theta^{-1} : \Theta(G) \cap 2^{2}\mathfrak{sl}_2(\mathbb{Z}_2) \to G$, and we finally deduce that the intersection $\Theta(G) \cap 2^{2}\mathfrak{sl}_2(\mathbb{Z}_2)=\Theta(G \cap \mathcal{B}_2(2))$ is compact, since this is true for $G \cap \mathcal{B}_2(2)$ and $\Theta$ is continuous.
\end{proof}

The core of the proof of theorem \ref{thm:SubgroupsSL2Z2} is contained in the following lemma:

\begin{lemma}\label{lemma:ApproximateAddition}
Let $e_1, e_2$ be integers not less than 2 and $x_1,x_2$ be elements of $\Theta(G)$. Suppose that $x_1 \equiv 0 \pmod {2^{e_1}}$ and $x_2 \equiv 0 \pmod {2^{e_2}}$: then $\Theta(G)$ contains an element $y$ congruent to $x_1+x_2$ modulo $2^{e_1+e_2-1}$. If, furthermore, both $x_1$ and $x_2$ are in upper-triangular form, then we can find such a $y$ having the same property.
\end{lemma}

\begin{proof}
Write $x_1=\Theta(g_1)$, $x_2=\Theta(g_2)$ and set $y=\Theta(g_1g_2)$. Applying lemma \ref{lemma:AdditionFormula} we find
\[
2\left( y - x_1-x_2\right) = [x_1,x_2]+(\operatorname{tr}(g_1)-2) x_2 + (\operatorname{tr}(g_2)-2) x_1.
\]

Consider the $2$-adic valuation of the various terms on the right. The commutator $[x_1,x_2]$ is clearly $0$ modulo $2^{e_1+e_2}$. We also have $\operatorname{tr}(g_1)-2 \equiv 0 \pmod {2^{2e_1}}$ and $\operatorname{tr}(g_2)-2 \equiv 0 \pmod {2^{2e_2}}$ by lemma \ref{lemma:TraceIsVeryMuch2}, so the last two terms are divisible respectively by $2^{2e_1+e_2}$ and $2^{e_1+2e_2}$. It follows that the right hand side of this equality is zero modulo $2^{e_1+e_2}$, and dividing by 2 we get the first statement in the lemma.

For the last claim simply note that if $x_1, x_2$ are upper-triangular then the same is true for all of the error terms, so $y=x_1+x_2 + (\mbox{triangular error terms})$ is indeed triangular.
\end{proof}

\smallskip

As a first application, we show that the image of $\Theta$ is stable under multiplication by 2 (up to units):

\begin{lemma}\label{lemma:2Theta}
Let $x \in \Theta(G)$ and $m \in \mathbb{N}$. There exists a unit $\lambda \in \mathbb{Z}_2^\times$ such that $\lambda \cdot 2^m x$  again belongs to $\Theta(G)$.
\end{lemma}

\begin{proof}
Clearly there is nothing to prove for $m=0$, so let us start with the case $m=1$. Write $x=\Theta(g)$ for a certain $g \in G$. By our assumptions on $G$, the trace of $g$ is congruent to $2$ modulo $4$, so $\lambda=\frac{\operatorname{tr}(g)}{2}$ is a unit in $\mathbb{Z}_2$. We can therefore form $\tilde{g}=\frac{1}{\lambda} g$, which certainly exists as a matrix in $\operatorname{GL}_2(\mathbb{Z}_2)$, even though it does not necessarily belong to $G$. Our choice of $\tilde{g}$ is made so as to ensure $\operatorname{tr}(\tilde{g})=2$, so the formula given in lemma \ref{lemma:AdditionFormula} (applied with $g_1=g_2=\tilde{g}$) yields
\[
2\left( \Theta \left(\tilde{g}^2 \right)-\Theta(\tilde{g})-\Theta(\tilde{g}) \right)=[\Theta(\tilde{g}),\Theta(\tilde{g})]+\left( \operatorname{tr}(\tilde{g})-2\right) \Theta(\tilde{g}) + \left( \operatorname{tr}(\tilde{g})-2\right) \Theta(\tilde{g}),
\]
where the right hand side vanishes. We deduce $\Theta(\tilde{g}^2)=2\Theta(\tilde{g})$, and it is now immediate to check that $\Theta(g^2)= \lambda \cdot 2 \Theta(g)$, whence the claim for $m=1$. An immediate induction then proves the general case.
\end{proof}

\medskip

We now take the first step towards understanding the structure of $\Theta(G)$, namely showing that a suitable basis of $L$ can be found inside $\Theta(G)$. Note that $L$, being open, is automatically of rank $3$.

\begin{lemma}\label{lemma:BasisInThetaG}
There exist a basis $\left\{x_1, x_2,x_3\right\} \subseteq \Theta(G)$ of $L$ and scalars $\tilde{\sigma}_{21}$, $\tilde{\sigma}_{31}$, $\tilde{\sigma}_{32} \in \mathbb{Z}_2$ with the following properties: $\pi_{21}(x_2-\tilde{\sigma}_{21} x_1)=0$, $\pi_{21}(x_3-\tilde{\sigma}_{31} x_1)=0$ and
\[
\pi_{21}( x_3-\tilde{\sigma}_{31} x_1 - \tilde{\sigma}_{32}(x_2 - \tilde{\sigma}_{21} x_1) )=\pi_{11}( x_3-\tilde{\sigma}_{31} x_1 - \tilde{\sigma}_{32}(x_2 - \tilde{\sigma}_{21} x_1) )=0.
\]
\end{lemma}

\begin{remark}
The slightly awkward equations appearing in the statement of this lemma actually have a simple interpretation: they mean it is possible to subtract a suitable multiple of $x_1$ from $x_2$ and $x_3$ so as to make them upper-triangular, and that it is then further possible to subtract one of the matrices thus obtained from the other so as to leave it with only one nonzero coefficient (in the top right corner).
\end{remark}

\begin{proof}
This is immediate from lemma \ref{lemma:DVRs}, which can be applied identifying $\mathfrak{sl}_2(\mathbb{Z}_2) \cong \mathbb{Z}_2^{3}$ via $\left( \begin{matrix} a& b \\ c & -a \end{matrix} \right) \mapsto (c,a,b)$. Note that with this identification the three canonical projections $\mathbb{Z}_2^{3} \to \mathbb{Z}_2$ become $\pi_{21}, \pi_{11}$ and $\pi_{12}$ respectively, and the vanishing conditions in the statement become exactly those of lemma \ref{lemma:DVRs}.
\end{proof}

\smallskip

As previously mentioned, in order to make the quadratic error terms appearing in lemma \ref{lemma:ApproximateAddition} negligible we need to work with matrices that are highly divisible by 2:

\begin{lemma}\label{lemma:BasisMultipliedBy23s}
Let $x_1,x_2,x_3$ be a basis of $L$.
There exist elements $y_1,y_2,y_3 \in \Theta(G)$ and units $\lambda_1, \lambda_2,\lambda_3 \in \mathbb{Z}_2^\times$ such that $y_i=\lambda_i \cdot 2^{4s}x_i$ for $i=1,2,3$; in particular $y_1,y_2,y_3$ are zero modulo $2^{4s}$, and the module generated by $y_1,y_2,y_3$ over $\mathbb{Z}_2$ contains $2^{5s} \mathfrak{sl}_2(\mathbb{Z}_2)$.
\end{lemma}
\begin{proof}
Everything is obvious (by lemma \ref{lemma:2Theta}) except perhaps the last statement. Note that $y_1,y_2,y_3$ differ from $2^{4s}x_1, 2^{4s}x_2, 2^{4s}x_3$ only by multiplication by units, so these two sets generate over $\mathbb{Z}_2$ the same module $M$. But the $x_i$ generate $L \supseteq 2^s \mathfrak{sl}_2(\mathbb{Z}_2)$, hence $M=2^{4s}L$ contains $2^{5s} \mathfrak{sl}_2(\mathbb{Z}_2)$.
\end{proof}

\noindent\textbf{Notation.} Let $x_1,x_2,x_3$ be a basis of $L$ as in lemma \ref{lemma:BasisInThetaG}, and let $y_1,y_2,y_3$ be the elements given by lemma \ref{lemma:BasisMultipliedBy23s} when applied to $x_1,x_2,x_3$. The properties of the $x_i$ become corresponding properties of the $y_i$:
\begin{itemize}
\item There is a scalar $\sigma_{21} \in \mathbb{Z}_2$ such that
\[
y_2 - \sigma_{21} \cdot y_1 = \left( \begin{matrix} b_{11} & b_{12} \\ 0 & -b_{11} \end{matrix} \right) \in \mathfrak{sl}_2(\mathbb{Z}_2);
\]
\item there are scalars $\sigma_{31}, \sigma_{32}$ such that
\[
y_3-\sigma_{31}y_1= \left( \begin{matrix} d_{11} & d_{12} \\ 0 & -d_{11} \end{matrix} \right) \in \mathfrak{sl}_2(\mathbb{Z}_2),
\]
\[
y_3-\sigma_{31}y_1-\sigma_{32} (y_2 - \sigma_{21} \cdot y_1)=\left( \begin{matrix} 0 & c_{12} \\ 0 & 0 \end{matrix} \right) \in \mathfrak{sl}_2(\mathbb{Z}_2).
\]
\end{itemize}

To ease the notation a little we set
\[
t_1=y_1 =\left( \begin{matrix} a_{11} & a_{12} \\ a_{21} & -a_{11} \end{matrix} \right), t_2=\left( \begin{matrix} b_{11} & b_{12} \\ 0 & -b_{11} \end{matrix} \right) \text{ and } t_3=\left( \begin{matrix} 0 & c_{12} \\ 0 & 0 \end{matrix} \right).
\]

It is clear that $\left\{t_1,t_2,t_3\right\}$ and $\left\{y_1,y_2,y_3\right\}$ generate the same module $M$ over $\mathbb{Z}_2$, so in particular $M$ contains $2^{5s}\mathfrak{sl}_2(\mathbb{Z}_2)$.

\begin{lemma}\label{lemma:ValuationsDoNotExplode}
The 2-adic valuations of $a_{21}, b_{11}$ and $c_{12}$ do not exceed $5s$.
\end{lemma}
\begin{proof}
We can express $\left( \begin{matrix} 0 & 0 \\ 2^{5s} & 0 \end{matrix} \right)$ as a $\mathbb{Z}_2$-linear combination of $t_1,t_2,t_3$,
\[
\left( \begin{matrix} 0 & 0 \\ 2^{5s} & 0 \end{matrix} \right) = \lambda_1 t_1 + \lambda_2 t_2 + \lambda_3 t_3,
\]
for a suitable choice of $\lambda_1,\lambda_2,\lambda_3$ in $\mathbb{Z}_2$. Comparing the bottom-left coefficient we find $\lambda_1a_{21}=2^{5s}$, so $v_2(a_{21}) \leq 5s$ as claimed.

The same argument, applied to the representation of $\left( \begin{matrix} 2^{5s} & 0 \\ 0 & -2^{5s} \end{matrix} \right)$ (resp. $\left( \begin{matrix} 0 & 2^{5s} \\ 0 & 0 \end{matrix} \right)$) as a combination of $t_1,t_2,t_3$, gives $b_{11} | 2^{5s}$ (resp. $c_{12} | 2^{5s}$) and finishes the proof of the lemma.
\end{proof}

\smallskip

For future reference, and since it is easy to lose track of all the notation, we record here two facts we will need later:

\begin{remark}\label{rmk_ValueOfSigma2}
We have $\sigma_{32}=\displaystyle \frac{d_{11}}{b_{11}}$ and $\displaystyle v_2(d_{12}-\sigma_{32} b_{12}) =v_2(c_{12}) \leq 5s$.
\end{remark}

\smallskip

We now further our investigation of the approximate additive structure of $\Theta(G)$. Since essentially all of the arguments are based on sequences of approximations the following notation will turn out to be very useful.

\medskip

\noindent\textbf{Notation.} We write $a=b+O\left(2^{n}\right)$ if $a \equiv b \pmod{2^n}$.

\smallskip

\begin{lemma}\label{lemma:PowerfulApproximateAddition}
Let $a_1, a_2 \in \Theta(G) \cap 2^{4s}\mathfrak{sl}_2(\mathbb{Z}_2)$ and $\xi \in \mathbb{Z}_2$. Then $\Theta(G)$ contains an element $z$ congruent to $a_1-\xi a_2$ modulo $2^{8s-1}$. If moreover $a_1, a_2$ are upper triangular then $z$ can be chosen to have the same property.
\end{lemma}

\begin{proof}
We construct a sequence $\left(z_n\right)_{n \geq 0}$ of elements of $\Theta(G)$ and a sequence $\left(\xi_n\right)_{n \geq 0}$ of elements of $\mathbb{Z}_2$ satisfying $\xi_n=\xi + O(2^n)$ and
\[
z_n = a_1 - \xi_n a_2 + O\left(2^{8s-1}\right).
\]
We can take $z_0=a_1$ and $\xi_0=0$. Given $z_n, \xi_n$ we proceed as follows. If we let $w_n=v_2(\xi_n-\xi)$, then $w_n \geq n$ by the induction hypothesis, and by lemma \ref{lemma:2Theta} we can find a unit $\lambda_n$ such that $2^{w_n} \lambda_n a_2$ also belongs to $\Theta(G)$. Note that both $z_n$ and $2^{w_n} \lambda_n a_2$ are zero modulo $2^{4s}$. Apply lemma \ref{lemma:ApproximateAddition} to $(x_1,x_2)=(z_n, 2^{w_n} \lambda_n a_2)$: it yields the existence of an element $z_{n+1}$ of $\Theta(G)$ of the form $z_n+2^{w_n} \lambda_n a_2 + O\left(2^{8s-1}\right)$. We take $\xi_{n+1}=(\xi_n - 2^{w_n} \lambda_n)$; let us check that $\xi_{n+1}, z_{n+1}$ have the right properties. Clearly
\[
z_{n+1} =  z_n+2^{w_n} \lambda_n a_2 + O\left(2^{8s-1}\right) = a_1 - (\xi_n - 2^{w_n} \lambda_n) a_2 + O\left(2^{8s-1}\right).
\]

On the other hand the definition of $w_n$ implies that $\xi_n-\xi = 2^{w_n} \cdot \mu_n$ where $\mu_n$ is a unit, so
\[
\begin{aligned}
v_2 \left( \xi_{n+1}- \xi \right) & =  v_2 \left( (\xi_n - 2^{w_n} \lambda_n)- \xi \right)\\
                                  &  = v_2 ( 2^{w_n} \cdot \mu_n - 2^{w_n} \cdot \lambda_n ) \\
																	& =w_n + v_2(\mu_n-\lambda_n) \geq w_n +1 \geq n+1,
\end{aligned}
\]
since $\mu_n,\lambda_n$ are both units and therefore odd. To conclude the proof it is simply enough to take $z=z_{8s-1}$: indeed 
\[
\begin{aligned}
a_1-\xi a_2 - z_{8s-1} & = a_1-\xi a_2 - \left(a_1-\xi_{8s-1}a_2 + O\left(2^{8s-1}\right)\right) \\
                       & =(\xi_{8s-1}-\xi)a_2 + O\left(2^{8s-1}\right) \\
											 & = O\left(2^{8s-1}\right)
\end{aligned}
\]
as required. The proof in the upper-triangular case goes through completely unchanged, simply using the corresponding second part of lemma \ref{lemma:ApproximateAddition}.
\end{proof}

\smallskip

The above lemma is still not sufficient, since it cannot guarantee that we will ever find a matrix with a coefficient that vanishes exactly. This last remaining obstacle is overcome through the following result:

\begin{lemma}\label{lemma:ExactAddition}
Let $a_1, a_2 \in \Theta(G) \cap 2^{4s}\mathfrak{sl}_2(\mathbb{Z}_2)$ and $\xi \in \mathbb{Z}_2$.
Suppose that for a certain pair $(i,j)$ the $(i,j)$-th coefficient of $a_1-\xi a_2$ vanishes while $v_2 \circ \pi_{ij}(a_2) \leq 5s$: then $\Theta(G)$ contains an element $z$ whose $(i,j)$-th coefficient is zero and that is congruent to $a_1 -\xi a_2$ modulo $2^{7s-1}$. If, furthermore, $a_1,a_2$ are upper-triangular, then this $z$ can be chosen to be upper-triangular as well (while still satisfying $\pi_{ij}(z)=0$).
\end{lemma}

\begin{proof}
Let $z_0$ be the element whose existence is guaranteed by lemma \ref{lemma:PowerfulApproximateAddition} when applied to $a_1,a_2,\xi$. We propose to build a sequence $(z_n)_{n \geq 0}$ of elements of $\Theta(G)$ satisfying the following conditions:
\begin{enumerate}
\item 
$z_{n+1} \equiv z_{n} \pmod {2^{7s-1}}$, and therefore $z_n \equiv z_0 \equiv 0 \pmod {2^{4s}}$;
\item
the sequence $w_n=v_2 \circ \pi_{ij}(z_n)$ is monotonically strictly increasing; in particular we have $w_n \geq w_0 \geq 8s-1$.
\end{enumerate}

Suppose we have constructed $z_n, w_n$ and let $k=v_2 \circ \pi_{ij}(a_2) \leq 5s$. By lemma \ref{lemma:2Theta} we can find a unit $\lambda$ such that $2^{w_n-k} \lambda a_2$ also belongs to $\Theta(G)$ (note that $w_n \geq 8s-1 \geq 5s \geq k$). We know that $z_n \equiv 0 \pmod {2^{4s}}$ and $2^{w_n-k} \lambda a_2 \equiv 0 \pmod {2^{w_n-k+4s}}$ (note that $a_2 \equiv 0 \pmod {2^{4s}}$). Apply lemma \ref{lemma:ApproximateAddition} to $(x_1,x_2)=(z_n, 2^{w_n-k} \lambda a_2)$: it yields the existence of an element $z_{n+1}$ of $\Theta(G)$ that is congruent to $z_n+2^{w_n-k} \lambda a_2$ modulo $2^{(4s+w_n-k) +4s -1 }$.

We can write $\pi_{ij}(z_n)=2^{w_n} \mu_n$ and $\pi_{ij}(a_2)=2^k \xi$ with $\mu_n, \xi \in \mathbb{Z}_2^\times$, so 
\[
v_2 \circ \pi_{ij}(z_n+2^{w_n-k} \lambda a_2) = v_2 ( 2^{w_n} \mu_n + 2^{w_n-k} 2^k \cdot \xi\lambda)= w_n + v_2(\mu_n+\xi\lambda),
\]
and since $\mu_n, \xi$ and $\lambda$ are all odd the last term is at least $w_n+1$. As $k$ is at most $5s$ by hypothesis we deduce
\[
\begin{aligned}
w_{n+1} & = v_2 \circ \pi_{ij}(z_{n+1}) \\
        & =v_2 \circ \pi_{ij}\left(z_n+2^{w_n-k} \lambda a_2 + O \left(2^{(4s+w_n-k) +4s -1 } \right) \right) \\
        & \geq \min\left\{ v_2 \circ \pi_{ij}\left(z_n+2^{w_n-k}\lambda a_2 \right), 8s-1+w_n-k\right\} \\
				& > w_n.
\end{aligned}
\]

As $2^{w_n-k} \lambda a_2 \equiv 0 \pmod {2^{w_n-k+4s}}$, the difference $z_{n+1}-z_n$ is zero modulo $2^{w_n-s}$, hence a fortiori modulo $2^{7s-1}$ since $w_n \geq w_0 \geq 8s-1$.

Lemma \ref{lemma:TraceIsVeryMuch2} says that $\Theta(G) \cap 2^2\mathfrak{sl}_2(\mathbb{Z}_2)$ is compact, so $z_n$ admits a subsequence converging to a certain $z \in \Theta(G)$. By continuity of $\pi_{ij}$ it is immediate to check that $\pi_{ij}(z)=0$, and since every $z_n$ is congruent modulo $2^{7s-1}$ to $z_0$ the same is true for $z$. Given that $z_0$ is congruent to $a_1-\xi a_2$ modulo $2^{8s-1}$, the last assertion follows.

Finally, the upper-triangular case is immediate, since it is clear from the construction that if $a_1, a_2$ are upper-triangular then the same is true for all the approximations $z_n$.
\end{proof}

\smallskip

The result we were really aiming for follows at once:

\begin{proposition}\label{prop_TopRightGenerator}
Let $G$ be a closed subgroup of $\operatorname{SL}_2(\mathbb{Z}_2)$ whose reduction modulo 2 is trivial, and let $s$ be an integer no less than 2. If $L(G)$ contains $2^s\mathfrak{sl}_2(\mathbb{Z}_2)$, then $\Theta(G)$ contains both an element of the form $\left( \begin{matrix} 0 & \tilde{c}_{12} \\ 0 & 0 \end{matrix} \right)$, where $v_2(\tilde{c}_{12}) \leq 5s$, and one of the form $\left( \begin{matrix} f_{11} & 0 \\ 0 & -f_{11} \end{matrix} \right)$, where $v_2(f_{11}) \leq 6s$.
\end{proposition}

\begin{proof}
We apply lemma \ref{lemma:ExactAddition} to $a_1=y_2$, $a_2=y_1$, $\xi=\sigma_{21}$, $(i,j)=(2,1)$; the hypotheses are satisfied since $y_1 \equiv y_2 \equiv 0 \pmod{2^{4s}}$ and $v_2 \circ \pi_{21}(y_1) \leq 5s$ by lemma \ref{lemma:ValuationsDoNotExplode}. It follows that $\Theta(G)$ contains a matrix $\tilde{b}$ of the form $\left( \begin{matrix} \tilde{b}_{11} & \tilde{b}_{12} \\ 0 & -\tilde{b}_{11} \end{matrix} \right)$, where we have $\tilde{b}_{ij} \equiv b_{ij} \pmod{2^{7s-1}}$ for every $1 \leq i,j \leq 2$; in particular, $v_2(\tilde{b}_{11}) \leq 5s$.

The same lemma, applied to $a_1=y_3, a_2=y_1$ and $\xi=\sigma_{31}$, implies that $\Theta(G)$ contains a matrix $\tilde{d}$ of the form $\left( \begin{matrix} \tilde{d}_{11} & \tilde{d}_{12} \\ 0 & -\tilde{d}_{11} \end{matrix} \right)$, where for every $i,j$ we have $\tilde{d}_{ij} \equiv d_{ij} \pmod{2^{7s-1}}$; in particular,
\[
v_2(\tilde{d}_{11}) \geq \min \left\{7s-1,v_2(d_{11})\right\} \geq v_2(b_{11})=v_2(\tilde{b}_{11}).
\]

Now since $v_2(\tilde{d}_{11}) \geq v_2(\tilde{b}_{11})$ we can find a scalar $\zeta$ such that
\[
\tilde{d}-\zeta \tilde{b} = \left( \begin{matrix} \tilde{d}_{11} & \tilde{d}_{12} \\ 0 & -\tilde{d}_{11} \end{matrix} \right)- \zeta \left( \begin{matrix} \tilde{b}_{11} & \tilde{b}_{12} \\ 0 & -\tilde{b}_{11} \end{matrix} \right) = \left( \begin{matrix} 0 & e_{12} \\ 0 & 0 \end{matrix} \right),
\]
so applying once again lemma \ref{lemma:ExactAddition} (more precisely, the version for triangular matrices) we find that $\Theta(G)$ contains a certain matrix $\displaystyle \tilde{e}=\left( \begin{matrix} 0 & \tilde{e}_{12} \\ 0 & 0 \end{matrix} \right)$, where $\tilde{e}_{12} \equiv e_{12} \pmod{2^{7s-1}}$. Observe now that
\[
\zeta=\frac{\tilde{d}_{11}}{\tilde{b}_{11}} = \frac{d_{11} + O\left(2^{7s-1} \right)}{b_{11} + O\left(2^{7s-1} \right)} = \frac{d_{11}}{b_{11}} + O\left( 2^{7s-1-v_2(b_{11})} \right)=\frac{d_{11}}{b_{11}} + O\left( 2^{2s-1} \right),
\]
so upon multiplying by $\tilde{b}_{12}$, which is divisible by $2^{4s}$, we obtain the congruence $\displaystyle \zeta \tilde{b}_{12} \equiv \frac{d_{11}}{b_{11}} \tilde{b}_{12} \pmod{2^{6s-1}}$. Since furthermore $\tilde{b}_{12} \equiv b_{12} \pmod{2^{6s-1}}$ we deduce $\displaystyle\zeta \tilde{b}_{12} \equiv \frac{d_{11}}{b_{11}}b_{12}  \pmod{2^{6s-1}}$. But then the inequality $v_2 \left( c_{12} \right) \leq 5s$ (cf. remark \ref{rmk_ValueOfSigma2}) implies
\[
\begin{aligned}
v_2 (\tilde{e}_{12}) & = v_2 \left(e_{12} + O\left( 2^{7s-1} \right) \right) \\
										 & = v_2 \left( \tilde{d}_{12} - \zeta \tilde{b}_{12} + O\left( 2^{7s-1} \right) \right) \\
										 & = v_2 \left( d_{12} - \frac{d_{11}}{b_{11}}b_{12}  + O\left( 2^{6s-1} \right) \right) \\
										 & = v_2 \left( c_{12} + O\left( 2^{6s-1} \right) \right) \\
										 & \leq 5s.
\end{aligned}
\]

The existence of the diagonal element is now almost immediate: indeed, we can apply once more lemma \ref{lemma:ExactAddition} to the difference
\[
2^s \left( \begin{matrix} \tilde{b}_{11} & \tilde{b}_{12} \\ 0 & -\tilde{b}_{11} \end{matrix} \right) - \frac{2^s \tilde{b}_{12}}{\tilde{e}_{12}}  \left( \begin{matrix} 0 & \tilde{e}_{12} \\ 0 & 0 \end{matrix} \right) = \left( \begin{matrix} \tilde{b}_{11} & 0 \\ 0 & -\tilde{b}_{11} \end{matrix} \right),
\]
the hypotheses being satisfied since clearly $2^s \tilde{b} \equiv 0 \pmod {2^{5s}}$ and $v_2(\tilde{e}_{12}) \leq 5s$ for what we have just seen. It follows that $\Theta(G)$ contains a matrix $\left( \begin{matrix} f_{11} & 0 \\ 0 & -f_{11} \end{matrix} \right)$ congruent to $2^s\left( \begin{matrix} \tilde{b}_{11} & 0 \\ 0 & -\tilde{b}_{11} \end{matrix} \right)$ modulo $2^{7s-1}$, and this is enough to deduce
\[
v_2(f_{11}) = v_2( 2^sb_{11} + O \left(2^{7s-1} \right) )=s+v_2(b_{11}) \leq 6s.
\]
\end{proof}

\smallskip

We are now ready for the proof of theorem \ref{thm:SubgroupsSL2Z2}:

\begin{proof}[Proof of theorem \ref{thm:SubgroupsSL2Z2}]
With all the preliminaries in place this is now quite easy: by proposition \ref{prop_TopRightGenerator} we know that $\Theta(G)$ contains an element of the form $\left( \begin{matrix} 0 & \tilde{c}_{12} \\ 0 & 0 \end{matrix} \right)$, where $v_2(\tilde{c}_{12}) \leq 5s$, and by the explicit description of $\Theta^{-1}$ (lemma \ref{lemma:TraceIsVeryMuch2}) this element must come from $R_{\tilde{c}_{12}}=\left( \begin{matrix} 1 & \tilde{c}_{12} \\ 0 & 1 \end{matrix} \right) \in G$. Similarly, if we let $f$ denote the diagonal element $\left( \begin{matrix} f_{11} & 0 \\ 0 & -f_{11} \end{matrix} \right)$, then
\[
\Theta^{-1} \left( f \right) = \left( \begin{matrix} f_{11} & 0 \\ 0 & -f_{11} \end{matrix} \right) + \sqrt{1+\frac{1}{2}\operatorname{tr} \left(f^2 \right)} \cdot \operatorname{Id}
\]
is an operator of the form $\diag_c=\left( \begin{matrix} 1+c & 0 \\ 0 & \frac{1}{c+1} \end{matrix} \right)$, where
\[
\begin{aligned}
v_2(c)& =v_2 \left(f_{11} + \sqrt{1+\frac{1}{2}\operatorname{tr} \left(f^2 \right)}-1\right) \\
      & =v_2 \left(f_{11} + O\left(2^{2v_2(f_{11})-1} \right)\right) \\
			& =v_2(f_{11}) \leq 6s.
\end{aligned}
\]

Observe now that replacing $G$ with $G^{t}$, the group $\left\{ g^t \bigm| g\in G \right\}$ endowed with the obvious product $g_1^t \cdot g_2^t = (g_2g_1)^t$, simply exchanges $L(G)$ for $L(G)^t$, so if $L(G)$ contains the (symmetric) set $2^s \mathfrak{sl}_2(\mathbb{Z}_2)$, then the same is true for $L(G^{t})$. Thus $G^{t}$ contains $R_{2^{5s}}$ and $G$ contains $L_{2^{5s}}$.  We have just shown that $G$ contains $L_a, R_b$ and $\diag_c$ for certain $a,b,c$ of valuation at most $6s$, so it follows from lemma \ref{lemma:LeftRightGenerators1} that $G$ contains $\mathcal{B}_2(6s)$.
\end{proof}

\begin{remark}\label{rmk:SL2Z2}
The above result should be thought of as an analogue of theorem \ref{thm:PinkSL2} for $\ell=2$, even though the present result is actually much weaker. It would of course be interesting to have a complete classification result for pro-$2$ groups purely in terms of Lie algebras, but as pointed out in \cite{MR1241950} the problem seems to be substantially harder than for $\ell \neq 2$.
\end{remark}

\smallskip

It is now easy to deduce theorem \ref{thm:GL2Z2} (i):

\begin{proof}
The proof follows closely that of theorem \ref{thm:Reconstruction2} (i): we can replace $G$ first by $H=G \cdot (1+8\mathbb{Z}_2)$ and then by $H_0=H \cap \operatorname{SL}_2(\mathbb{Z}_2)$ without altering $L(G)$ nor $G'$, so we are reduced to working with subgroups of $\operatorname{SL}_2(\mathbb{Z}_2)$. Note now that $n \geq 2$ since by hypothesis every element in $G$ (and hence in $H_0$) has its off-diagonal coefficients divisible by 4. Theorem \ref{thm:SubgroupsSL2Z2} then guarantees that $H_0$ contains $\mathcal{B}_2(6n)$, so $G'=H_0'$ contains $\mathcal{B}_2(12n+2)$ because of lemma \ref{lemma:DerivedSubgroup}.
\end{proof}

\section{Lie algebras modulo $\ell^n$}\label{sec:LieAlgebrasModelln}
Fix any prime number $\ell$ and let $L$ be a topologically open and closed, $\mathbb{Z}_\ell$-Lie subalgebra of $\mathfrak{sl}_2(\mathbb{Z}_\ell)$. The same arguments of the previous section, namely an application of lemma \ref{lemma:DVRs}, yield the existence of a basis of $L$ of the form
\[
x_1 = \left( \begin{matrix} a_{11} & a_{12} \\ a_{21} & -a_{11} \end{matrix} \right), x_2 = \left( \begin{matrix} b_{11} & b_{12} \\ 0 & -b_{11} \end{matrix} \right), x_3 = \left( \begin{matrix} 0 & c_{12} \\ 0 & 0 \end{matrix} \right).
\]

\begin{definition} A basis of this form will be called a \textit{reduced} basis.\end{definition}

There is clearly no uniqueness of such an object, but in what follows we will just assume that the choice of a reduced basis has been made.

\medskip

\noindent\textbf{Notation.} We let $k(L)$, or simply $k$, denote the number $\min_{m \in L} v_\ell (m_{21})$, where $m_{21}$ is the bottom-left coefficient of $m$ in the standard matrix representation of elements of $\mathfrak{sl}_2(\mathbb{Z}_\ell)$.
Furthermore, for every positive $n$ we denote by $L \left(\ell^n\right)$ be the image of the mod-$\ell^n$ reduction map $\pi_n : L \to \mathfrak{sl}_2 (\mathbb{Z}/\ell^n\mathbb{Z})$; clearly $L \left(\ell^n\right)$ is a Lie algebra over $\mathbb{Z}/{\ell^n}\mathbb{Z}$.

\begin{remark} It is apparent from the very definition of a reduced basis that $k(L)=v_\ell(a_{21})$. Also notice that, by definition, the images of $x_1, x_2, x_3$ in $L \left(\ell^n\right)$ generate it as a $\left(\mathbb{Z}/\ell^n\mathbb{Z}\right)$-module.
\end{remark}

The following statement allows us to deduce properties of $G(\ell^n)$ from corresponding properties of $L(\ell^n)$:

\begin{proposition}\label{prop_Main}
Suppose $L$ as above is obtained as $\overline{\Theta(G)}$ for a certain closed subgroup $G$ of $\operatorname{GL}_2(\mathbb{Z}_\ell)$ (whose reduction modulo $2$ is trivial if $\ell=2$). For every integer $m \geq 1$ let $G(\ell^m)$ be the image of $G$ in $\operatorname{GL}_2(\mathbb{Z}/\ell^m\mathbb{Z})$, and 
let $j_m=\left|\{ i \in \{1,2,3\} \bigm\vert x_i \not \equiv 0 \pmod{\ell^m} \} \right|$ (that is, exactly $j_m$ among $x_1,x_2$ and $x_3$ are nonzero modulo $\ell^m$).
For every $n \geq 1$ the following are the only possibilities (recall that $v=v_\ell(2)$):

\begin{itemize}[leftmargin=*]
\item $j_n$ is at most $1$ and $G(\ell^n)$ is abelian.

\item $j_n=2$ and either $j_{2n}=3$ or $G(\ell^{n-k(L)+1-2v})$ is contained in the subgroup of upper-triangular matrices (up to a change of coordinates in $\operatorname{GL}_2(\mathbb{Z}_\ell)$).

\item $j_n=3$ and $L$ contains $\ell^{n+2k(L)-1}\mathfrak{sl}_2(\mathbb{Z}_\ell)$.
\end{itemize}
\end{proposition}

\begin{remark}
The exponent $n+2k(L)-1$ is best possible: fix integers $k \geq 0$, $n \geq 1$ and let $L$ be the Lie algebra generated (as a $\mathbb{Z}_\ell$-module) by $x_1=\left(
\begin{array}{cc}
 1 & 0 \\
 \ell^k & -1 \\
\end{array}
\right), x_2=\left(
\begin{array}{cc}
 \ell^{k+n-1} & 0 \\
 0 & -\ell^{k+n-1} \\
\end{array}
\right)$, and $x_3=\left(
\begin{array}{cc}
 0 & \ell^{n-1} \\
 0 & 0 \\
\end{array}
\right)$. Then clearly $k(L)=k$, $j_n(L)=3$, and it is easy to check that $n+2k-1$ is the smallest exponent $s$ such that $\ell^s \mathfrak{sl}_2(\mathbb{Z}_\ell)$ is contained in $L$.
\end{remark}

\begin{proof}
Assume first $j_n \leq 1$. It is clear that every element of $G(\ell^n)$ can we written as $\lambda \operatorname{Id} + m_{n}$ for some $\lambda \in \mathbb{Z}/\ell^n\mathbb{Z}$ and $m_{n} \in L \left(\ell^n\right)$. Now $L$ is generated by $x_1,x_2,x_3$, so in turn every $m_{n}$ is of the form $\pi_n \left( \mu_1 x_1 + \mu_2 x_2 + \mu_3 x_3 \right)$, and since at most one of $\pi_n(x_1), \pi_n(x_2), \pi_n(x_3)$ is non-zero we can find an $l_{n} \in L\left(\ell^{n}\right)$ such that, for every $m_{n}$, there exists a scalar $\mu \in \mathbb{Z}/\ell^{n}\mathbb{Z}$ with $m_{n}=\mu \, l_{n}$. It follows that every element of $G(\ell^n)$ can be written as $\lambda \operatorname{Id} + \mu \, l_{n}$ for suitable $\lambda, \mu$, and since $\operatorname{Id}$ and $l_{n}$ commute our claim follows.

\medskip

Next consider the case $j_n=2$. We can safely assume that $j_{2n}=2$, for otherwise we are done (notice that $j_{2n} \geq j_n=2$). Under this assumption, it is clear that for $i=1,2,3$ we have $\pi_n(x_i)=0$ if and only if $\pi_{2n}(x_i)=0$.
Suppose first $\pi_n(x_1)=0$, so that $k(L) \geq 1$. Then $G(\ell^n)$ is a subset of
\[
\mathbb{Z}/\ell^{n}\mathbb{Z} \cdot \operatorname{Id} + \mathbb{Z}/\ell^{n}\mathbb{Z} \cdot \pi_n(x_2) + \mathbb{Z}/\ell^{n}\mathbb{Z} \cdot \pi_n(x_3),
\]
and $\operatorname{Id}, \pi_n(x_2), \pi_n(x_3)$ are upper-triangular matrices, so $G(\ell^n)$ -- hence also $G(\ell^{n-k(L)+1-2v})$, since $k(L) \geq 1$ -- is in triangular form.

Suppose next $\pi_n(x_1) \neq 0$. Assume that $\pi_n(x_3)=0$ (the other case being analogous, as we are only going to use that $x_2$ is upper triangular). $L$ is a Lie algebra, hence so is $L\left(\ell^{2n}\right)$; furthermore, every element in $L\left(\ell^{2n}\right)$ is a combination of $\pi_{2n}(x_1), \pi_{2n}(x_2)$ with coefficients in $\mathbb{Z}/\ell^{2n}\mathbb{Z}$. In particular, there exist $\xi_1, \xi_2 \in \mathbb{Z}/\ell^{2n}\mathbb{Z}$ such that
\[
\begin{aligned}
\left[x_1,x_2\right]- 2 b_{11} x_1 + 2a_{11}x_2 & = \left( \begin{matrix} -a_{21}b_{12} & 4(a_{11}b_{12} -a_{12}b_{11}) \\ 0 & a_{21}b_{12} \end{matrix} \right) \\
                                     & \equiv \xi_1 x_1 + \xi_2 x_2 \pmod {\ell^{2n}}.
\end{aligned}
\]

Matching the bottom-left coefficients we find $\xi_1 a_{21} \equiv 0 \pmod{\ell^{2n}}$, so, using $v_\ell(a_{21})=k(L)$, we immediately deduce $\xi_1 \equiv 0 \pmod{\ell^{2n-k(L)}}$. Reducing the above congruence modulo $\ell^{2n-k(L)}$ we then have the relations
\begin{equation}\label{eq_Congruences}
\begin{cases}
-a_{21}b_{12} \equiv \xi_2 b_{11} \pmod{\ell^{2n-k(L)}} \\
4(a_{11}b_{12}-a_{12}b_{11}) \equiv \xi_2 b_{12} \pmod{\ell^{2n-k(L)}}.
\end{cases}
\end{equation}

We now introduce the vector $y=\left( \begin{matrix} b_{12} \\ -2b_{11} \end{matrix} \right) \in \mathbb{Z}_\ell^2$. An immediate calculation shows that this is an exact eigenvector for $x_2$ (associated with the eigenvalue $-b_{11}$), and on the other hand it is also an approximate eigenvector for $2x_1$, in the sense that $2x_1 \cdot y  \equiv \left(\xi_2-2a_{11} \right) y \pmod{\ell^{2n-k(L)}}$. Indeed,
\[
2x_1 \cdot y = \left(\begin{matrix} a_{11} & a_{12} \\ a_{21} & -a_{11} \end{matrix}\right)\left( \begin{matrix} 2b_{12} \\ -4b_{11} \end{matrix}  \right) = \left( \begin{matrix} 2a_{11}b_{12} -4a_{12}b_{11} \\ 2a_{21}b_{12}+4a_{11}b_{11} \end{matrix}  \right),
\]
and using \eqref{eq_Congruences} we find
\[
\begin{aligned}
2x_1 \cdot y & = \left( \begin{matrix} 2a_{11}b_{12} -4a_{12}b_{11} \\ 2a_{21}b_{12}+4a_{11}b_{11} \end{matrix}  \right) \\
             & \equiv \left( \begin{matrix} 2a_{11}b_{12} +\xi_2 b_{12}-4a_{11}b_{12} \\ -2\xi_2b_{11}+4a_{11}b_{11} \end{matrix}  \right) \\
						 & \equiv (\xi_2-2a_{11}) y \pmod {\ell^{2n-k(L)}}
\end{aligned}
\]
as claimed.

Now if $\ell \neq 2$ we immediately deduce $x_1 \cdot y \equiv \left(\frac{\xi_2}{2}-a_{11}\right) y \pmod {\ell^{2n-k(L)}}$. If, on the other hand, $\ell=2$, then we would like to prove that $v_2(\xi_2) \geq 1$ in order to be able to divide by 2. Observe that $y$ is not zero modulo $2^{n+1}$, since its coordinates are (up to a factor of 2) the entries of $x_2$, which we have assumed not to reduce to zero in $L \left(2^{n}\right)$.

Let $\alpha=\min \left\{v_2(2b_{11}),v_2(b_{21})\right\} \leq n$ and reduce the last congruence modulo $2^{\alpha+1}$. Then $2x_1 \cdot y \equiv x_1 \cdot (2y) \equiv 0 \pmod {2^{\alpha+1}}$, so $\left(\xi_2 - 2a_{11} \right) y \equiv 0 \pmod {2^{\alpha+1}}$, which implies that $\xi_2$ is even (that is to say, $v_2(\xi_2) \geq 1$), for otherwise multiplying by $\lambda-2a_{11}$ would be invertible modulo $2^{\alpha+1}$ and we would find $y \equiv 0 \pmod {2^{\alpha+1}}$, against the definition of $\alpha$. It follows that we can indeed divide the above congruence by $2$ to get
\[
x_1 \cdot y \equiv \left(\frac{\xi_2}{2} - a_{11} \right) y \pmod{2^{2n-k(L)-1}}.
\]

Equivalently, the following congruence holds for \textit{every} prime $\ell$:
\[
x_1 \cdot y \equiv \left(\frac{\xi_2}{2}-a_{11}\right) y \pmod {\ell^{2n-k(L)-v}}.
\]

Note now that it is in fact true for \textit{every} $\ell$ that $y$ is not zero modulo $\ell^{n+v}$ (its coordinates are, up to a factor of 2, the entries of $x_2$, which we have assumed not to reduce to zero modulo $\ell^{n}$).

Let again $\alpha=\min \left\{v_\ell(2b_{11}),v_\ell(b_{21})\right\} \leq n-1+v$ and set $\tilde{y}=\ell^{-\alpha}y$. Dividing by $\ell^\alpha$ the congruence $x_1 \cdot y \equiv \left(\frac{\xi_2}{2} - a_{11} \right) y \pmod{\ell^{2n-k(L)-v}}$ we get $x_1 \cdot \tilde{y} \equiv \left(\frac{\xi_2}{2} - a_{11} \right) \tilde{y} \pmod{\ell^{n-k(L)+1-2v}}$, where $\tilde{y}=\left( \begin{matrix} \tilde{y}_1 \\ \tilde{y}_2 \end{matrix} \right)$ is a vector at least one of whose coordinates is an $\ell$-adic unit. Assume by symmetry that $v_\ell(\tilde{y}_1)=0$ and introduce the base-change matrix $P=\left( \begin{matrix} \tilde{y}_1 & 0 \\ \tilde{y}_2 & 1 \end{matrix} \right)$: this is then an element of $\operatorname{GL}_2(\mathbb{Z}_\ell)$, since its determinant $\tilde{y}_1$ is not divisible by $\ell$.

An element of $G(\ell^{n-k(L)+1-2v})$ will be of the form $g=\lambda \operatorname{Id} + \mu_1 x_1 + \mu_2 x_2$, so by construction conjugating $G$ via $P$ puts $G(\ell^{n-k(L)+1-2v})$ in upper-triangular form. Indeed, the first column of $x_i$ (for $i=1,2$) in the coordinates defined by $P$ is given by
\[
\begin{aligned}
P^{-1} x_i P \left( \begin{matrix} 1 \\ 0 \end{matrix} \right) & = P^{-1} x_i \cdot \tilde{y} = P^{-1} \left((\xi_2/2-a_{11})\tilde{y} + \ell^{n-k(L)+1-2v}w\right)	\\
                                                               & =(\xi_2/2-a_{11}) \left( \begin{matrix} 1 \\ 0 \end{matrix} \right) + \ell^{n-k(L)+1-2v} P^{-1} w \\
																															& \equiv (\xi_2/2-a_{11}) \left( \begin{matrix} 1 \\ 0 \end{matrix} \right) \pmod {\ell^{n-k(L)+1-2v}}
\end{aligned}
\]
where $w$ is a suitable vector in $\mathbb{Z}_\ell^2$ (that vanishes for $i=2$).

\medskip

Finally, suppose $j_n=3$. Then we have in particular $\pi_n(x_3) \neq 0$, so $v_\ell(c_{12}) \leq n-1$. As $L$ is a Lie algebra, we see that it contains
\[
x_4=[x_1,x_3]-2a_{11}x_3= \left(\begin{matrix} -a_{21} c_{12} & 0 \\ 0 & a_{21} c_{12} \end{matrix} \right),
\]
whose diagonal entries have valuation at most $v_\ell(a_{21})+v_\ell(c_{12}) \leq k(L)+(n-1)$. Furthermore, $L$ also contains the linear combination
\[
x_5 = \ell^{n+k(L)-1} x_1  + \frac{\ell^{n+k(L)-1}a_{11}}{a_{21} c_{12}}x_4 - \frac{\ell^{n+k(L)-1} a_{12}}{c_{12}}x_3 = \left(\begin{matrix} 0 & 0 \\ \ell^{n+k(L)-1}a_{21} & 0 \end{matrix} \right):
\]
notice that the coefficients $\displaystyle \frac{\ell^{n+k(L)-1}a_{11}}{a_{21} c_{12}}$ and $\displaystyle \frac{\ell^{n+k(L)-1} a_{12}}{c_{12}}$ have positive $\ell$-adic valuation by what we have already shown, and that the valuation of the only non-zero coefficient of $x_5$ is $n+2k(L)-1$.
Setting
\[
s_1=\left( \begin{matrix} 0 & 1 \\ 0 & 0 \end{matrix} \right), s_2=\left( \begin{matrix} 1 & 0 \\ 0 & -1 \end{matrix} \right), s_3=\left( \begin{matrix} 0 & 0 \\ 1 & 0 \end{matrix} \right)
\]
we see that $L$ contains the three elements $x_3=c_{12} s_1$, $x_4=-a_{21}c_{12} s_2$, $x_5=\ell^{n+k(L)-1} a_{21} s_3$. By what we have already proved we have
\[
\max\left\{v_\ell(c_{12}), v_\ell(-a_{21}c_{12}), v_\ell\left( \ell^{n+k(L)-1} a_{21}\right) \right\} = n+2k(L)-1,
\]
so the $\mathbb{Z}_\ell$-module generated by $x_3, x_4, x_5$ contains $\ell^{n+2k(L)-1}\mathfrak{sl}_2(\mathbb{Z}_\ell)$, and a fortiori so does $L$.\end{proof}

\begin{corollary}\label{cor:Possibilities}
Let $G$ be a closed subgroup of $\operatorname{GL}_2(\mathbb{Z}_\ell)$ satisfying property $(\star \star)$ of theorem \ref{thm:Reconstruction2} (resp. $G(4)=\left\{\operatorname{Id}\right\}$ and $\det(G) \equiv 1 \pmod {8}$ if $\ell=2$). Then for every positive integer $n \geq k(L(G))$ at least one of the following holds:

\begin{enumerate}
\item $G(\ell^n)$ is abelian.
\item $G(\ell^{n-k(L(G))+1-2v})$ is contained in the subgroup of upper-triangular matrices (up to a change of coordinates in $\operatorname{GL}_2(\mathbb{Z}_\ell)$).
\item $G'$ contains the principal congruence subgroup
\[
\mathcal{B}_\ell(16n-4)=\left(\operatorname{Id}+\ell^{16n-4}\mathfrak{gl}_2(\mathbb{Z}_\ell)\right) \cap \operatorname{SL}_2(\mathbb{Z}_\ell),
\]
if $\ell$ is odd, and it contains $\mathcal{B}_2(48n-10)$, if $\ell=2$.
\end{enumerate}
\end{corollary}

\begin{proof}
To ease the notation set $L=L(G)$. Consider $L \left(\ell^n\right)$ and distinguish cases depending on $j_n$ as in the statement of the previous proposition. If $j_n \leq 1$ we are in case (1) and we are done. If $j_n \geq 2$ we begin by proving that either (2) holds or $L$ contains $\ell^{4n-1} \mathfrak{sl}_2(\mathbb{Z}_\ell)$.

If $j_n=2$ and $j_{2n}=2$, then we are in situation (2) by the previous proposition. If, on the other hand, $j_n=2$ and $j_{2n}=3$, then (again by  proposition \ref{prop_Main}) we have
\[
L \supseteq \ell^{2n+2k(L)-1}\mathfrak{sl}_2(\mathbb{Z}_\ell) \supseteq \ell^{4n-1}\mathfrak{sl}_2(\mathbb{Z}_\ell)
\]
since $n \geq k(L)$. Finally, for $j_n=3$ the proposition yields directly
\[
L \supseteq \ell^{n+2k(L)-1}\mathfrak{sl}_2(\mathbb{Z}_\ell) \supseteq \ell^{3n-1}\mathfrak{sl}_2(\mathbb{Z}_\ell).
\]

\smallskip

In all cases, property $(\star \star)$ (resp.~theorem \ref{thm:GL2Z2} (i) for $\ell=2$) now implies that $G'$ contains $\mathcal{B}_\ell(16n-4)$ (resp. $\mathcal{B}_2(48n-10)$) as claimed.
\end{proof}

\section{Application to Galois groups}\label{sec:GaloisGroups}
We now plan to apply the above machinery to the Galois representations attached to an elliptic curve. Let therefore $K$ be a number field and $E$ an elliptic curve over $K$ without (potential) complex multiplication.

\smallskip

\noindent\textbf{Notation.} $\ell$ is any rational prime, $n$ a positive integer and $G_\ell$ the image of $\abGal{K}$ inside $\operatorname{Aut} T_\ell(E) \cong \operatorname{GL}_2(\mathbb{Z}_\ell)$. As before, $v$ is $0$ or $1$ according to whether $\ell$ is respectively odd or even.

\smallskip

If $\ell$ is odd (resp. $\ell=2$), then by theorem \ref{thm:Reconstruction2} (resp.~theorem \ref{thm:GL2Z2}) we know that either $G_\ell$ contains a subgroup $H_\ell$ satisfying $[G_\ell:H_\ell] \leq 24$ (respectively $[G_\ell:H_\ell] \leq 192$ for $\ell=2$) and the hypotheses of corollary \ref{cor:Possibilities}, or otherwise $G_\ell'=\operatorname{SL}_2(\mathbb{Z}_\ell)$. In this second case we put $H_\ell=G_\ell$.

We also denote $K_\ell$ the extension of $K$ fixed by $H_\ell$. The degree $[K_\ell:K]$ is then bounded by 24, for odd $\ell$, and $2 \cdot |\operatorname{GL}_2(\mathbb{Z}/4\mathbb{Z})|=2 \cdot 96$, for $\ell=2$. For a fixed $\ell$, upon replacing $K$ with $K_\ell$ we are reduced to the case where $G_\ell$ satisfies the hypotheses of corollary \ref{cor:Possibilities}. In order to apply this result we want to have numerical criteria to exclude the `bad' cases (1) and (2). These numerical bounds form the subject of lemma \ref{lemma_ellTooBigThenNotBorel} and proposition \ref{prop_ellTooBigThenNotAbelian} below, whose proofs are inspired by the arguments of \cite{MR1209248} and \cite{MR1034259}.

\begin{lemma}\label{lemma_ellTooBigThenNotBorel}
If $\ell^n \nmid b_0(K,E)$ the group $G_\ell(\ell^n)$ cannot be put in triangular form.
\end{lemma}

\begin{proof}
Suppose that $G_\ell(\ell^n)$ is contained (up to a change of basis) in the group of upper-triangular matrices. The subgroup $\Gamma$ of $E[\ell^n]$ given (in the coordinates in which $G_\ell(\ell^n)$ is triangular) by
\[
\Gamma=\left\{ \left( \begin{matrix} a \\ 0  \end{matrix} \right) \bigm\vert a \in \mathbb{Z}/\ell^n\mathbb{Z} \right\}
\]
is $\abGal{K}$-stable, hence defined over $K$. Consider then $E^*=E/\Gamma$ and the natural projection $\pi:E \to E^*$ of degree $|\Gamma|=\ell^n$. By theorem \ref{cor:MassersTrick} we also have an isogeny $E^* \to E$ of degree $b$, with $b \bigm| b_0(K,E)$. Composing the two we get an endomorphism of $E$ that kills $\Gamma$, and therefore corresponds (since $\left( \begin{matrix} 1 \\ 0  \end{matrix} \right)$ is annihilated by $\ell^n$) to multiplication by a certain $\ell^nd$, $d \in \mathbb{Z}$. Taking degrees we get $\ell^n \cdot b = |\Gamma| \cdot b = d^2 \ell^{2n}$, so $\ell^{n} \bigm | b$ and $\ell^n \bigm | b_0(K,E)$.
\end{proof}

\begin{corollary}\label{cor:ValueOfk}
Let $L$ be the special Lie algebra of $G_\ell$ (supposing that $G_\ell(2)$ is trivial if $\ell=2$). The inequality $k(L) \leq v_\ell(b_0(K,E))$ holds, so that in particular $\ell^{k(L)} \bigm | b_0(K,E)$.
\end{corollary}

\begin{proof}
Let $t=v_\ell(b_0(K,E))$. If by contradiction we had $k(L) \geq t+1$, then $L\left(\ell^{t+1}\right)$ would be triangular, and therefore so would be $G_\ell(\ell^{t+1}) \subseteq \mathbb{Z}/\ell^{t+1} \mathbb{Z} \cdot \operatorname{Id} + L\left(\ell^{t+1}\right)$, which is absurd, since $\ell^{t+1} \nmid b_0(K,E)$. 
\end{proof}

\begin{corollary}
If $\ell^n \nmid b_0(K,E)$ the group $G_\ell(\ell^n)$ does not consist entirely of scalar matrices. In particular this is true for $G_\ell(\ell^{v_\ell(b_0(K,E))+1})$.
\end{corollary}

Using this last corollary we find:

\begin{proposition}\label{prop_ellTooBigThenNotAbelian}
If $\ell^{2n}$ does not divide $b_0(K,E)^{4} b_0(K,E \times E)$ the group $G_\ell(\ell^n)$ is not abelian. In particular, the group $G_\ell(\ell)$ is not abelian if $\ell$ does not divide $b_0(K,E) b_0(K,E \times E)$.
\end{proposition}

\begin{proof}
For the sake of simplicity set $d=b_0(K,E)$. By the previous corollary, there is an $\alpha \in G_\ell$ whose image modulo $\ell^{1+v_\ell(d)} $ is not a scalar matrix. Suppose now that $G_\ell(\ell^n)$ is abelian. Consider the subgroup $\Gamma=\left\{(x,\alpha(x)) \bigm| x \in E[\ell^n] \right\} \subset E \times E$; this is defined over $K$, since for any $\gamma \in G_\ell(\ell^n)$ we have $\gamma\cdot(x,\alpha(x))=(\gamma \cdot x, \gamma \cdot \alpha(x)) = (\gamma \cdot x,\alpha(\gamma \cdot x))$ as $G_\ell(\ell^n)$ is commutative. We can therefore form the quotient $K$-variety $E^*=\left(E \times E\right)/\Gamma$, which comes equipped with a natural isogeny $E \times E \twoheadrightarrow E^*$ of degree $|\Gamma|=E[\ell^n]=\ell^{2n}$; on the other hand, theorem \ref{cor:MassersTrick} yields the existence of a $K$-isogeny $E^* \to E \times E$ of degree $b \bigm| b_0(K,E \times E)$. Composing the two we end up with an endomorphism $\psi$ of $E \times E$, which (given that $E$ does not admit complex multiplication) can be represented as a $2 \times 2$ matrix $\left( \begin{matrix} e_{11} & e_{12} \\ e_{21} & e_{22} \end{matrix} \right)$ with coefficients in $\mathbb{Z}$ and nonzero determinant.

Now since $\psi$ kills $\Gamma$ we must have $e_{11} x + e_{12} \alpha(x) =0$ and $e_{21} x + e_{22} \alpha(x)=0$ for every $x \in E[\ell^n]$. Let $\eta=\min \left\{v_\ell(e_{ij}) \right\}$ and suppose by contradiction $\eta < n-v_\ell(d)$. For the sake of simplicity, let us assume this minimum is attained for $e_{12}$ (the other cases being completely analogous: the situation is manifestly symmetric in the index $i$, and to show that it is symmetric in $j$ it is enough to compose with $\alpha^{-1}$, which is again a non-scalar matrix). Dividing the equation $e_{11} x + e_{12} \alpha(x) =0$ by $\ell^\eta$ we get
\[
\frac{e_{11}}{\ell^\eta} x +  \frac{e_{12}}{\ell^\eta} \alpha(x) \equiv 0 \pmod {\ell^{n-\eta}} \quad \forall x \in E[\ell^n],
\]
whence
\[
\frac{e_{11}}{\ell^\eta} x +  \frac{e_{12}}{\ell^\eta} \alpha(x) = 0 \quad \forall x \in E[\ell^{n-\eta}],
\]
where now $\displaystyle \frac{e_{12}}{\ell^\eta}$ is invertible modulo $\ell^{n-\eta}$, being relatively prime to $\ell$. Multiplying by the inverse of $\displaystyle \frac{e_{12}}{\ell^\eta}$, then, we find that
\[
\alpha (x) = -\frac{e_{11}}{\ell^\eta} \left(\frac{e_{12}}{\ell^\eta}\right)^{-1} x \quad \forall x \in E[\ell^{n-\eta}],
\]
i.e.~$\alpha$ is a scalar modulo $\ell^{n-\eta}$. By definition of $\alpha$, this implies $\ell^{n-\eta} \bigm| d$, so $n-\eta \leq v_\ell(d)$, a contradiction. It follows that $\ell^{2n}\ell^{-2v_\ell(d)} \bigm\vert \ell^{2\eta} \bigm\vert \det \left( \begin{matrix} e_{11} & e_{12} \\ e_{21} & e_{22} \end{matrix} \right)$. Squaring this last divisibility we find
\[
\ell^{4n}\ell^{-4v_\ell(d)} \bigm\vert \left( \det \left( \begin{matrix} e_{11} & e_{12} \\ e_{21} & e_{22} \end{matrix}\right) \right) ^ {2} = \deg(\psi) = b  \ell^{2n},
\]
so
$\ell^{2n}\ell^{-4v_\ell(d)} \bigm\vert b$ and $\ell^{2n} \bigm\vert \ell^{4v_\ell(d)} b_0(K,E \times E) \bigm\vert d^4 \, b_0(K,E \times E)$.
The second assertion follows immediately from the fact that $\ell$ is prime.\end{proof}

\smallskip

With these results at hand it is now immediate to deduce the following theorem, where we use the notation introduced at the beginning of this section and the symbol $\mathcal{B}_\ell(n)$ of section \ref{sec:GroupTheory}.

\begin{theorem}\label{thm:GeneralIndexBound}
Let $\ell$ be a prime and set
$
D(\ell)=b_0(K_\ell,E)^5b_0(K_\ell,E \times E).
$
Let $n$ be a positive integer. Suppose that $\ell^{n-v}$ does not divide $D(\ell)$: then $H_\ell'$ contains $\mathcal{B}_\ell(16n-4)$, for odd $\ell$, and it contains $\mathcal{B}_2(48n-10)$, for $\ell=2$.
\end{theorem}

\begin{proof}
By the discussion at the beginning of this section there are two possibilities: if the derived subgroup $G_\ell'$ is all of $\operatorname{SL}_2(\mathbb{Z}_\ell)$ then the conclusion is obvious since $H_\ell=G_\ell$; if this is not the case, then $H_\ell$ satisfies the hypotheses of corollary \ref{cor:Possibilities}. Note that the image of $\abGal{K_\ell}$ in $\operatorname{Aut} T_\ell(E)$ is exactly $H_\ell$ by construction.
We wish to apply corollary \ref{cor:Possibilities} to $G=H_\ell$, assuming that $\ell^{n-v}$ does not divide $D(\ell)$.

Since $\ell^{k(L)} \bigm\vert b_0(K_\ell,E)$ by corollary \ref{cor:ValueOfk}, we deduce $\ell^{n-k(L)-v} \nmid b_0(K_\ell,E)^4b_0(K_\ell,E \times E)$, and a fortiori $\ell^{n-k(L)+1-2v} \nmid b_0(K_\ell,E)^4b_0(K_\ell,E \times E)$. Lemma \ref{lemma_ellTooBigThenNotBorel} then implies that $G(\ell^{n-k(L)+1-2v})$ cannot be put in triangular form, and on the other hand $\ell^{n-v} \nmid b_0(K_\ell,E)^5b_0(K_\ell,E \times E)$ implies that $\ell^{2n}$ does not divide $b_0(K_\ell,E)^4b_0(K_\ell,E \times E)$, so $G(\ell^{n})$ is not abelian (thanks to proposition \ref{prop_ellTooBigThenNotAbelian}). It then follows from corollary \ref{cor:Possibilities} that $G'=H_\ell'$ contains the principal congruence subgroup $\mathcal{B}_\ell(16n-4)$ (resp. $\mathcal{B}_\ell(48n-10)$ for $\ell=2$).
\end{proof}

\begin{corollary}\label{cor:IndexInSL2}
Notation as above. The index $[\operatorname{SL}_2(\mathbb{Z}_\ell) : (H_\ell' \cap \mathcal{B}_\ell(1))]$ is of the form $|\operatorname{SL}_2(\mathbb{F}_\ell)| B(\ell)$, where for $\ell \neq 2$ the number $B(\ell)$ is a power of $\ell$ dividing $\ell^{33} \cdot D(\ell)^{48}$ (resp. $B(2)$ is a power of 2 dividing $2^{255}D(2)^{144}$).
\end{corollary}

\begin{proof}
We can write the index $[\operatorname{SL}_2(\mathbb{Z}_\ell) : (H_\ell' \cap \mathcal{B}_\ell(1))]$ as
\[
[\operatorname{SL}_2(\mathbb{Z}_\ell) : \mathcal{B}_\ell(1)] \cdot [\mathcal{B}_\ell(1):(H_\ell' \cap \mathcal{B}_\ell(1))] = |\operatorname{SL}_2(\mathbb{F}_\ell)| \cdot [\mathcal{B}_\ell(1):(H_\ell' \cap \mathcal{B}_\ell(1))],
\]
so we just need to prove that $B(\ell)=[\mathcal{B}_\ell(1):(H_\ell' \cap \mathcal{B}_\ell(1))]$ divides $\ell^{33} D(\ell)^{48}$ (and the analogous statement for $\ell=2$). Notice that since $\mathcal{B}_\ell(1)$ is a pro-$\ell$ group the number $B(\ell)$ is a power of $\ell$.

Choose $n$ such that $\ell^{n-v} \bigm | \bigm| D(\ell)$: then $\ell^{n+1-v} \nmid D(\ell)$, and therefore the above theorem implies that $H_\ell'$ contains $\mathcal{B}_\ell(16(n+1)-4) \subseteq \mathcal{B}_\ell(1)$ (resp.~$\mathcal{B}_2(48(n+1)-10)$ for $\ell=2$): the index of $\mathcal{B}_\ell(16(n+1)-4)$ in $\mathcal{B}_\ell(1)$ is $\ell^{3(16(n+1)-5)}$, so we get
\[
[\mathcal{B}_\ell(1):(H_\ell' \cap \mathcal{B}_\ell(1))] \bigm\vert \ell^{48n+33} \bigm\vert \ell^{33} \cdot D(\ell)^{48}
\]
for $\ell \neq 2$, and likewise we have
\[
[\mathcal{B}_2(1):(H_2' \cap \mathcal{B}_2(1))] \bigm\vert 2^{3(48(n-1)+85)} \bigm\vert 2^{255} D(2)^{144}
\]
for $\ell=2$.
\end{proof}

\section{The determinant and the large primes}\label{sec:LargePrimes}
We now turn to studying the determinant of the adelic representation and the behaviour at the very large primes.

\begin{proposition}\label{prop_AdelicDeterminant}
The index
\[
\left[\widehat{\mathbb{Z}}^{\times} : \prod_{\ell} \det \rho_\ell(\abGal{K}) \right]
\]
is bounded by $[K:\mathbb{Q}]$.
\end{proposition}

\begin{proof}
The Weil pairing induces an identification of the determinant
$
\abGal{K} \xrightarrow{\rho_\ell} G_\ell \xrightarrow{\det} \mathbb{Z}_\ell^{\times}
$
with $\abGal{K} \stackrel{\chi_\ell}{\rightarrow} \mathbb{Z}_\ell^{\times}$, where $\chi_\ell$ denotes the $\ell$-adic cyclotomic character; by Galois theory we have
\[
\prod_{\ell} \det \rho_\ell \left(\abGal{K}\right) = \prod_{\ell} \chi_\ell \left(\abGal{K}\right) \cong \operatorname{Gal}\left( K\left(\mu_{\infty}\right)/K \right).
\]

Let $F=K \cap \mathbb{Q}\left(\mu_{\infty} \right)$: it is a finite Galois extension of $\mathbb{Q}$. As $\mathbb{Q}\left(\mu_{\infty}\right)$ is Galois over $\mathbb{Q}$, the restriction map $\operatorname{Gal}\left( K\left(\mu_{\infty}\right)/K \right) \to \operatorname{Gal}\left( \mathbb{Q}\left(\mu_{\infty}\right)/F \right)$ is well-defined and induces an isomorphism. Therefore
\[
\begin{aligned}
\left[\widehat{\mathbb{Z}}^{\times} : \prod_{\ell}\chi_\ell(\abGal{K}) \right] & = [\operatorname{Gal}\left(\mathbb{Q}\left(\mu_{\infty}\right) / \mathbb{Q} \right) : \operatorname{Gal}\left( \mathbb{Q}\left(\mu_{\infty}\right)/F\right)] \\
&  =[F:\mathbb{Q}] \leq [K:\mathbb{Q}]
\end{aligned}
\]
as claimed.
\end{proof}

We will also need a surjectivity result (on $\operatorname{SL}_2$) modulo $\ell$ for every $\ell$ sufficiently large: as previously mentioned, these are essentially the ideas of \cite{MR1209248} and \cite{MR1619802}, in turn inspired by those of Serre.

\begin{lemma}\label{lemma:SL2Fl}
If $\ell \nmid b_0(K,E \times E;2) b_0(K,E;60)$ then the group $G_\ell(\ell)$ contains $\operatorname{SL}_2(\mathbb{F}_\ell)$.
\end{lemma}

\begin{proof}
Let $\ell$ be a prime for which $G_\ell(\ell)$ does not contain $\operatorname{SL}_2(\mathbb{F}_\ell)$ and let, for the sake of clarity, $G=G_\ell(\ell)$. By theorem \ref{thm:Dickson}, if $G$ does not contain $\operatorname{SL}_2(\mathbb{F}_\ell)$, then the following are the only possibilities:

\begin{enumerate}[label=(\Roman*)]
\item $G$ is contained in a Borel subgroup of $\operatorname{GL}_2(\mathbb{F}_\ell)$: by definition, such a subgroup fixes a line, therefore $\ell \bigm | b_0(K,E)$ by lemma \ref{lemma_ellTooBigThenNotBorel}.
\item $G$ is contained in the normalizer of a Cartan subgroup of $\operatorname{GL}_2(\mathbb{F}_\ell)$: let $\mathcal{C}$ be this Cartan subgroup and $N$ its normalizer. By Dickson's classification $\mathcal{C}$ has index 2 in $N$, so the morphism $\displaystyle \abGal{K} \to G \to \frac{G}{G \cap \mathcal{C}} \hookrightarrow \frac{N}{\mathcal{C}}$ induces a quadratic character of $\abGal{K}$, whose kernel corresponds to a certain field $K'$ satisfying $[K':K] \leq |N/\mathcal{C}|=2$. By construction, the image of $\abGal{K'}$ in $\operatorname{Aut}\left(E[\ell]\right)$ is contained in $\mathcal{C}$, so applying proposition \ref{prop_ellTooBigThenNotAbelian} to $E_{K'}$ we get
\[
\ell \bigm\vert b_0(K',E) b_0(K',E \times E) \bigm\vert b_0(K,E;2) b_0(K,E \times E;2).
\]
Notice that this also covers the case of $G$ being contained in a Cartan subgroup.

\item The projectivization $\mathbb{P}G$ of $G$ is a finite group of order at most 60: we essentially copy the previous argument. Let $H=\mathbb{P}G$; then we have a morphism
\[
\abGal{K} \to G \to \frac{\mathbb{F}_\ell^{\times} G}{\mathbb{F}_\ell^{\times}} = H
\]
whose kernel defines an extension $K''$ of $K$ with $[K'':K]=|H| \leq 60$ and such that the image of the representation of $\operatorname{Gal}\left(\overline{K''}/K''\right)$ on $E[\ell]$ is contained in $\mathbb{F}_\ell^{\times}$: lemma \ref{lemma_ellTooBigThenNotBorel} then yields $\ell \bigm\vert b_0(K'',E) \bigm\vert b_0(K,E;60)$.
\end{enumerate}

It is then apparent that the lemma is true with the condition
\[
\ell \nmid b_0(K,E) b_0(K,E \times E) b_0(K,E;2) b(K,E \times E;2) b_0(K,E;60);
\]
however, since
\[
b_0(K,E) \bigm\vert b_0(K,E;2) \bigm\vert b_0(K,E;60), \quad b_0(K,E \times E) \bigm\vert b_0(K,E \times E;2),
\]
and since $\ell$ is prime, we see that $\ell$ divides
\[
b_0(K,E) b_0(K,E \times E) b_0(K,E;2) b_0(K,E \times E;2) b_0(K,E;60)
\]
if and only if it divides $b(K,E \times E;2) b_0(K,E;60)$, which finishes the proof.
\end{proof}

\smallskip

\begin{corollary}
Let $\Psi=30 \cdot b_0(K,E \times E;2) b_0(K,E;60)$. If $\ell \nmid \Psi$, then $G_\ell'$ is all of $\operatorname{SL}_2(\mathbb{Z}_\ell)$.
\end{corollary}
\begin{proof}
The previous lemma implies that $G_\ell(\ell)$ contains $\operatorname{SL}_2(\mathbb{F}_\ell)$, and by hypothesis $\ell$ is strictly larger than 3, so the corollary follows from lemma \ref{lemma_SerreLift}.
\end{proof}

\section{The adelic index and some consequences}\label{sec:Finale}
We have thus acquired a good understanding of the $\ell$-adic representation for every prime $\ell$, and we are now left with the task of bounding the overall index of the full adelic representation. The statement we are aiming for is:
\begin{theorem}\label{thm:Final}
Let $E/K$ be an elliptic curve without complex multiplication with stable Faltings height $h(E)$. Let
$
\rho_\infty:\abGal{K} \to \operatorname{GL}_2\big( \widehat{\mathbb{Z}} \big)
$
be the adelic Galois representation associated with $E$, and set
\[
\Psi=2 \cdot 3 \cdot 5 \cdot b_0(K,E \times E;2) b_0(K,E;60), \quad D(\infty)=b_0(K,E;24)^5b_0(K,E\times E;24);
\]
let moreover $K_2$ be as in section \ref{sec:GaloisGroups} and
\[
D(2)= b_0(K_2,E)^5 b_0(K_2,E \times E).
\]
With this notation we have
\[
\big[ \operatorname{GL}_2\big( \widehat{\mathbb{Z}} \big) : \rho_\infty \abGal{K} \big] \leq [K:\mathbb{Q}] \cdot 2^{222} \cdot D(2)^{144} \cdot \operatorname{rad}(\Psi)^{36} \cdot D(\infty)^{48},
\]
where $\displaystyle \operatorname{rad}(\Psi) = \prod_{\ell \mid \Psi} \ell$ is the product of the primes dividing $\Psi$.
\end{theorem}

The strategy of proof, which essentially goes back to Serre, is to pass to a suitable extension of $K$ over which the adelic representation decomposes as a direct product and then use the previous bounds. For this we will need some preliminaries.
If $L$ is any number field, we let $L_{cyc}=L\left(\mu_\infty\right)$ be its maximal cyclotomic extension. From the exact sequence
\[
1 \to \frac{\operatorname{SL}_2(\widehat{\mathbb{Z}})}{\operatorname{Gal}\left( \overline{K} / K_{cyc} \right)} \to \frac{\operatorname{GL}_2(\widehat{\mathbb{Z}})}{\rho_\infty\left(\abGal{K}\right)} \to \frac{\widehat{\mathbb{Z}}^{\times}}{\det \circ \rho_\infty \left(\abGal{K}\right)} \to 1
\]
we see that $[\operatorname{GL}_2(\widehat{\mathbb{Z}}) : \rho_\infty\left(\abGal{K}\right)]$ equals
\[
[\widehat{\mathbb{Z}}^{\times}:\det \circ \rho_\infty \left(\abGal{K}\right)] \cdot [\operatorname{SL}_2(\widehat{\mathbb{Z}}) : \rho_\infty\left( \operatorname{Gal}\left( \overline{K} / K_{cyc} \right)\right)],
\]
where the first term is bounded by $[K:\mathbb{Q}]$ thanks to proposition \ref{prop_AdelicDeterminant}. It thus remains to understand the term $[\operatorname{SL}_2(\widehat{\mathbb{Z}}) : \rho_\infty\left( \operatorname{Gal}\left( \overline{K} / K_{cyc} \right)\right)]$.
Let $\mathcal{P}$ be the (finite) set consisting of $2, 3, 5$, and the prime numbers $\ell$ for which $G_\ell$ does not contain $\operatorname{SL}_2(\mathbb{Z}_\ell)$, and let $F$ be the field generated over $K$ by $\displaystyle \bigcup_{\ell \in \mathcal{P}} E[\ell]$. It is clear that
\[
[\operatorname{SL}_2(\widehat{\mathbb{Z}}) : \rho_\infty\left( \operatorname{Gal}\left( \overline{K} / K_{cyc} \right)\right)] \leq [\operatorname{SL}_2(\widehat{\mathbb{Z}}) : \rho_\infty\left( \operatorname{Gal}\left( \overline{K} / F_{cyc} \right)\right)].
\]

\smallskip

\noindent\textbf{Notation.} We set $S=\rho_\infty \left(\operatorname{Gal} \left(\overline{K} / F_{cyc} \right)\right) \subseteq \operatorname{SL}_2 (\widehat{\mathbb{Z}})=\prod_{\ell} \operatorname{SL}_2(\mathbb{Z}_\ell)$ and let $S_\ell$ be the projection of $S$ on $\operatorname{SL}_2(\mathbb{Z}_\ell)$.

\medskip

The core of the argument is contained in the following proposition.

\begin{proposition}\label{prop_Boulot}
Let $B(\ell)$ be as in corollary \ref{cor:IndexInSL2} and $D(2)$ be as in the statement of theorem \ref{thm:Final}. The following hold:
\begin{enumerate}
\item $S=\prod_{\ell} S_\ell$.
\item For $\ell \in \mathcal{P}$, $\ell \neq 2$, we have
\[
\big[ \operatorname{SL}_2(\mathbb{Z}_\ell) : S_\ell \big] \bigm\vert \left|\operatorname{SL}_2(\mathbb{F}_\ell) \right| \cdot B(\ell);
\]
for $\ell=2$ we have
\[
\big[ \operatorname{SL}_2(\mathbb{Z}_2) : S_2 \big] < 2^{258} D(2)^{144}.
\]
\item For $\ell \notin \mathcal{P}$ the equality $S_\ell=\operatorname{SL}_2\left(\mathbb{Z}_\ell\right)$ holds.
\end{enumerate}

\end{proposition}

\begin{proof}

\noindent (1) This would follow from \cite[Théorème 1]{MR3093502}, but since we do not need the added generality and the proof is quite short we include it here for the reader's convenience.

Regard $S$ as a closed subgroup of $\prod_\ell S_\ell \subseteq \prod_\ell \operatorname{SL}_2(\mathbb{Z}_\ell) = \operatorname{SL}_2(\widehat{\mathbb{Z}})$. For each finite set of primes $B$, let $p_B\colon S\to S_B=\prod_{\ell \in B} S_\ell$ be the canonical projection. We plan to show that for every such $B$ containing $\mathcal{P}$ we have $p_B(S)=S_B$. Indeed let us consider the case $B=\mathcal{P}$ first. Our choice of $F$ implies that $S_{\ell}=\rho_\ell(\abGal{F})$ is a pro-$\ell$ group for every $\ell \in \mathcal{P}$: the group $S_{\ell}$ has trivial reduction modulo $\ell$ by construction, and therefore $S_\ell$ admits the usual congruence filtration by the kernels of the reductions modulo $\ell^k$ for varying $k$. Now a pro-$\ell$ group is obviously pro-nilpotent, so $p_B(S)$ is pro-nilpotent as well and therefore it is the product of its pro-Sylow subgroups (which are just the $S_\ell$).
To treat the general case we recall some terminology from \cite{SerreAbelianRepr}. Following Serre, we say that a finite simple group $\Sigma$ \textit{occurs} in the profinite group $Y$ if there exist a closed subgroup $Y_1$ of $Y$ and an open normal subgroup $Y_2$ of $Y_1$ such that $\Sigma \cong Y_1/Y_2$. We also write $\operatorname{Occ}(Y)$ for the set of isomorphism classes of finite simple non abelian groups occurring in $Y$.
From \cite[IV-25]{SerreAbelianRepr} we read the following description of the sets $\operatorname{Occ}(\operatorname{GL}_2(\mathbb{Z}_p))$:
\begin{itemize}
\item $\operatorname{Occ}(\operatorname{GL}_2(\mathbb{Z}_p))=\emptyset$ for $p=2,3$;
\item $\operatorname{Occ}(\operatorname{GL}_2(\mathbb{Z}_5))=\left\{ A_5\right\}$;
\item $\operatorname{Occ}(\operatorname{GL}_2(\mathbb{Z}_p))=\left\{\operatorname{PSL}_2(\mathbb{F}_p),A_5 \right\}$ for $p \equiv \pm 1 \pmod 5$, $p>5$;
\item $\operatorname{Occ}(\operatorname{GL}_2(\mathbb{Z}_p))=\left\{\operatorname{PSL}_2(\mathbb{F}_p) \right\}$ for $p \equiv \pm 2 \pmod 5$, $p>5$.
\end{itemize}

Let $B$ be a finite set of primes containing $\mathcal{P}$ and satisfying $p_B(S)=S_B$, and fix a prime $\ell_0 \notin B$. We claim that $p_{B \cup \{\ell_0\}}(S)=S_{B \cup \{\ell_0\}}$. Notice first that $\operatorname{PSL}_2(\mathbb{F}_{\ell_0})$ occurs in $S_{\ell_0}$ and therefore in $p_{B \cup \{\ell_0\}}(S)$; set $N_{\ell_0}= \ker \left( p_{B \cup \{\ell_0\}}(S) \to p_{B}(S) \right)$. From the exact sequence
\begin{equation}\label{eq_ExactSequence}
1 \to N_{\ell_0} \to p_{B \cup \{\ell_0\}}(S) \to p_{B}(S) \to 1
\end{equation}
we see that $\text{Occ} \big(p_{B \cup \{\ell_0\}}(S)\big) = \text{Occ} \big(p_{B}(S)\big) \cup \text{Occ} \big(N_{\ell_0}\big)$. On the other hand, the only finite non-abelian simple groups that can occur in $p_{B}(S)$ are $A_5$ and groups of the form $\operatorname{PSL}_2(\mathbb{F}_\ell)$ for $\ell \neq \ell_0$, so $\operatorname{PSL}_2(\mathbb{F}_{\ell_0})$ does not occur in $p_{B}(S)$ (notice that $ \operatorname{PSL}_2(\mathbb{F}_{\ell_0}) \not \cong A_5$ since $\ell_0 \neq 5$), and therefore it must occur in $N_{\ell_0}$. Denote by $\overline{N_{\ell_0}}$ the image of $N_{\ell_0}$ in $\operatorname{SL}_2(\mathbb{F}_{\ell_0})$. The kernel of $N_{\ell_0} \to \operatorname{SL}_2(\mathbb{F}_{\ell_0})$ is a pro-$\ell_0$ group, so $\text{Occ}\big( N_{\ell_0}\big)$ equals $\text{Occ} \big(\overline{N_{\ell_0}}\big)$ and therefore $\overline{N_{\ell_0}}$ projects surjectively onto $\operatorname{PSL}_2(\mathbb{F}_{\ell_0})$. Hence we have $\overline{N_{\ell_0}}=\operatorname{SL}_2(\mathbb{F}_{\ell_0})$ by \cite[IV-23, Lemma 2]{SerreAbelianRepr}, and by lemma \ref{lemma_SerreLift} this implies $N_{\ell_0}=\operatorname{SL}_2(\mathbb{Z}_{\ell_0})$: by (\ref{eq_ExactSequence}) we then have $p_{B \cup \{\ell_0\}}(S) = p_B(S) \times \operatorname{SL}_2(\mathbb{Z}_{\ell_0})$ as claimed. By induction, the equality $p_B(S)=S_B$ holds for any finite set of primes $B$ containing $\mathcal{P}$, and since $S$ is profinite we deduce that $S=\prod_{\ell} S_\ell$.

\medskip

\noindent (2) The group $S_\ell$ is the kernel of the projection map $\left(G_\ell \cap \operatorname{SL}_2\left(\mathbb{Z}_\ell\right) \right) \to \operatorname{SL}_2(\mathbb{F}_\ell)$; as such, it contains the intersection $H_\ell' \cap B_\ell(1)$ (notation as in section \ref{sec:GaloisGroups}), so we just need to invoke corollary \ref{cor:IndexInSL2} to have
\[
\big[ \operatorname{SL}_2(\mathbb{Z}_\ell) : S_\ell \big] \bigm\vert \big[ \operatorname{SL}_2(\mathbb{Z}_\ell) : (H_\ell' \cap B_\ell(1)) \big] \bigm\vert |\operatorname{SL}_2(\mathbb{F}_\ell)| B(\ell)
\]
as claimed.
On the other hand, for $\ell=2$ the group $H_2$ is a subgroup of $\rho_2(\operatorname{Gal}\left(\overline{K}/K(E[4])\right))$, while $S_2$ is $\rho_2(\operatorname{Gal}\left(\overline{K}/K_{cyc}(E[2])\right))$, so $S_2$ is larger than $H_2' \cap \mathcal{B}_2(1)$ and we can again use the bound of corollary \ref{cor:IndexInSL2}, which now reads
\[
\big[ \operatorname{SL}_2(\mathbb{Z}_2) : S_2 \big] \leq 2^{255} D(2)^{144} |\operatorname{SL}_2(\mathbb{F}_2)| < 2^{258} D(2)^{144}.
\]

\medskip

\noindent (3) As $\ell \not \in \mathcal{P}$ we know that $\rho_{\ell}(\abGal{K})$ contains $\operatorname{SL}_2(\mathbb{Z}_\ell)$, so the group $\operatorname{PSL}_2(\mathbb{F}_\ell)$ occurs in $\rho_{\ell}(\abGal{K})$. Consider the Galois group $\operatorname{Gal}(F/K)$: it is by construction a subquotient of $\prod_{p \in \mathcal{P}} \operatorname{GL}_2\left(\mathbb{Z}_p\right)$, so the only groups that can occur in it are those in $\bigcup_{p \in \mathcal{P}} \text{Occ} \left(\operatorname{GL}_2\left(\mathbb{Z}_p\right)\right)$, and in particular $\operatorname{PSL}_2(\mathbb{F}_\ell)$ does not occur in $\operatorname{Gal}(F/K)$. Now $\rho_\ell(\abGal{K})$ is an extension of a quotient of $\operatorname{Gal}(F/K)$ by $\rho_\ell\left(\operatorname{Gal}\left(\overline{K}/F \right)\right)$, so $\operatorname{PSL}_2(\mathbb{F}_\ell)$ occurs in $\rho_\ell\left(\operatorname{Gal}\left(\overline{K}/F \right)\right)$, and furthermore $\rho_\ell\left(\operatorname{Gal}\left(\overline{K}/F \right)\right)$ is an extension of an abelian group by $\rho_\ell\left(\operatorname{Gal}\left(\overline{K}/F_{cyc} \right)\right)$, so $\operatorname{PSL}_2(\mathbb{F}_\ell)$ also occurs in $\rho_\ell\left(\operatorname{Gal}\left(\overline{K}/F_{cyc} \right)\right)=S_\ell$: reasoning as in (i), we then see that $S_\ell$ projects surjectively onto $\operatorname{PSL}_2(\mathbb{F}_\ell)$, and therefore $S_\ell=\operatorname{SL}_2(\mathbb{Z}_\ell)$.\end{proof}

The proof of theorem \ref{thm:Final} is now immediate: 

\begin{proof}[Proof of theorem \ref{thm:Final}]
We have already seen that the index $\left[\operatorname{GL}_2(\widehat{\mathbb{Z}}) : \rho_\infty\left(\abGal{K}\right)\right]$ equals $[\mathbb{Z}^{\times}:\det \circ \rho_\infty \abGal{K}] \cdot [\operatorname{SL}_2(\widehat{\mathbb{Z}}) : \rho_\infty\left( \operatorname{Gal}\left( \overline{K} / K_{cyc} \right)\right)]$.
Now the first factor in this product is at most $[K:\mathbb{Q}]$, while the second is bounded by $[\operatorname{SL}_2(\widehat{\mathbb{Z}}) : S]$; it follows that the adelic index is bounded by 
\begin{equation}\label{eq:AdelicBoundIntermediate}
\begin{aligned}
\left[K:\mathbb{Q}\right] \cdot [\operatorname{SL}_2(\widehat{\mathbb{Z}}) : S] & \leq [K:\mathbb{Q}] \cdot \prod_{\ell \in \mathcal{P}}  [\operatorname{SL}_2(\mathbb{Z}_\ell): S_\ell] \\
																														 & \leq [K:\mathbb{Q}] \cdot \prod_{\ell | \Psi}  [\operatorname{SL}_2(\mathbb{Z}_\ell): S_\ell] \\
																														 & < [K:\mathbb{Q}] \cdot 2^{258} \cdot D(2)^{144} \cdot \prod_{\ell | \Psi, \ell \neq 2} |\operatorname{SL}_2(\mathbb{F}_\ell)| \cdot \prod_{\ell | \Psi, \ell \neq 2} B(\ell),
\end{aligned}
\end{equation}
where we have used the fact that $\ell \nmid \Psi \Rightarrow \ell \notin \mathcal{P}$. We now observe that by construction for all odd primes $\ell$ we have $v_\ell(D(\infty)) \geq v_\ell(D(\ell))$, so by corollary \ref{cor:IndexInSL2} the quantity $\prod_{\ell | \Psi, \ell \neq 2} B(\ell)$ divides
\[
\begin{aligned}
\prod_{\ell | \Psi, \ell \neq 2} \ell^{33} \ell^{48 v_\ell(D(\ell))}
														& \bigm\vert  \prod_{\ell | \Psi, \ell \neq 2} \ell^{33} \ell^{48 v_\ell(D(\infty))},
\end{aligned}
\]
which in turn divides $\left(\frac{\operatorname{rad}(\Psi)}{2}\right)^{33} \cdot D(\infty)^{48}$. 
Combining this fact with equation \eqref{eq:AdelicBoundIntermediate} and the trivial bound $|\operatorname{SL}_2(\mathbb{F}_\ell)|<\ell^3$ we find that the adelic index is at most
\[
[K:\mathbb{Q}] \cdot 2^{225} \cdot D(2)^{144} \cdot \left(\prod_{\ell | \Psi, \ell \neq 2} \ell^{3}\right) \cdot \operatorname{rad}(\Psi)^{33} \cdot D(\infty)^{48},
\]
which in turn is less than $[K:\mathbb{Q}] \cdot 2^{222} \cdot D(2)^{144} \cdot \operatorname{rad}(\Psi)^{36} \cdot D(\infty)^{48}$, whence the theorem.
\end{proof}

\medskip

Using the estimates of proposition \ref{prop_b0} to bound $\Psi, D(2)$ and $D(\infty)$ we get:

\begin{corollary}{(Theorem \ref{thm:OverK})}
Let $E/K$ be an elliptic curve that does not admit complex multiplication. The inequality
\[
\left[\operatorname{GL}_2(\widehat{\mathbb{Z}}) : \rho_\infty\left(\abGal{K}\right)\right] < \displaystyle \gamma_1 \cdot [K:\mathbb{Q}]^{\gamma_2} \cdot \max\left\{1,h(E),\log [K:\mathbb{Q}]\right\}^{2\gamma_2}
\]
holds, where $\gamma_1=\exp(10^{21483})$ and $\gamma_2=2.4 \cdot 10^{10}$.
\end{corollary}

\begin{remark}\label{rmk:ImprovedFinale}
With some work, the techniques used in \cite{2012arXiv1212.4713L} (cf.~especially Theorem 4.2 of \textit{op.~cit.}) could be used to improve the above bound on $\Psi$; unfortunately, the same methods do not seem to be easily applicable to bound $D(\infty)$. Notice that our estimates for $\Psi$ and $D(\infty)$ are essentially of the same order of magnitude, so using a finer bound for $\Psi$ without changing the one for $D(\infty)$ would only yield a minor improvement of the final result. 

On the other hand, it is easy to see that using the improved version of the isogeny theorem mentioned in remarks \ref{rmk:IsogenyImprovedVersion1} and \ref{rmk:IsogenyImprovedVersion2} one can prove
\[
\left[\operatorname{GL}_2(\widehat{\mathbb{Z}}) : \rho_\infty\left(\abGal{K}\right)\right] < \displaystyle \gamma_3 \cdot \left([K:\mathbb{Q}]\cdot \max\left\{1,h(E),\log [K:\mathbb{Q}]\right\}\right)^{\gamma_4}
\]
with $\gamma_3=\exp\left(1.9 \cdot 10^{10} \right)$ and $\gamma_4=12395$. 
\end{remark}

\subsection{The field generated by a torsion point}
As an easy consequence of our main result we can also prove:
{
\renewcommand{\thetheorem}{\ref{cor:FieldGeneratedByTorsion}}
\begin{corollary}
Let $E/K$ be an elliptic curve that does not admit complex multiplication. There exists a constant $\const(E/K)$ with the following property: for every $x \in E_{\operatorname{tors}}(\overline{K})$ (of order denoted $N(x)$) the inequality
\[
[K(x):K] \geq \const(E/K) N(x)^2
\]
holds. We can take $\const(E/K)=\left(\zeta(2) \cdot \big[ \operatorname{GL}_2\big( \widehat{\mathbb{Z}} \big) : \rho_\infty \abGal{K} \big] \right)^{-1}$, which can be explicitly bounded thanks to the main theorem.
\end{corollary}
\addtocounter{theorem}{-1}
}

\begin{proof}
For any such $x$ set $N=N(x)$ and choose a point $y \in E[N]$ such that $(x,y)$ is a basis of $E[N]$ as $\left(\mathbb{Z}/N\mathbb{Z}\right)$-module. Let $G(N)$ be the image of $\abGal{K}$ inside $\operatorname{Aut} E[N]$, which we identify to $\operatorname{GL}_2(\mathbb{Z}/N\mathbb{Z})$ via the basis $(x,y)$. We have a tower of extensions $K(E[N])/K(x)/K$, where $K(E[N])$ is Galois over $K$ and therefore over $K(x)$. The Galois groups of these extensions are given -- essentially by definition -- by
\[
\operatorname{Gal}(K(E[N])/K) = G(N), \quad \operatorname{Gal}(K(E[N])/K(x)) = \operatorname{Stab}(x),
\]
where $\operatorname{Stab}(x)=\left\{ \sigma \in G(N) \bigm\vert \sigma(x)=x \right\}$. It follows that
\[
[K(x):K]=\displaystyle \frac{[K(E[N]):K]}{[K(E[N]):K(x)]} = \frac{\left|G(N)\right|}{\left|\operatorname{Stab}(x)\right|},
\]
and furthermore it is easy to check that
\[
\begin{aligned}
\left|G(N)\right| = \displaystyle \frac{\left| \operatorname{GL}_2(\mathbb{Z}/N\mathbb{Z}) \right|}{[\operatorname{GL}_2(\mathbb{Z}/N\mathbb{Z}):G(N)]} & = \frac{\displaystyle N^3 \varphi(N) \prod_{p|N}\left( 1-\frac{1}{p^2} \right) }{[\operatorname{GL}_2(\mathbb{Z}/N\mathbb{Z}):G(N)]}.
\end{aligned}
\]

On the other hand, the stabilizer of $x$ in $G(N)$ is contained in the stabilizer of $x$ in $\operatorname{GL}_2(\mathbb{Z}/N\mathbb{Z})$, which is simply
\[
\left\{ \left( \begin{matrix} 1 & a \\ 0 & b \end{matrix} \right) \bigm\vert a \in \mathbb{Z}/N\mathbb{Z}, \; b \in \left(\mathbb{Z}/N\mathbb{Z}\right)^\times\right\},
\]
so $\left|\operatorname{Stab}(x)\right| \leq \left| \mathbb{Z}/N\mathbb{Z} \right| \cdot   \left| \left(\mathbb{Z}/N\mathbb{Z}\right)^\times \right|= N \varphi(N)$. Finally, the index of $G(N)$ inside $\operatorname{GL}_2(\mathbb{Z}/N\mathbb{Z})$ is certainly not larger than the index of $G_\infty$ inside $\operatorname{GL}_2(\widehat{\mathbb{Z}})$. Putting everything together we obtain
\[
[K(x):K]=\frac{\displaystyle N^3 \varphi(N) \prod_{p|N}\left( 1-\frac{1}{p^2} \right) }{[\operatorname{GL}_2(\mathbb{Z}/N\mathbb{Z}):G(N)] \cdot \left|\operatorname{Stab}(x)\right|} \geq \frac{\displaystyle N^3 \varphi(N) \prod_{p \text{ prime}}\left( 1-\frac{1}{p^2} \right) }{N\varphi(N) \cdot [\operatorname{GL}_2(\widehat{\mathbb{Z}}):G_\infty] },
\]
and the corollary follows by remarking that $\displaystyle \prod_{p \text{ prime}}\left( 1-\frac{1}{p^2} \right)= \frac{1}{\zeta(2)}$.
\end{proof}

\medskip

\noindent\textbf{Acknowledgments.} It is a pleasure to thank my advisor, N. Ratazzi, for suggesting the problem, for his unfailing support, and for the many helpful discussions. I am grateful to the anonymous referee for the numerous helpful suggestions. I would also like to thank G. Rémond and E. Gaudron for their many valuable comments on a preliminary version of this text, and J-P. Serre for pointing out a problem in a later version. The author gratefully acknowledges financial support from the Fondation Mathématique Jacques Hadamard (grant ANR-10-CAMP-0151-02 in the “Programme des Investissements d’Avenir”).

\bibliography{Biblio}{}
\bibliographystyle{alpha}

\end{document}